\documentclass[10pt,reqno]{amsart}
\usepackage[margin=1in]{geometry}
\usepackage{dsfont}
\usepackage{amsmath,amssymb,amsthm,graphicx,amsxtra, setspace}
\usepackage[utf8]{inputenc}
\usepackage{mathrsfs}
\usepackage{alltt}
\usepackage{relsize}
\usepackage{hyperref}
\usepackage{aliascnt}
\usepackage{tikz}
\usepackage{mathtools}
\usepackage{multicol}
\usepackage{upgreek}
\usepackage{multirow}
\usepackage{subcaption}

\usepackage{graphicx,type1cm,eso-pic,color}
\allowdisplaybreaks

\newtheorem{theorem}{Theorem}[section]

\newaliascnt{lemma}{theorem}
\newtheorem{lemma}[lemma]{Lemma}
\aliascntresetthe{lemma} 

\newaliascnt{proposition}{theorem}

\aliascntresetthe{proposition}

\newaliascnt{assumption}{theorem}

\aliascntresetthe{assumption}

\newaliascnt{corollary}{theorem}
\newtheorem{corollary}[corollary]{Corollary}
\aliascntresetthe{corollary}

\newaliascnt{definition}{theorem}

\aliascntresetthe{definition}

\newaliascnt{example}{theorem}

\aliascntresetthe{example}

\newaliascnt{remark}{theorem}
\newtheorem{remark}[remark]{Remark}
\aliascntresetthe{remark}

\newaliascnt{hypothesis}{theorem}

\aliascntresetthe{hypothesis}

\newaliascnt{property}{theorem}

\aliascntresetthe{property}

\let\originalleft\left
\let\originalright\right
\renewcommand{\left}{\mathopen{}\mathclose\bgroup\originalleft}
\renewcommand{\right}{\aftergroup\egroup\originalright}

\def\L{\mathrm{L}}

\def\C{\mathrm{C}}

\def\J{\mathrm{J}}

\def\no{\nonumber}

\def\P{\mathbb{P}}

\def\H{\mathrm{H}}

\def\red{\textcolor{black}}

\mathtoolsset{showonlyrefs} 
\noeqref{5.16,5.17,5.18, 5.35, 5.36, 5.51, 5.52,5.53}
\newcommand{\Addresses}{{
		\footnote{
			\noindent \textsuperscript{1}Department of Mathematics, Indian Institute of Technology Roorkee (IITR), Roorkee, Uttarakhand 247667, INDIA.\par\nopagebreak
			\noindent  \textit{e-mails:} Arbaz Khan:  \texttt{arbaz@ma.iitr.ac.in,} Sumit Mahajan: \texttt{sumit{\_}m@ma.iitr.ac.in} .
			
			\noindent \textsuperscript{*}Corresponding author.\\
			\textit{Keywords: Burgers-Huxley equation, $hp$-FEM, weakly singular kernel, A priori error analysis, DG time stepping.} 
			
			Mathematics Subject Classification (2020):  45L10, 65N12, 65N15, 65N30.
}}}
\begin{document}
	\title[\lowercase{$hp$}-DGFEM for GBHE with memory]{\lowercase{$hp$}-Discontinuous Galerkin method for the Generalized Burgers-Huxley equation with weakly singular kernels\Addresses}
	\author[S. Mahajan and A. Khan ]{Sumit Mahajan\textsuperscript{1*} and Arbaz Khan\textsuperscript{1}}
	
	\begin{abstract}
In this work, we investigate the $hp$-discontinuous Galerkin (DG) time-stepping method for the generalized Burgers-Huxley equation with memory, a non-linear advection-diffusion-reaction problem featuring weakly singular kernels. We derive a priori error estimates for the semi-discrete scheme using $hp$-DG time-stepping, with explicit dependence on the local mesh size, polynomial degree, and solution regularity, achieving optimal convergence in the energy norm. For the fully-discrete scheme, we initially implement the $hp$-finite element method (conforming), followed by the $hp$-discontinuous Galerkin method. We establish the well-posedness and stability of the fully-discrete scheme and provide corresponding a priori estimates. The effectiveness of the proposed method is demonstrated through numerical validation on a series of test problems.
	\end{abstract}
	\maketitle
\section{Introduction}\setcounter{equation}{0} 
In this work we will study the discretization in both time and space of the \emph{generalized Burgers-Huxley equation (GBHE) with memory}, which is  defined on $\Omega \times (0,T]$ where $\Omega\subset \mathbb{R}^d$ $(1\leq d\leq 3)$ be a bounded Lipschitz domain. The GBHE with memory is defined as:
\begin{equation}\label{5.GBHE}
\left\{\begin{aligned}
\frac{\partial u(x,t)}{\partial t}&+\alpha u(x,t)^{\delta}\sum_{i=1}^d\frac{\partial u(x,t)}{\partial x_i}-\nu\Delta u(x,t)-\eta\int_{0}^{t} K(t-s)\Delta u(x,s) \mathrm{~d}s\\&=\beta u(x,t)(1-u(x,t)^{\delta})(u(x,t)^{\delta}-\gamma)+f(x,t),  \ (x,t)\in\Omega\times(0,T],\\ u(x,t)&=0, \ x \in \partial\Omega ,\ t\in(0,T],\\
u(x,0)&=u_0(x), \ x\in\overline{\Omega},
\end{aligned}\right.
\end{equation}
where $K(t)$ denotes the weakly singular kernel. 

The generalized Burgers-Huxley equation is a non-linear advection-diffusion-reaction problem with broad applications across diverse fields. Examples include modeling traffic flows, nuclear waste disposal, the movement of fish populations, and the dynamics of domain walls in ferroelectric materials under electric fields.

The incorporation of weakly singular kernels is crucial in such models, particularly in scenarios like heat conduction with memory. In traditional formulations, memory effects, such as time delays, are often overlooked when deriving equations like the heat equation using Fourier's law. However, in contexts like population or epidemic models, these memory terms play a significant role. They reflect how current diffusion processes are influenced by both the current and past concentrations of species, thus shaping the system's dynamics (see \cite{SWW} and references therein). The existence, uniqueness, and regularity results for the generalized Burgers-Huxley equation (GBHE) with weakly singular kernels have been established by the authors in \cite{GBHE} and the numerical approximation using conforming finite element scheme have also been discussed. For the $1D$ case of GBHE we refer to \cite{MKH}. Additionally, \cite{GBHE2} discusses numerical approximation methods under minimal regularity assumptions, using conforming, non-conforming, and discontinuous Galerkin methods ($h$-version). For the stationary counterpart of the GBHE, the numerical results using different finite element methods are covered in \cite{KMR}.

Discontinuous Galerkin (DG) methods, initially developed in the early 1970's for neutron transport problems \cite{LRA,RHi}, have since evolved into a versatile numerical technique with a wide range of applications. The DGFEM  was also explored as a time-stepping approach for initial value ordinary differential equations (ODEs) in \cite{RHi}, demonstrating its alignment with specific Gauss–Radau type Runge–Kutta schemes. For parabolic partial differential equations (PDEs), DG time-stepping was introduced by Jamet \cite{Jam}. Subsequent work by Eriksson, Johnson, Thomée, and their  collaborators \cite{EJo,KJC,EJT,SVL,Vth} expanded the application of DGFEM to parabolic problems. Initially, these studies focused on the convergence of the approximate solution to the exact solution by refining the space discretization ($h \rightarrow 0$) and the time discretization ($k \rightarrow 0$). This approach, known as the $h$-version of FEM, achieves algebraic convergence in both space and time.

Solutions to parabolic problems often exhibit piecewise analytic behavior over time and may develop time singularities due to incompatible or discontinuous data. In cases involving weak singularities in the kernel that gradually smooth out, it becomes clear that the $h$-version of the Finite Element Method (FEM) alone may not suffice. To address these challenges, the $p$- and $hp$-versions of FEM, initially developed by Babuška, Suri, and their colleagues in the early 1980s \cite{BSu}, offered a powerful approach by combining mesh refinement with polynomial degree enhancement. This strategy enabled exponential rates of convergence even for non-smooth solutions arising from singularities associated with domain geometry or mixed boundary conditions. Building on this foundation, Schötzau and Schwab extended the $hp$-FEM concept to Discontinuous Galerkin (DG) time stepping in 2000 \cite{SSc1,SSc2}, further advancing the method's application to time-dependent problems.

A substantial body of literature exists on the $hp$-DGFEM for parabolic integrodifferential equations and Volterra integrodifferential equations, particularly focusing on the application of DG time-stepping methods. For example, the $hp$-version of the discontinuous Galerkin (DG) method was initially developed for non-linear initial value problems as presented in \cite{SSc1} and later extended to parabolic problems in \cite{SSc2}. This approach was subsequently applied to Volterra integral and integro-differential equations (VIDEs), including parabolic VIDEs, as detailed in \cite{BSc1,MBM}. Moreover, an $hp$-version of the continuous Galerkin method for non-linear initial value problems was introduced in \cite{Wih}. The $hp$-version of the continuous Petrov–Galerkin (CPG) method was also explored, specifically for linear and non-linear VIDEs with smooth solutions, as outlined in \cite{Yi,YGU}. Recently, the exponential convergence of the $hp$-time-stepping method in space-time has been proven in \cite{PSZ}, and the $hp$-version of the discontinuous Galerkin method for fractional integro-differential equations with weakly singular kernels was established in \cite{CCH}.

To the best of the authors' knowledge, the literature on the $hp$-version for non-linear problems is scarce, with no existing studies specifically addressing the GBHE with weakly singular kernels. The main contributions of this works are as follow:
\begin{itemize} \item \textbf{$hp$-DG Time Stepping:} We introduce a novel $hp$-DG time stepping method for the GBHE with a weakly singular kernel. Our method achieves optimal error estimates for the semi-discrete scheme. While the literature extensively covers $hp$-DG methods for linear integro-differential and Volterra integral equations, approximation results for non-linear models are still limited.
	
	\item \textbf{Fully-Discrete Scheme:} We explore a fully-discrete scheme that uses $hp$-version conforming finite elements for spatial discretization and $hp$-DG time stepping for temporal discretization. The analysis shows optimal error estimates. We also investigate the $h$-$p$-$r$-$q$ variant using Discontinuous Galerkin Finite Element Methods (DGFEM), building on the novel formulations introduced in \cite{GBHE2}. The error estimates are optimal in all parameters except for the $p$ and $r$ space and time polynomial approximations.
	
	\item \textbf{Numerical Validation:} Numerical tests validate the accuracy of our method across various problem dimensions. Section \ref{5.sec4} discusses the application of our algorithm to solving fractional differential equations, demonstrating its effectiveness in different scenarios. \end{itemize}

This work introduces novel a priori error estimates for the $hp$-FEM, showcasing spectral convergence. Furthermore, we present the fully-discrete scheme using $hp$-DG time-stepping, combined with $hp$-DGFEM in space, through a novel formulation similar to the $h$-version of DGFEM discussed in \cite{GBHE2}. This formulation relaxes the conditions on parameter restrictions and contributes to the well-posedness of the problem.

The outline of this work is as follows: In Section \ref{5.sec2}, we first present some preliminaries and notations used throughout the paper, along with the abstract formulation for the GBHE with memory as defined in \eqref{5.GBHE}. The $hp$-DG time-stepping method is proposed in Section \ref{5.ss2}, where we also discuss key lemmas that are useful for deriving a priori error estimates. Theorem \ref{5.th2} provides the error estimates for the semi-discrete scheme, and Corollary \ref{5.co1} presents the convergence estimates for uniform discretization. In Section \ref{5.sec3}, we introduce the fully-discrete FEM scheme \eqref{5.3.2}, starting with the conforming finite elements, followed by the $hp$-DGFEM in Section \ref{5.sub3}. The formulation for DGFEM is discussed in detail in \eqref{5.3.2Dg}. Finally, numerical results are presented in Section \ref{5.sec4} to validate the scheme on various test problems.

\section{$hp$-DG time-stepping }\label{5.sec2}\setcounter{equation}{0}
\subsection{Notations and Weak formulation}
Let $\mathrm{C}_0^{\infty}(\Omega)$ denotes the space of all infinitely differentiable functions with compact support on $\Omega$. For $p\in[2,\infty)$, the Lebesgue spaces are denoted by $\L^p(\Omega)$ and the Sobolev spaces are denoted by $\H^{k}(\Omega)$. The norm in $\L^p(\Omega)$ is denoted by $\H^1_{\L^p}$ and for $p=2$, the inner product in $\L^2(\Omega)$ is denoted by $(\cdot,\cdot)$. Let $\H_0^1(\Omega)$ denotes closure of $\mathrm{C}_0^{\infty}(\Omega)$ in $\|\nabla \cdot\|_{\L^2}$ norm. Using the Poincar\'e inequality, we have $\|\nabla \cdot\|_{\L^2}$ is equivalent to full $\H^1$ norm. As $\Omega$ is a bounded domain, note that $\|\nabla\cdot\|_{\L^2}$  defines a norm on $\H^1_0(\Omega)$ and we have the continuous embedding $\H_0^1(\Omega)\subset\L^2(\Omega)\subset\H^{-1}(\Omega)$, where $\H^{-1}(\Omega)$ is the dual space of $\H_0^1(\Omega)$. Remember that the embedding of $\H_0^1(\Omega)\subset\L^2(\Omega)$ is compact. The duality paring between $\H_0^1(\Omega)$ and $\H^{-1}(\Omega)$ is denoted by $\langle\cdot,\cdot\rangle$.
The abstract formulation of the system (\ref{5.GBHE}) is given by 
\begin{equation}\label{5.AFGBHEM}
	\left\{
	\begin{aligned}
		\frac{\mathrm{~d}u(t)}{\mathrm{~d}t}+\nu Au(t)+\eta ( K*Au)(s)&=-\alpha B(u(t))+\beta c(u(t))+f(t), \ \text{ in }\ \H^{-1}(\Omega),\\
		u(0)&=u_0\in\L^2(\Omega),
	\end{aligned}
	\right.
\end{equation}
for a.e. $t\in[0,T],$ where the operators involved are as follows
\begin{align*}
	Au =-\Delta u, \quad (K*Au(t)) = \int_0^tK(t-s)Au(s) \mathrm{d}s, \quad B(u)= u^{\delta}\sum\limits_{i=1}^d\frac{\partial u}{\partial x_i},\quad \text{ and }\ c(u)=u (1-u^{\delta} )(u^{\delta} -\gamma). 
\end{align*}
For the Dirichlet Laplacian operator $A$, the $D(A) = \H^2(\Omega)\cap \H_0^1(\Omega).$ The non-linear locally Lipschitz operators (see Sec 2.1.2., \cite{GBHE}) are defined as $B(\cdot): \H_0^1(\Omega)\cap \L^{2(\delta+1)}(\Omega) \to\H^{-1}(\Omega)+\L^{\frac{2(\delta+1)}{2\delta+1}}(\Omega)$ and $c(\cdot):\L^{2(\delta+1)}(\Omega)\to\L^{\frac{2(\delta+1)}{2\delta+1}}(\Omega).$

 As the embedding  $\H_0^1(\Omega)\hookrightarrow \L^2(\Omega)$ is compact, $A^{-1}$ exists and is a symmetric compact operator on $\L^2(\Omega)$. Then, by using spectral theorem \cite{RDJL}, there exists a sequence of eigenvalues of $A$, $0<\lambda_1\leq\lambda_2\leq\ldots\to\infty$  and an orthonormal basis $\{w_k\}_{k=1}^{\infty}$ of $\L^2(\Omega)$ consisting of eigenfunctions of $A$ (\cite{RDJL} p. 504). We associate with $A$ the bilinear form $A: \H_0^1(\Omega) \times \H_0^1(\Omega)\rightarrow \mathbb{R}$ defined in the in terms of eigenfunctions as: For $u,v \in \H_0^1(\Omega)$ 
$$ (Au, v) := \sum_{k=1}^{\infty} \lambda_k u_k v_k, \qquad u_k = \langle u, w_k\rangle \text{ and }  v_k = \langle v, w_k\rangle.$$

Moreover, one can define the fractional powers of $A$ and 
$$\|A^{1/2}u\|_{\L^2}^2=|\sum_{j=1}^{\infty}\lambda_j\langle u,w_j\rangle|^2\geq \lambda_1\sum_{j=1}^{\infty}|\langle u,w_j\rangle|^2,$$ which is the Poincar\'e inequality. Note also that $\|u\|_{\H^{s}}=\|A^{s/2}u\|_{\L^2},$ for all $s\in\mathbb{R}$. An integration by parts yields $$(Au,v)=(\nabla u,\nabla v)=:a(u,v), \ \text{ for all } \ v\in\H_0^1(\Omega),$$ so that $A:\H_0^1(\Omega)\to\H^{-1}(\Omega)$. The bilinear form $a(u,v)$ is symmetric, continuous and coercive. That is, we have 
\begin{align*}
	a(u,v) &= a(u,v) \hspace{10mm}  \forall u,v \in \H_0^1(\Omega),\\
	a(u,v) &\leq \|u\|_{\H_0^1}\|v\|_{\H_0^1} \hspace{3mm}  \forall u,v \in \H_0^1(\Omega),\\
	a(u,u) &\geq  \|u\|_{\H_0^1} \hspace{12mm}  \forall u,v \in \H_0^1(\Omega).
\end{align*}

For any $v\in \H_0^1(\Omega)$ we have 
\begin{align*}
	\eta\left\langle\int_{0}^{t} K(t-s)A u(s) \mathrm{~d}s,v\right\rangle &= \left\langle \int_0^tK(t-s) \sum_{k=1}^{\infty} \lambda_k u_k(s) w_k \mathrm{~d}s, \sum_{j=1}^{\infty} v_j w_j\right\rangle\\\nonumber & = \sum_{k=1}^{\infty} \lambda_k \int_{0}^{t}K(t-s) u_k(s) \mathrm{~d}s \left\langle w_k, \sum_{j=1}^{\infty} v_jw_k\right\rangle \\\nonumber & = \sum_{k=1}^{\infty} \int_{0}^{t}K(t-s)\lambda_k u_k(s) v_k \mathrm{~d}s \\\nonumber & =  \int_{0}^{t}K(t-s)\sum_{k=1}^{\infty}\lambda_k u_k(s) v_k \mathrm{~d}s = \int_0^t K(t-s) (Au(s),v) \mathrm{~d}s.
\end{align*}
The kernel function $K(\cdot)$ is  \emph{weakly singular} such  that $K\in\L^1(0,T)$.  Moreover, we assume that the function $K(\cdot)$ is a \emph{positive kernel}, that is,  if for any $T> 0$, we have
\begin{align}\label{5.pk}
	\int_0^T \int_0^t K(t-s)u(s)u(t)ds \mathrm{~d}t\geq 0,\  \ \text{ for all }\ u\in \L^2(0, T).
\end{align}
The regularity results, under the smoothness assumptions on the initial data $u_0\in D(A)$ and external forcing $f\in\H^1(0,T;\L^2(\Omega))\cap \L^2(0,T;\H^1(\Omega))$, ensure that the solution $u \in \L^{\infty}(0,T;\H^2(\Omega))$ and $\partial_tu\in\L^{\infty}(0,T;\L^2(\Omega))\cap\L^2(0,T;\H_0^1(\Omega))$. These results have already been established in \cite[Theorem 2.11]{GBHE}.
\subsection{Time discretization}\label{5.ss2}
To study the $hp$-DG method, consider a partition $\mathcal{M}$ of the time interval $[0, T]$ (which may be non-uniform) given by 
$ 0 = t_0 < t_1 < \cdots < t_N = T. $
Define the sub-intervals $\J_n = (t_{n-1}, t_n]$, the time step sizes $ k_n = t_n - t_{n-1} $, and the maximum time step size $ k = \max\limits_{1 \leq n \leq N} k_n $. Each sub-interval $\J_n$ is associated with a polynomial degree $ p_n \in \mathbb{N}_0 $, which are collectively stored in the vector
\begin{align}\label{5.6}
	\textbf{p}\coloneqq (p_1,p_2, \cdots, p_n).
\end{align} 
We introduce the discontinuous finite element space as follows:
\begin{align}\label{5.7}
	\mathcal{S}(\mathcal{M},\textbf{p})=\{{v}: [0,T] \rightarrow \H_0^1(\Omega) : {v}|_{\J_n} \in \mathbb{P}_{p_n}, 1 \leq n \leq N\},
\end{align}
for $\mathbb{P}_{p_n}$ being the space of all polynomials of degree at most $p_n$ on $\J_n$. We adopt the usual convention that a function $v \in 	\mathcal{S}(\mathcal{M},\textbf{p})$ is left-continuous at each time level $t_n$, as follows
\begin{align*}
v^n = v(t_n)= v(t_n^{-}), \quad v_{+}^n = v(t_n^{+}), \quad [v]^n = v_{+}^n - v^{n}.
\end{align*}
The $hp$-DG formulation $U \in \mathcal{S}(\mathcal{M},\textbf{p}) $ of the GBHE with memory \eqref{5.AFGBHEM} in this context reads: Given $U(t)$ for $0 \leq t \leq t_{n-1},$
find the approximation $U\in \mathbb{P}_{p_n}$ on the next time-steps $\J_n$ by 
\begin{align}\label{5.DGS}
\nonumber	\langle U_{+}^{n-1}, X_{+}^{n-1} \rangle + \int_{\J_n}\Big[\langle U',X\rangle &+ \nu a(U,X) + \alpha(B(U),X) + \eta \left(K*\nabla U,\nabla X \right) \\& - \beta (c(U),X) \Big] \mathrm{~d}t = (U^{n-1},X_{+}^{n-1}) + \int_{\J_n} \langle f,X \rangle \mathrm{~d}t \qquad \forall X\in \mathbb{P}_{p_n}.
\end{align}

This time-stepping procedure begins with an appropriate approximation $U_0$ of $u_0$, and after 
$N$ steps, it provides the approximate solution $U \in \mathcal{S}(\mathcal{M},\textbf{p})$, for $0\leq t \leq t_N$. 

The global formulation for the DG scheme \eqref{5.DGS} will be given by 
 \begin{align}\label{5.GWF}
 \nonumber F_N(U,X) &= \langle U_{+}^{0}, X_{+}^{0} \rangle +\sum_{n=1}^{N-1}\langle [U]^n, X_{+}^n \rangle + \sum_{n=1}^N \int_{t_{n-1}}^{t_n}\Big[\langle U',X\rangle + \nu a(U,X) + \alpha(B(U),X) \\&\quad + \eta \left(K*\nabla U,\nabla X \right)  - \beta (c(U),X) \Big] \mathrm{~d}t.
 \end{align}
By summing over all the time steps, the DG method can equivalently be written as follows: Find $U\in \mathcal{S}(\mathcal{M},\textbf{p})$ such that 
\begin{align}\label{5.DGwf}
	F_N(U,X) = \langle U^0, X_{+}^0\rangle + \int_0^{t_N}\langle f,X \rangle \mathrm{~d}t  \quad \forall ~X\in  \mathcal{S}(\mathcal{M},\textbf{p}).
\end{align}

\begin{remark}\label{5.remark1}
	
	The alternative expression for $F_N$ can be derived by applying integration by parts to equation \eqref{5.GWF} as follows
	\begin{align}\label{5.AGWF}
		\nonumber F_N(U,X) &= \langle U^{N}, X^{N} \rangle - \sum_{n=1}^{N-1}\langle U^n, [X]^n \rangle + \sum_{n=1}^N \int_{t_{n-1}}^{t_n}\Big[-\langle U,X'\rangle + \nu a(U,X) + \alpha(B(U),X) \\&\quad + \eta  \left(K*\nabla U,\nabla X \right) - \beta (c(U),X) \Big] \mathrm{~d}t.
	\end{align}
\end{remark}
\begin{remark}\label{5.rem2}
	For the polynomial approximation, we set $U$ on the interval $J_n$ as follows:
	\begin{equation}
		U = \sum_{j=0}^{p} \tilde{U}_j k_n^{-j}\left(t - t_{n-1}\right)^j,
	\end{equation}
	where $\tilde{U}_j \in \L^2(\Omega)$ are the coefficients to be determined. This form is analogous to that given in \cite[eq. 2.2]{EJT}.
	
	The determination of the coefficients $\tilde{U}_j $ is governed by the following system, where $\delta_{i, j}$ denotes the Kronecker delta function:
	\begin{align}\label{5.sys}
	&\no	\int_{J_n}\bigg(\sum\limits_{j=1}^{p} j \tilde{U}_j k_n^{-j}\left(s-t_{n-1}\right)^{j-1}+\sum\limits_{j=0}^{p} A \tilde{U}_j k_n^{-j}\left(s-t_{n-1}\right)^j+\sum\limits_{j=0}^{p} B ( \tilde{U}_j k_n^{-j}\left(s-t_{n-1}\right)^j)\\&\no+\sum\limits_{j=0}^{p} c\big( \tilde{U}_j k_n^{-j}\left(s-t_{n-1}\right)^j\big) + \sum\limits_{j=0}^{p}\int_0^tK(t-s) A \tilde{U}_j k_n^{-j}\left(s-t_{n-1}\right)^j\bigg) k_n^{-l}\left(s-t_{n-1}\right)^l d s+\tilde{U}_0 \delta_{l, 0} \\&
		=U^{n-1} \delta_{l, 0}+\int_{J_n} f(s) k_n^{-l}\left(s-t_n\right)^l d s \quad \text { for } \quad l=0,1, \ldots, p.
	\end{align}

	If we choose approximated solution to be a constant, that is, $\textbf{p}= (0,\cdots,0)$ in \eqref{5.DGS}, we have $U(t) = \tilde{U}_0 = U^{n}$ in $\J_n.$ Thus, we are solving

	\begin{align*}
		\langle U^{n}, X \rangle + k_n\Big[ \nu a(U^{n},X) + \alpha(B(U^{n}),X) & - \beta (c(U^{n}),X) \Big] \mathrm{~d}t + \eta\int_{J_n} \left(K*\nabla U,\nabla X \right) \mathrm{~d}t \\&= (U^{n-1},X) + \left(\ \int_{\J_n}f(t)\mathrm{~d}t,X \right) .
	\end{align*}
	for memory term to be zero ($\eta = 0$)  it can written as 
	\begin{align*}
		\left(\frac{U^{n}-U^{n-1}}{k_n}, X \right)+ \nu a(U^{n},X) + \alpha(B(U^{n}),X) & - \beta (c(U^{n}),X) \mathrm{~d}t = \frac{1}{k_n}\left( \int_{\J_n}f(t)\mathrm{~d}t,X \right) .
	\end{align*}
	which is the backward Euler scheme, where $f$ is replaced by the average over the $\J_n$. For the non zero memory ($\eta \neq 0$) term, we have 
	$$
	\frac{1}{k_n} \int_{J_n} \left(K*\nabla U,\nabla X \right)\mathrm{~d}t =\frac{1}{k_n} \int_{t_{n-1}}^{t_n}\int_0^t K(t-s)(\nabla U(s),\nabla X)\mathrm{~d}s \mathrm{~d}t = \sum_{j=1}^n \omega_{n j} \nabla U^j k_j,
	$$
	where 
	$$
	\omega_{n j}=\frac{1}{k_n k_j} \int_{t_{n-1}}^{t_n} \int_{t_{j-1}}^{\min \left(t, t_j\right)} K(t-s) \mathrm{~d}s \mathrm{~d}t .
	$$
	Thus, at each time step we are solving system of non linear equation given as
	$$
	U^{n} + k_n\nu AU^n + k_n\alpha B(U^n) -k_n\beta c(U^n)  +\eta\omega_{n n} A U^{n}=U^{n-1}+k_n \int_{\J_n}f(t)-k_n\eta \sum_{j=1}^{n-1} \omega_{n j} A U^j k_j .
	$$
	Note that for an approximating solution of degree $ j$, we are solving a system of  $j$ equations as discussed in \eqref{5.sys}.

\end{remark}
\subsection {Projection operators}
	Let us introduce a projection operator for the analysis of the $h p$-version of DG time-stepping methods, \cite{MBM}. In our setting, it is given as follows: For a continuous function $\widehat{u}$ : $[-1,1] \rightarrow \H_0^1$, we define $\widehat{\Pi}^p \widehat{u}:[-1,1] \rightarrow \mathbb{P}_p$ given by $(p+1)$ conditions as 	
	 \begin{align}\label{5.8}
	 	 \widehat{\Pi}^p \widehat{u}(1)=\widehat{u}(1) \in \H_0^1 \quad and \quad \int_{-1}^1\left\langle\widehat{u}-\widehat{\Pi}^p \widehat{u}, v\right\rangle d t=0 \quad \forall v \in \mathbb{P}_{p-1}.
	 \end{align} 
	Now for the linear case ($p=0$), the first condition $\widehat{\Pi}^p \widehat{u}(1)=\widehat{u}(1)$ is necessary. From \cite[Lemma 3.2]{SSc2}, it follows that $\widehat{\Pi}^p$ is well defined.
	
	Consider the affine linear transinformation, $F_n:[-1,1] \rightarrow \bar{J}_n$ given by $F_n(\hat{t})=\left(k_n \hat{t}+t_n+t_{n-1}\right) / 2$. For any continuous function $u:[0, T] \rightarrow \H_0^1$ we now define the piecewise $hp$-interpolant $\Pi u:[0, T] \rightarrow \mathcal{S}(\mathcal{M}, \mathbf{p})$ by setting
	\begin{align}\label{5.hpp}
	\left.(\Pi u)\right|_{\J_n}=\widehat{\Pi}^{p_n}\left(u \circ F_n\right) \circ F_n^{-1}, \quad 1 \leq n \leq N.
	\end{align}
Further to discuss the approximation properties of $\Pi$, we define
	$$
	\Gamma_{p, q}=\frac{\Gamma(p+1-q)}{\Gamma(p+1+q)},\qquad \text{and} \qquad	\|\phi\|_{\J_n}=\sup _{t \in \J_n}\|\phi(t)\|.
	$$
	The approximation results \cite{SSc1, SSc2} are given as:
	\begin{theorem}\label{5.th1}
	 For $1 \leq n \leq N$, let $u \in \C\left(\left[t_{n-1}, t_n\right] ; \H_0^1\right)$. Then we have the following:\\
	(i) If $u$ is analytic on $\left[t_{n-1}, t_n\right]$ with values in $\H_0^1$, there holds
	$$
	\int_{t_{n-1}}^{t_n}\|\nabla(\Pi u-u)\|_{\L^2}^2 \mathrm{~d}t+k_n\|\Pi u-u\|_{\J_n}^2 \leq C k_n \exp \left(-\tilde{b} p_n\right).
	$$
	(ii) For any $0 \leq q_n \leq p_n$ and $\left.u\right|_{\J_n} \in \H^{q_n+1}\left(\J_n ; \H_0^1\right)$, there holds
	$$
	\int_{t_{n-1}}^{t_n}\|\nabla(\Pi u-u)\|_{\L^2}^2 \mathrm{~d}t \leq \frac{C}{\max \left\{1, p_n^2\right\}}\left(\frac{k_n}{2}\right)^{2 q_n+2} \Gamma_{p_n, q_n} \int_{t_{n-1}}^{t_n}\left\|\nabla u^{\left(q_n+1\right)}\right\|_{\L^2}^2 \mathrm{~d}t.
	$$
	(iii) For any $0 \leq q_n \leq p_n$ and $\left.u\right|_{\J_n} \in \H^{q_n+1}\left(\J_n ; \L^2\right)$, there holds
	\begin{align*}
			\|\Pi u-u\|_{\J_n}^2 &\leq C\left(\frac{k_n}{2}\right)^{2 q_n+1} \Gamma_{p_n, q_n} \int_{t_{n-1}}^{t_n}\left\|u^{\left(q_n+1\right)}\right\|^2_{\L^2} \mathrm{~d}t\\
			&\leq C\left(\frac{k_n}{2}\right)^{2 q_n+1} \Gamma_{p_n, q_n} \int_{t_{n-1}}^{t_n}\left\|\nabla u^{\left(q_n+1\right)}\right\|^2_{\L^2} \mathrm{~d}t,
	\end{align*}

	where the constant $C$ and $\tilde{b} $ is independent of $k_n$ and $p_n.$
	\end{theorem}
\begin{proof}
The subsequent result is directly inferred from Theorem 3.2 and Remark 3.3 in \cite{MBM}.
\end{proof}
To derive the error bounds in $\H^1_{J_n}$, we use the following inverse estimate [\cite{SSc1}, lemma 3.1] as follow
\begin{lemma}\label{5.leminv}
	Let $\phi\in \mathcal{S}(\mathcal{M},\textbf{p}).$ Then for $1\leq n\leq N$, we have 
	\begin{align}\label{5.inv}
		\|\phi\|_{J_n}^2 \leq C \left(\log(p_n+2)\int_{t_{n-1}}^{t_n}\|\phi'\|_{\L^2}(t-t_{n-1}) \mathrm{~d}t + \|\phi^n\|_{\L^2}\right).
	\end{align} 
\end{lemma}
\subsection{Error Estimates}
We will now deduce the error estimate associated with the  $hp$-DG approximation. Consider $U$ as the approximate solution defined in equation \eqref{5.DGwf}, and let $u : [0,T]\rightarrow \H_0^1$ be a continuous function. For the analysis of the error, we will employ the following lemma for kernel and the discrete version of Gr\"onwall's inequality.
The following lemmas are useful for the further analysis:
\begin{lemma}\cite{MTM}\label{5.l11}
	Let $K \in \L^1(0,T)$ and $\phi\in \L^2(0,T)$ for some $T>0$. Then
	$$\left(\int_0^{T}\left(\int_0^sK(s-\tau)\phi(\tau)\mathrm{~d}\tau\right)^2\mathrm{~d}s\right)^{\frac{1}{2}}\leq \left(\int_0^{T}|K(s)|\mathrm{~d}s\right)\left(\int_0^{T}\phi^2(s)\mathrm{~d}s\right)^\frac{1}{2}.$$
\end{lemma}
In the particular case for $K(t) = \frac{1}{t^{\epsilon-1}}$ for $0\leq \epsilon\leq 1$, we have the subsequent lemma.
\begin{lemma}[\cite{SVL}, lemma 6.3]\label{5.l1}
	Let $\phi\in \L^2(0,T)$ and $\alpha\in(0,1)$, then
	$$\int_0^{T}\left(\int_0^s(t-s)^{\alpha-1}\phi(\tau)\mathrm{~d}\tau\right)^2\mathrm{~d}s\leq \frac{T^{\alpha}}{\alpha} \int_0^{T}(T-s)^{\alpha-1}\int_0^{s}\phi^2(\gamma)\mathrm{~d}\gamma\mathrm{~d}s.$$
\end{lemma}
\begin{lemma}[\cite{SVL}, lemma 6.4]\label{5.l2}
Let $\{a_j\}_{j=1}^{N}$ and $\{b_j\}_{j=1}^{N}$  be a sequence of non-negative numbers with $0\leq b_1\leq b_2\leq \cdots \leq b_N.$ Assume that there exists a constant $K\geq 0,$ such that
$$
a_n\leq b_n + K\sum_{j=1}^n a_j\int_{t_{j-1}}^{t_j}(t_n-t)^{\alpha-1} \mathrm{~d}t \quad \text{for}\quad 1\leq n\leq N\quad \text{and}\quad \alpha \in (0,1).
$$
Assume further that $\kappa = \frac{Kk^{\alpha}}{\alpha}<1$. Then for $n= 1, \cdots, N, $ we have $a_n\leq Cb_n$, where $C$ is a constant that depends on $K, T, \alpha,$ and $\kappa.$
\end{lemma}

For the error $U-u$, we decompose it into two terms as:
\begin{align}\label{5.10}
	\underbrace{U-u}_{e} = \underbrace{(U - \Pi u)}_{\Upsilon} + \underbrace{(\Pi u- u)}_{\rho},
\end{align}
where $\Pi$ is the $hp$-version interpolation operator defined in \eqref{5.hpp}. It's important to note that we already possess estimates for $\rho$ as provided in Theorem \eqref{5.th1}. Hence, our focus now shifts to estimating the first term $\Upsilon \in \mathcal{S}(\mathcal{M},\textbf{p})$. It's worth observing that due to the construction of the interpolant $\Pi$, we have $\rho^n = 0$ for all $1 \leq n \leq N$.
\begin{align}
\nonumber F_N(U,X)- F_N(u,X)&= \langle (U-u)_{+}^{0}, X_{+}^{0} \rangle +\sum_{n=1}^{N-1}\langle [U-u]^n, X_{+}^n \rangle + \sum_{n=1}^N \int_{t_{n-1}}^{t_n}\Big[\langle (U-u)',X\rangle\\\nonumber&\quad  + \nu a(U-u,X) + \alpha(B(U)-B(u),X) \\&\quad + \eta \left(K*\nabla (U-u),\nabla X \right)  - \beta (c(U)-c(u),X) \Big] \mathrm{~d}t.
\end{align}
Using the decomposition \eqref{5.10} and the alternative form in Remark \ref{5.AGWF} for $\rho$, we have 
\begin{align}\label{5..31}
	\nonumber&\langle \Upsilon_{+}^{0}, X_{+}^{0} \rangle +\sum_{n=1}^{N-1}\langle [\Upsilon]^n, X_{+}^n \rangle + \sum_{n=1}^N \int_{t_{n-1}}^{t_n}\Big[\langle \Upsilon',X\rangle + \nu a(\Upsilon,X)  + \eta \left(K*\nabla \Upsilon,\nabla X \right)   \Big] \mathrm{~d}t= \langle U^0-u_0, X_+^{0} \rangle  \\&- \int_{0}^{t_n}\Big[\nu a(\rho,X) + \eta \left(K*\nabla \rho,\nabla X \right) + \alpha(B(U)-B(u),X) - \beta (c(U)-c(u),X) \Big] \mathrm{~d}t,
\end{align}
 where, by utilizing the definition of the interpolant $\Pi$, we obtain $\int_{t_{n-1}}^{t_n} \langle \rho, X' \rangle \mathrm{~d}t = 0$ (it's important to note that for $p_n = 0$, we have $X' = 0$).
 \begin{lemma}\label{5.l3}
 	For $1\leq n\leq N,$ we have
 \begin{align}
 	\nonumber & \|\Upsilon^n\|_{\L^2} ^2 + \nu \int_{0}^{t_n}\|\nabla\Upsilon\|^2_{\L^2}\mathrm{~d}t + \frac{\beta}{4}\int_0^{t_n}\|U^{\delta} \Upsilon\|_{\L^2}^2 \mathrm{~d}t+ \frac{\beta}{4} \int_0^{t_n}\|(\Pi u)^{\delta} \Upsilon\|_{\L^2}^2 \mathrm{~d}t \\\nonumber&\leq \|U^0-u_0\|_{\L^2} + 2\left(\nu+ \frac{C_K\eta^2}{\nu}\right) \int_{0}^{t_n} \|\nabla\rho\|^2_{\L^2} \mathrm{~d}t 
+ \int_0^{t_n}C(\alpha,\nu)\left(\|U\|^{\frac{8\delta}{4-d}}_{\L^{4\delta}}+\|\Pi U\|^{\frac{8\delta}{4-d}}_{\L^{4\delta}}\right)\|\Upsilon\|_{\L^2}^2\mathrm{~d}t\\&\quad + C(\beta,\gamma, \delta) \int_{0}^{t_n}\|\Upsilon\|_{\L^2} \mathrm{~d}t -  2\alpha\int_0^{t_n} \big(B(\Pi u)-B(u),\Upsilon\big) \mathrm{~d}t +2\beta \int_0^{t_n}\big(c(\Pi u)-c(u),\Upsilon\big) \mathrm{~d}t.
 \end{align}
 \end{lemma}
\begin{proof}
 Take $X = \Upsilon$ and using the fact that $\langle \Upsilon',\Upsilon\rangle= \frac{1}{2}\frac{d}{dt}\|\Upsilon\|^2$ in \eqref{5..31}, we have 
 \begin{align}\label{5.14}
 	\nonumber &\|\Upsilon_{+}^{0}\|_{\L^2} ^2 + \|\Upsilon^n\|_{\L^2} ^2 +\sum_{n=1}^{N-1}\| [\Upsilon]^n\|_{\L^2} ^2+ 2\nu \int_{0}^{t_n}\|\nabla\Upsilon\|^2_{\L^2}\mathrm{~d}t +2 \eta \left(K*\nabla \Upsilon,\nabla \Upsilon \right)\\\nonumber&=2 \langle U^0-u_0, \Upsilon_+^{0} \rangle - 2\int_{0}^{t_n}\Big[\nu a(\rho,\Upsilon) + \eta \left(K*\nabla \rho,\nabla \Upsilon \right)\\&\quad\no+ \alpha(B(U)-B(u),\Upsilon) - \beta (c(U)-c(u),\Upsilon) \Big] \mathrm{~d}t\\& \coloneqq 2 \langle U^0-u_0, \Upsilon_+^{0} \rangle + 2\sum_{i=1}^4 \mathcal{J}^n_i .
 \end{align}
 Using  Cauchy-Schwarz and Young's inequalities gives
 $$
 2 \langle U^0-u_0,\Upsilon_+^{0} \rangle \leq  \|U^0-u_0\|^2_{\L^2} + \|\Upsilon_+^{0}\|^2_{\L^2}.
$$ 
Estimating $\mathcal{J}_1^n$ using  Cauchy-Schwarz and Young's inequalities as 
\begin{align}\label{5.15}
	\nonumber|\mathcal{J}_1^n| = \nu|a(\rho,\Upsilon)| = \nu\left| \int_{0}^{t_n}(\nabla \rho,\nabla \Upsilon) \mathrm{~d}t\right| &\leq \nu \int_{0}^{t_n} \|\nabla\rho\|_{\L^2}\|\nabla\Upsilon\|_{\L^2} \mathrm{~d}t \\&\leq \nu\int_{0}^{t_n}\|\nabla\rho\|_{\L^2}^2 \mathrm{~d}t +\frac{\nu}{4} \int_{0}^{t_n}\|\nabla\Upsilon\|_{\L^2}^2\mathrm{~d}t.
\end{align}
An application of Lemma \ref{5.l11} with $(T=t_n)$, we attain
\begin{align}\label{5.16}
|\mathcal{J}_2^n| = \eta \left(K * \nabla \rho, \nabla \Upsilon \right) &\leq \eta \int_{0}^{t_n} \|K * \nabla\rho\|_{\L^2} \|\nabla \Upsilon\|_{\L^2} \mathrm{~d}s 
\\&\leq\frac{ \eta^2}{\nu} \int_{0}^{t_n} \|K * \nabla\rho\|_{\L^2}^2 \mathrm{~d}s + \frac{\nu}{4} \int_{0}^{t_n} \|\nabla \Upsilon\|_{\L^2} \mathrm{~d}s \\&
\leq \frac{C_K\eta^2}{\nu} \int_{0}^{t_n} \|\nabla\rho\|_{\L^2}^2 \mathrm{~d}s + \frac{\nu}{4} \int_{0}^{t_n} \|\nabla \Upsilon\|_{\L^2} \mathrm{~d}s,
\end{align}
for $C_K = \Big(\int_0^T|K(t)| \mathrm{d}t\Big)^2.$
Rewriting the other terms, we get
\begin{align}
	\mathcal{J}_3^n &= -\alpha\int_0^{t_n}(B(U)-B(u),\Upsilon)\mathrm{~d}t  =  -\alpha\int_0^{t_n}\Big(\big(B(U)-B(\Pi u),\Upsilon\big) + \big(B(\Pi u)-B(u),\Upsilon\big)\Big) \mathrm{~d}t, \\
	\mathcal{J}_4^n &= \beta\int_0^{t_n}(c(U)-c(u),\Upsilon)\mathrm{~d}t  =  \beta\int_0^{t_n}\Big(\big(c(U)-c(\Pi u),\Upsilon\big) + \big(c(\Pi u)-c(u),\Upsilon\big)\Big) \mathrm{~d}t.
\end{align}
We derive a bound for $-\alpha\langle B(U)-B(\Pi u),(U-u)\rangle$, using integrating by parts, Taylor's formula $(0< \zeta< 1)$,  H\"older's and Young's inequalities as 
\begin{align}\label{5.18}
	-\alpha\langle B(U)-B(\Pi u),\Upsilon\rangle&=\frac{\alpha}{\delta+1} \left((U^{\delta+1}-(\Pi u)^{\delta+1})\left(\begin{array}{c}1\\\vdots\\1\end{array}\right),\nabla \Upsilon\right)\nonumber\\&=\alpha\left(\Upsilon(\zeta U+(1-\zeta) \Pi u)^{\delta}\left(\begin{array}{c}1\\\vdots\\1\end{array}\right),\nabla \Upsilon\right)\nonumber\\&\leq 2^{\delta-1} \alpha(\|U\|^{\delta}_{\L^{4\delta}}+\|\Pi U\|^{\delta}_{\L^{4\delta}})\|\Upsilon\|_{\L^{4}}\|\nabla \Upsilon\|_{\L^2}\nonumber\\&\leq2^{\delta-1} \alpha(\|U\|^{\delta}_{\L^{4\delta}}+\|\Pi U\|^{\delta}_{\L^{4\delta}})\|\Upsilon\|_{\L^{2}}^{\frac{4-d}{4}}\|\nabla \Upsilon\|_{\L^2}^{\frac{4+d}{4}}\nonumber\\&\leq \frac{\nu}{4}\|\nabla \Upsilon\|_{\L^2}^2+C(\alpha,\nu)\left(\|U\|^{\frac{8\delta}{4-d}}_{\L^{4\delta}}+\|\Pi U\|^{\frac{8\delta}{4-d}}_{\L^{4\delta}}\right)\|\Upsilon\|_{\L^2}^2,
\end{align}
for $C(\alpha, \nu) = \left(\frac{4+d}{\nu}\right)^{\frac{4+d}{4-d}}\left(\frac{4-d}{8}\right)(2^{\delta-1}\alpha)^{\frac{8}{4-d}}$.
It can be easily seen that
\begin{align}\label{5.24}
	&\beta(c(U)-c(\Pi u),\Upsilon)=\beta\left[(U(1-U^{\delta})(U^{\delta}-\gamma)-\Pi u(1-(\Pi u)^{\delta})((\Pi u)^{\delta}-\gamma),\Upsilon)\right]\nonumber\\&=-\beta\gamma\|\Upsilon\|_{\L^2}^2-\beta(U^{2\delta+1}-(\Pi u)^{2\delta+1},\Upsilon)+\beta(1+\gamma)(U^{1+\delta}-(\Pi u)^{\delta+1},\Upsilon).
\end{align}
Estimating the term $-\beta(U^{2\delta+1}-(\Pi u)^{2\delta+1},\Upsilon)$ from \eqref{5.24} as
\begin{align}\label{5.25}
	&-\beta(U^{2\delta+1}-(\Pi u)^{2\delta+1},\Upsilon)\nonumber\\&= -\beta(U^{2\delta},\Upsilon^2) -\beta((\Pi u)^{2\delta},\Upsilon^2)-\beta(U^{2\delta}(\Pi u)-(\Pi u)^{2\delta}U,\Upsilon)\nonumber\\&=-\beta\|U^{\delta}\Upsilon\|_{\L^2}^2-\beta\|(\Pi u)^{\delta}\Upsilon\|_{\L^2}^2-\beta(U(\Pi u),U^{2\delta}+(\Pi u)^{2\delta})+\beta((\Pi u)^2,u^{2\delta})+\beta((\Pi u)^2,U^{2\delta})\nonumber\\&=-\frac{\beta}{2}\|U^{\delta}\Upsilon\|_{\L^2}^2-\frac{\beta}{2}\|(\Pi u)^{\delta}\Upsilon\|_{\L^2}^2-\frac{\beta}{2}(U^{2\delta}-(\Pi u)^{2\delta},U^2-(\Pi u)^2)\nonumber\\&\leq -\frac{\beta}{2}\|U^{\delta}\Upsilon\|_{\L^2}^2-\frac{\beta}{2}\|u^{\delta}\Upsilon\|_{\L^2}^2.
\end{align}
Estimating $\beta(1+\gamma)(U^{\delta+1}-(\Pi u)^{\delta+1},\Upsilon)$ from \eqref{5.24} using Taylor's formula, H\"older's and Young's inequalities as
\begin{align}\label{5.26}
	\beta(1+\gamma)(U^{\delta+1}-(\Pi u)^{\delta+1},\Upsilon)&=\beta(1+\gamma)(\delta+1)((\zeta U+(1-\zeta)\Pi u)^{\delta}\Upsilon,\Upsilon)\nonumber\\&\leq \beta(1+\gamma)(\delta+1)2^{\delta-1}(\|U^{\delta}\Upsilon\|_{\L^2}+\|\Pi u^{\delta}\Upsilon\|_{\L^2})\|\Upsilon\|_{\L^2}\nonumber\\&\leq\frac{\beta}{4}\|U^{\delta}\Upsilon\|_{\L^2}^2+\frac{\beta}{4}\|\Pi u^{\delta}\Upsilon\|_{\L^2}^2+\frac{\beta}{2}2^{2\delta}(1+\gamma)^2(\delta+1)^2\|\Upsilon\|_{\L^2}^2.
\end{align}
By combining \eqref{5.25} and \eqref{5.26} and substituting them into \eqref{5.24}, we obtain
\begin{align}\label{5.27}
	&\beta\left[(U(1-U^{\delta})(U^{\delta}-\gamma)-\Pi u(1-(\Pi u)^{\delta})((\Pi u)^{\delta}-\gamma),\Upsilon)\right]\nonumber\\&\leq -\beta\gamma\|\Upsilon\|_{\L^2}^2-\frac{\beta}{4}\|U^{\delta}\Upsilon\|_{\L^2}^2-\frac{\beta}{4}\|(\Pi u)^{\delta}\Upsilon\|_{\L^2}^2+\frac{\beta}{2}2^{2\delta}(1+\gamma)^2(\delta+1)^2\|\Upsilon\|_{\L^2}^2.
\end{align}
Combining \eqref{5.15}-\eqref{5.27}, substituting back in \eqref{5.14} and using the positivity of the kernel \eqref{5.pk},  we attain the desired result. 
\end{proof}
To estimate the final two terms on the right-hand side, corresponding to the advection and reaction terms, we present the following lemma:
 \begin{lemma}\label{5.l4}
	For $u\in\L^{2(2\delta+1)}(J_j;\L^{2(2\delta+1)}(\Omega)) \cap\H^{q_j+1}(J_j; \H_0^1),$   $1\leq n \leq N, 0\leq q_j \leq p_j,$ we have 
	\begin{align}
\int_0^{t_n} \Big(-2\alpha\big(B(\Pi u)-B(u),\Upsilon\big)+2\beta\big(c(\Pi u)-c(u),\Upsilon\big)\Big) \mathrm{~d}t\leq C\int_0^{t_n} \left( \|\nabla \rho\|^2_{\L^2}+ \|\Upsilon\|^2_{\L^2}\right) \mathrm{~d}t.
\end{align}
\end{lemma}
	Note that for the initial data $u_0\in \L^{d\delta}(\Omega)$, we have that $u\in\L^{2(2\delta+1)}(0, T ;  \L^{2(2\delta+1)}(\Omega)),$ as discussed in \cite[Remark 2.6]{GBHE}. 
 \begin{proof}
	For the first part, using an integration by parts, Taylor's formula, H\"older's and Young's inequalities, we rewrite $\big(B(\Pi u)-B(u),\Upsilon\big)$ as
	\begin{align}\label{5.32}
		-\alpha\int_0^{t_n}\big(B(\Pi u)-B(u),\Upsilon\big)\mathrm{~d}t&=-\frac{\alpha}{\delta+1}\sum_{i=1}^d\int_0^{t_n}\left(\frac{\partial}{\partial{x_i}}(({\Pi u})^{\delta+1}-u^{\delta+1}),\Upsilon\right)\mathrm{~d}t\no\\&=\frac{\alpha}{\delta+1}\sum_{i=1}^d\int_0^{t_n}\left((\Pi u)^{\delta+1}-u^{\delta+1},\frac{\partial\Upsilon}{\partial x_i}\right)\mathrm{~d}t\nonumber\\&=\alpha\sum_{i=1}^d\int_0^{t_n}\left((\zeta (\Pi u)+(1-\zeta)u)^{\delta}(\Pi u-u),\frac{\partial\Upsilon}{\partial{x_i}}\right)\mathrm{~d}t\nonumber\\&\leq 2^{\delta-1}\alpha\int_0^{t_n}\left(\|(\Pi u)^{2\delta}\|_{\L^2}^{\frac{1}{2} }+ \|u^{2\delta}\|_{\L^2}^{\frac{1}{2}} \right)\|\Pi u-u\|_{\L^4}\|\nabla \Upsilon\|_{\L^2}\mathrm{~d}t\nonumber\\&\leq \frac{2^{2\delta-1}\alpha^2}{\nu}\int_0^{t_n}\left(\|\Pi u\|^{2\delta}_{\L^{4\delta}}+ \|u\|^{2\delta}_{\L^{4\delta}} \right)\|\nabla \rho\|^2_{\L^2} \mathrm{~d}t+ \frac{\nu}{8}\int_0^{t_n}\|\nabla \Upsilon\|^2_{\L^2}\mathrm{~d}t.
	\end{align}
	And we can also rewrite the second term $(c(\Pi u)-c(u),\Upsilon)$ as 
	\begin{align}\label{5.33}
	&\int_0^{t_n}	(c(\Pi u)-c(u),\Upsilon)\mathrm{~d}t\\&= \underbrace{\beta(1+\gamma)\int_0^{t_n}((\Pi u)^{\delta+1}-u^{\delta+1},\Upsilon)\mathrm{~d}t}_{\mathcal{J}^n_5}-\underbrace{\beta\gamma\int_0^{t_n}(\Pi u-u,\Upsilon)\mathrm{~d}t}_{\mathcal{J}^n_6}
	-\underbrace{\beta\int_0^{t_n}((\Pi u)^{2\delta+1}-u^{2\delta+1},\Upsilon)\mathrm{~d}t}_{\mathcal{J}^n_7}.
	\end{align}
	We estimate $\mathcal{J}^n_5$ using Taylor's formula, H\"older's and Young's inequalities as 
	\begin{align}\label{5.35}
		\no \mathcal{J}^n_5&=
		2\beta(1+\gamma)(\delta+1)\int_0^{t_n} ((\zeta \Pi u+(1-\zeta)u)^{\delta}(\Pi u-u),\Upsilon)\mathrm{~d}t\\
		&\leq 2^{\delta}\beta(1+\gamma)(\delta+1)\int_0^{t_n} \left(\|{(\Pi u)}^{2\delta}\|_{\L^2}^{1/2}+\|{u}^{2\delta}\|_{\L^2}^{1/2}\right)\|\Pi u-u\|_{\L^4}\|\Upsilon\|_{\L^2}\mathrm{~d}t\no\\&\leq \frac{\beta}{2}\int_0^{t_n} \left(\|{\Pi u}\|_{\L^{4\delta}}^{2\delta}+\|{u}\|_{\L^{4\delta}}^{2\delta}\right)\|\nabla\rho\|^2_{\L^2}\mathrm{~d}t+ 2^{2\delta-1}\beta(1+\gamma)^2(\delta+1)^2\int_0^{t_n} \|\Upsilon\|^2_{\L^2}\mathrm{~d}t.
	\end{align}
	Using Cauchy-Schwarz and Young's inequalities, an estimate for $J_6$ reads 
	\begin{align}\label{5.36}
	\mathcal{J}^n_6 \leq \beta\gamma\int_0^{t_n}\|\Pi u-u\|_{\L^2}\|\Upsilon\|_{\L^2}\mathrm{~d}t\leq\frac{\beta\gamma}{2}\int_0^{t_n}\|\rho\|^2_{\L^2}\mathrm{~d}t+ \frac{\beta\gamma}{2}\int_0^{t_n}\|\Upsilon\|^2_{\L^2}\mathrm{~d}t.
\end{align}
A bound for $\mathcal{J}^n_7$ can be obtained by applying Taylor's formula in conjunction with H\"older's and Young's inequalities, as follows:
	\begin{align}\label{5.37}
		\mathcal{J}^n_7&=-(2\delta+1)\beta\int_0^{t_n}((\zeta \Pi u+(1-\zeta)u)^{2\delta}(\Pi u-u),\Upsilon)\mathrm{~d}t \nonumber\\
		&\leq 2^{2\delta-1}(2\delta+1)\beta\int_0^{t_n} \left(\|{(\Pi u)}^{2\delta}\|_{\L^2}+\|{u}^{2\delta}\|_{\L^2}\right)\|\Pi u-u\|_{\L^4}\|\Upsilon\|_{\L^4}\mathrm{~d}t \no\\&\leq \frac{\nu}{8}\int_0^{t_n}\|\nabla\Upsilon\|_{\L^2}^2\mathrm{~d}t + 2^{4\delta-1}(2\delta+1)^2\beta^2 \int_0^{t_n}\left(\|{\Pi u}\|^{4\delta}_{\L^{4\delta}}+\|{u}\|^{4\delta}_{\L^{4\delta}}\right)\|\nabla\rho\|_{\L^2}^2\mathrm{~d}t.
	\end{align}
	Combining \eqref{5.32}-\eqref{5.37}, we have 
	\begin{align}\label{5.28}
		\nonumber&\int_0^{t_n} \Big(-2\alpha\big(B(\Pi u)-B(u),\Upsilon\big)+2\beta\big(c(\Pi u)-c(u),\Upsilon\big)\Big) \mathrm{~d}t\\&\leq  \left(\frac{2^{2\delta-1}\alpha^2}{\nu}+\frac{\beta}{2}\right)\int_0^{t_n}\left(\|\Pi u\|^{2\delta}_{\L^{4\delta}}+ \|u\|^{2\delta}_{\L^{4\delta}} \right)\|\nabla \rho\|^2_{\L^2} \mathrm{~d}t+ \frac{\nu}{4}\int_0^{t_n}\|\nabla \Upsilon\|^2_{\L^2}\mathrm{~d}t\no\\&\quad+ \left(2^{2\delta-1}\beta(1+\gamma)^2(\delta+1)^2+\frac{\beta\gamma}{2}\right)\int_0^{t_n} \|\Upsilon\|^2_{\L^2}\mathrm{~d}t+\frac{\beta\gamma}{2}\int_0^{t_n}\|\rho\|^2_{\L^2}\mathrm{~d}t\no\\&\quad+ 2^{4\delta-1}(2\delta+1)^2\beta^2 \int_0^{t_n}\left(\|{\Pi u}\|^{4\delta}_{\L^{4\delta}}+\|{u}\|^{4\delta}_{\L^{4\delta}}\right)\|\nabla\rho\|_{\L^2}^2\mathrm{~d}t.
	\end{align}
	Finally using the regularity of $u\in\L^{2(2\delta+1)}(J_j;\L^{2(2\delta+1)}(\Omega)) \cap\H^{q_j+1}(J_j; \H_0^1),$ and  Poincar\'e inequality we have the desired result.
	\end{proof}
So, we achieve the following bound
	\begin{lemma}\label{5.l.4}
		For $1\leq n\leq N,$ we have
		\begin{align}
			 \|\Upsilon^n\|_{\L^2} ^2 + \nu \int_{0}^{t_n}\|\nabla\Upsilon\|^2_{\L^2}\mathrm{~d}t \leq C\left(\|U^0-u_0\|^2_{\L^2} + \int_0^{t_n} \left( \|\nabla \rho\|^2_{\L^2}\right) \mathrm{~d}t\right).
		\end{align}
	\end{lemma}
\begin{proof}
	Combining Lemma \ref{5.l2} with Lemma \ref{5.l3}, gives
	 \begin{align}
		\nonumber & \|\Upsilon^n\|_{\L^2} ^2 + \nu \int_{0}^{t_n}\|\nabla\Upsilon\|^2_{\L^2}\mathrm{~d}t + \frac{\beta}{4}\int_0^{t_n}\|U^{\delta} \Upsilon\|_{\L^2}^2 \mathrm{~d}t+ \frac{\beta}{4} \int_0^{t_n}\|(\Pi u)^{\delta} \Upsilon\|_{\L^2}^2 \mathrm{~d}t \\\nonumber&\leq \|U^0-u_0\|_{\L^2} + 2\left(\nu+ \frac{C_K\eta^2}{\nu} \right) \int_{0}^{t_n} \|\nabla\rho\|^2_{\L^2} \mathrm{~d}t 
+ \int_0^{t_n}C(\alpha,\nu)\left(\|U\|^{\frac{8\delta}{4-d}}_{\L^{4\delta}}+\|\Pi U\|^{\frac{8\delta}{4-d}}_{\L^{4\delta}}\right)\|\Upsilon\|_{\L^2}^2\mathrm{~d}t\\&\quad + C(\beta,\gamma, \delta) \int_{0}^{t_n}\|\Upsilon\|_{\L^2} \mathrm{~d}t+ \int_0^{t_n} \left( \|\nabla \rho\|^2_{\L^2}+ \|\Upsilon\|^2_{\L^2}\right) \mathrm{~d}t.
	\end{align}
An application of Gr\"onwall's inequality, yields the desired result
\begin{align}
	\nonumber \|\Upsilon^n\|_{\L^2} ^2 + \nu \int_{0}^{t_n}\|\nabla \Upsilon\|^2_{\L^2}\mathrm{~d}t&\leq  C\bigg(\|U^0-u_0\|_{\L^2}   + 2\left(\nu+ \frac{C_K\eta^2}{\nu} \right) \int_{0}^{t_n} \|\nabla\rho\|^2_{\L^2} \mathrm{~d}t\bigg)\times e^{\left(C(\beta,\gamma,\delta)\right)T}\\&\quad\times\exp\bigg\{\int_0^{t_n}C(\alpha,\nu)\left(\|U(t)\|^{\frac{8\delta}{4-d}}_{\L^{4\delta}}+\|u(t)\|^{\frac{8\delta}{4-d}}_{\L^{4\delta}}\right)\mathrm{~d}t\bigg\},
\end{align}
which yields the required result.
\end{proof}
Our first main result for the error bounds in the semi-discrete case ($hp$-DG time stepping) is given by:
\begin{theorem}\label{5.th2}
	Let $u$ be the exact solution and let $U$ be the approximated DG solution defined by \eqref{5.DGS}, for $1\leq n \leq N, 0\leq q_j \leq p_j,$ and $u\in \H^{q_j+1}(J_j; \H_0^1),$ we have 
	\begin{align}
		\nonumber&\int_0^{t_n} \|\nabla(U-u)\|_{\L^2}^2 \mathrm{~d}t +\|(U-u)^n\|_{\L^2}^2 \leq  C \left(\|U^0-u_0\|_{\L^2}^2 + \int_{0}^{t_n}\|\nabla\rho\|_{\L^2}^2\mathrm{~d}t \right).
	\end{align}
\end{theorem}
\begin{proof}
Using the decomposition $U - u$ defined in \eqref{5.10}, the triangle inequality, Lemma \ref{5.l.4}, and the fact that $\rho^n = 0$ for $1 \leq n \leq N$, the result follows.
\end{proof}
\begin{remark}
	Given that the kernel $ K(\cdot)$ is not necessarily of positive type, but instead satisfies $|K(t-s)| \leq |t-s|^{1-\epsilon} $ for $ 0 \leq \epsilon \leq 1 $, we can avoid the positive restriction on the kernel. Under these conditions, the term $ \mathcal{J}_2^n $ can be estimated using Lemma $\ref{5.l1}$ as follows:
	
	\begin{align}\label{5.16.1}
		|\mathcal{J}_2^n| = \eta \left(K*\nabla \rho,\nabla \Upsilon \right) &\leq \eta \int_{0}^{t_n}\int_{0}^{t}(t-s)^{\alpha-1}\|\nabla \rho(s)\|_{\L^2}\|\nabla \Upsilon(t)\|_{\L^2} \, \mathrm{d}s \, \mathrm{d}t \nonumber\\
		&\leq \eta \left(\int_{0}^{t_n}\left(\int_{0}^{t}(t-s)^{\alpha-1}\|\nabla\rho(s)\|_{\L^2}\, \mathrm{d}s\right)^2 \, \mathrm{d}t\right)^{\frac{1}{2}} \left(\int_0^{t_n} \|\nabla \Upsilon(t)\|^2_{\L^2} \, \mathrm{d}t\right)^{\frac{1}{2}} \nonumber\\
		&\leq \frac{\eta^2}{\nu} \int_{0}^{t_n}\left(\int_{0}^{t}(t-s)^{\alpha-1}\|\nabla\rho(s)\|_{\L^2}\, \mathrm{d}s\right)^2 \, \mathrm{d}t + \frac{\nu}{4}\int_0^{t_n} \|\nabla \Upsilon(t)\|^2_{\L^2} \, \mathrm{d}t \nonumber\\
		&\leq \frac{\eta^2 t_n^{\alpha}}{\nu\alpha} \int_{0}^{t_n}(t_n-t)^{\alpha-1}\int_{0}^t \|\nabla\rho(s)\|^2_{\L^2}\, \mathrm{d}s \, \mathrm{d}t + \frac{\nu}{4}\int_0^{t_n} \|\nabla \Upsilon(t)\|^2_{\L^2} \, \mathrm{d}t \nonumber\\
		&\leq \frac{\eta^2 t_n^{2\alpha}}{\nu\alpha^2} \int_{0}^{t_n} \|\nabla \rho(t)\|^2_{\L^2} \, \mathrm{d}t + \frac{\nu}{4}\int_0^{t_n} \|\nabla \Upsilon(t)\|^2_{\L^2} \, \mathrm{d}t.
	\end{align}
	
	Similarly, we can estimate:
	\begin{align}\label{5.17}
		\eta \left(K*\nabla \Upsilon,\nabla \Upsilon \right) &\leq \frac{\eta^2 t_n^{\alpha}}{\nu\alpha} \int_{0}^{t_n}(t_n-t)^{\alpha-1}\int_{0}^t \|\nabla \Upsilon(s)\|^2_{\L^2} \, \mathrm{d}s \, \mathrm{d}t + \frac{\nu}{4}\int_0^{t_n} \|\nabla \Upsilon(t)\|^2_{\L^2} \, \mathrm{d}t.
	\end{align}
	The final result follows by applying the discrete Gronwall inequality (Lemma \ref{5.l2}), provided the condition
	$$
	\frac{2\eta^2 T^{\alpha}}{\nu\alpha^2}k^{\alpha} < 1,
	$$
	is satisfied. It is important to note that this condition is independent of the polynomial degrees $p_n $.
\end{remark}
We will now establish the subsequent bound as follows:
\begin{lemma}\label{5.l2.8}
	For $1\leq n\leq N,$ we achieve
	\begin{align}\label{5..48}
	\nonumber\int_{t_{n-1}}^{t_n}\|\Upsilon'\|^2_{\L^2}(t-t_{n-1})\mathrm{~d}t &\leq C p_n^2\Big(\|U^0-u_0\|_{\L^2} + \int_{0}^{t_n} \|\nabla\rho\|^2_{\L^2} \mathrm{~d}t   + 2\alpha\int_0^{t_n} \big(B(U)-B(u),(t-t_{n-1})\Upsilon'\big) \mathrm{~d}t \\&\quad +2\beta \int_0^{t_n}\big(c(U)-c(u), (t-t_{n-1})\Upsilon'\big) \mathrm{~d}t\Big).
	\end{align}
\end{lemma}
\begin{proof}
	We choose $X = (t-t_{n-1})\Upsilon' \in \mathbb{P}_{p_n} $ on $J_n$ and zero elsewhere in \eqref{5..31}, we have 
	\begin{align}\label{5.49}
\nonumber&\int_{t_{n-1}}^{t_n}\left[\|\Upsilon'\|^2_{\L^2}(t-t_{n-1}) + \nu a(\Upsilon,(t-t_{n-1})\Upsilon')  + \eta \left(K*\nabla \Upsilon,(t-t_{n-1})\nabla\Upsilon' \right)   \right] \mathrm{~d}t\\\nonumber&= - \int_{t_{n-1}}^{t_n}\Big[\nu a(\rho,(t-t_{n-1})\Upsilon') + \eta \left(K*\nabla \rho, (t-t_{n-1})\nabla\Upsilon' \right)\\&\quad + \alpha(B(U)-B(u),(t-t_{n-1})\Upsilon') - \beta (c(U)-c(u),(t-t_{n-1})\Upsilon') \Big] \mathrm{~d}t.
\end{align}
Using integration by parts, we attain
\begin{align}\label{5.50}
\nonumber	\int_{t_{n-1}}^{t_n}a(\Upsilon,(t-t_{n-1})\Upsilon')\mathrm{~d}t &= \frac{1}{2}\int_{t_{n-1}}^{t_n}(t-t_{n-1})\frac{d}{dt} a(\Upsilon,\Upsilon)\mathrm{~d}t\\& = \frac{k_n}{2}\|\nabla \Upsilon^n\|_{\L^2} - \frac{1}{2}	\int_{t_{n-1}}^{t_n}\|\nabla \Upsilon\|_{\L^2} \mathrm{~d}t.
\end{align}
Substituting the result \eqref{5.50} back to \eqref{5.49}, we get
\begin{align}\label{5.51}
		\int_{t_{n-1}}^{t_n}\|\Upsilon'\|^2_{\L^2}(t-t_{n-1})\mathrm{~d}t & \leq \frac{1}{2}	\int_{t_{n-1}}^{t_n}\|\nabla \Upsilon\|_{\L^2} \mathrm{~d}t + \sum_{i=8}^{10} \mathcal{J}^n_i + 2\alpha\int_0^{t_n} \big(B(U)-B(u),(t-t_{n-1})\Upsilon'\big) \mathrm{~d}t \\&\quad +2\beta \int_0^{t_n}\big(c(U)-c(u), (t-t_{n-1})\Upsilon'\big) \mathrm{~d}t,
	\end{align}
where
\begin{align}\label{5.52}
	\no&\mathcal{J}^n_8 = -\nu\int_{t_{n-1}}^{t_n}a(\rho,(t-t_{n-1})\Upsilon')\mathrm{~d}t,\qquad \quad\qquad \mathcal{J}^n_{9} = -\eta\int_{t_{n-1}}^{t_n}\left(K*\nabla \rho, (t-t_{n-1})\nabla\Upsilon' \right)\mathrm{~d}t,\\& \mathcal{J}^n_{10} = -\eta \int_{t_{n-1}}^{t_n}\left(K*\nabla \Upsilon, (t-t_{n-1})\nabla\Upsilon' \right)\mathrm{~d}t. 
\end{align}
Estimating $\mathcal{J}^n_8$ using Cauchy-Schwarz inequality and the inverse estimate as
\begin{align}\label{5.53}
\no|\mathcal{J}^n_8| &\leq \nu\int_{t_{n-1}}^{t_n}\|\nabla\rho\|_{\L^2}(t-t_{n-1})\|\nabla\Upsilon'\|_{\L^2}\mathrm{~d}t \\&\leq \frac{\nu^2k_n^2p_n^{-2}}{2}\int_{t_{n-1}}^{t_n}\|\nabla\Upsilon'\|^2_{\L^2} \mathrm{~d}t + \frac{p_n^2}{2} \int_{t_{n-1}}^{t_n}\|\nabla\rho\|^2_{\L^2}\mathrm{~d}t\no\\&\leq \frac{\nu^2p_n^{2}}{2}\int_{t_{n-1}}^{t_n}\|\nabla\Upsilon\|^2_{\L^2} \mathrm{~d}t + \frac{p_n^2}{2} \int_{t_{n-1}}^{t_n}\|\nabla\rho\|^2_{\L^2}\mathrm{~d}t.
\end{align}
To estimate $|\mathcal{J}^n_{9}|$ using the Lemma \ref{5.l11} for $T = t_n$ and the inverse inequality as follows
\begin{align}\label{5.54}
	|\mathcal{J}^n_{9}| &\leq \eta\int_{t_{n-1}}^{t_n} \|K*\nabla \rho(s)\|_{\L^2} \mathrm{~d}s ~k_n\|\nabla\Upsilon'\|_{\L^2} \mathrm{~d}t\no
	\\& \leq \left(\int_0^{t_n}\left(\int_0^tK(t-s)\|\nabla \rho(s)\|_{\L^2}\mathrm{~d}s~\right)^2\mathrm{~d}t\right)^{\frac{1}{2}}\left(k_n^2\int_{t_{n-1}}^{t_n} \|\nabla\Upsilon'\|_{\L^2}^2 \mathrm{~d}t\right)^{\frac{1}{2}}\no\\&\leq C_K^{\frac{1}{2}}\left(\int_0^{t_n}\|\nabla\rho(t)\|_{\L^2}^2 \mathrm{~d}t \right)^{\frac{1}{2}}\left(p_n^4\int_{t_{n-1}}^{t_n}\|\nabla \Upsilon\|_{\L^2}^2 \mathrm{~d}t\right)^{\frac{1}{2}}\no\\&\leq C_K^2p_n^2 \int_0^{t_n}\|\nabla\rho(s)\|_{\L^2}^2\mathrm{~d}s + \frac{p_n^2}{2}\int_{t_{n-1}}^{t_n}\|\nabla \Upsilon\|_{\L^2}^2 \mathrm{~d}t.
\end{align}

Similarly, we have 
\begin{align}\label{5.55}
	|\mathcal{J}^n_{10}| \leq C_K^{\frac{1}{2}}\left(\int_0^{t_n}\|\nabla\Upsilon(s)\|_{\L^2}^2 \mathrm{~d}t \right)^{\frac{1}{2}}\left(p_n^4\int_{t_{n-1}}^{t_n}\|\nabla \Upsilon\|_{\L^2}^2 \mathrm{~d}t\right)^{\frac{1}{2}}\leq C_K^{\frac{1}{2}}p_n^2 \int_0^{t_n}\|\nabla\Upsilon\|_{\L^2}^2\mathrm{~d}t.
\end{align}
By combining the estimates \eqref{5.50} through \eqref{5.54} and substituting them into \eqref{5.49}, the desired estimate \eqref{5..48} follows directly.
\end{proof}
It remains to bound the non-linear terms, which can be estimated as follows:
\begin{lemma}\label{5.l2.9} 
	There holds:
	\begin{align}
		&\int_{t_{n-1}}^{t_n} \Big[\alpha(B(U)-B(u),(t-t_{n-1})\Upsilon') - \beta (c(U)-c(u),(t-t_{n-1})\Upsilon') \Big] \\&\leq Cp_n^2\bigg(\int_{t_{n-1}}^{t_n}\|\nabla \rho\|_{\L^2}^2\mathrm{~d}t +\int_{t_{n-1}}^{t_n}\|\nabla e\|_{\L^2}^2\mathrm{~d}t \bigg).
	\end{align}
\end{lemma}
\begin{proof}
	The decomposition \eqref{5.10} we have $e = \Upsilon + \rho$ can be rewritten as $\Upsilon = e-\rho$ or  more preciely, we are going to use $(t-t_{n-1})\Upsilon' = (t-t_{n-1})e'-(t-t_{n-1})\rho'$ as follow
	\begin{align}\label{5.57.1}
		&\int_{t_{n-1}}^{t_n} \alpha(B(U)-B(u),(t-t_{n-1})\Upsilon') - \beta (c(U)-c(u),(t-t_{n-1})\Upsilon') \Big] \mathrm{~d}t\no\\&=   \alpha\int_{t_{n-1}}^{t_n}(B(U)-B(u),(t-t_{n-1})e') \mathrm{~d}t  - \alpha\int_{t_{n-1}}^{t_n}(B(U)-B(u),(t-t_{n-1})\rho') \mathrm{~d}t\no\\&\quad- \beta\int_{t_{n-1}}^{t_n}\big(c(U)-c( u),(t-t_{n-1}) e'\big) \mathrm{~d}t + \beta\int_{t_{n-1}}^{t_n} \big(c(U)-c(u),(t-t_{n-1})\rho'\big) \mathrm{~d}t.
	\end{align}

	\begin{align}\label{5.58.1}
		|\mathcal{J}^n_{11}|= 	&\bigg|\int_{t_{n-1}}^{t_n}\alpha(B(U)-B( u),(t-t_{n-1}) e')\mathrm{~d}t\bigg|\\&\leq \frac{\nu p_n^{-2} }{4}\int_{t_{n-1}}^{t_n}(t-t_{n-1})\|\nabla e'\|_{\L^2}^2\mathrm{~d}t+C(\alpha,\nu)p_n^2\int_{t_{n-1}}^{t_n}\left(\|U\|^{\frac{8\delta}{4-d}}_{\L^{4\delta}}+\|u\|^{\frac{8\delta}{4-d}}_{\L^{4\delta}}\right)\|e\|_{\L^2}^2\mathrm{~d}t\\&\leq \frac{\nu p_n^{2} }{4}\int_{t_{n-1}}^{t_n}\|\nabla e\|_{\L^2}^2\mathrm{~d}t+C(\alpha,\nu)p_n^2\int_{t_{n-1}}^{t_n}\left(\|U\|^{\frac{8\delta}{4-d}}_{\L^{4\delta}}+\|u\|^{\frac{8\delta}{4-d}}_{\L^{4\delta}}\right)\|e\|_{\L^2}^2\mathrm{~d}t,
	\end{align}
	for $C(\alpha, \nu) = \left(\frac{4+d}{\nu}\right)^{\frac{4+d}{4-d}}\left(\frac{4-d}{8}\right)(2^{\delta-1}\alpha)^{\frac{8}{4-d}}$. On the similar lines we get 
	\begin{align}\label{5.59.1}
		|\mathcal{J}^n_{12}|=& \bigg|\int_{t_{n-1}}^{t_n}\alpha(B(U)-B( u),(t-t_{n-1}) \rho')\mathrm{~d}t\bigg|\\&\leq \frac{\nu p_n^{2} }{4}\int_{t_{n-1}}^{t_n}\|\nabla \rho\|_{\L^2}^2\mathrm{~d}t+C(\alpha,\nu)p_n^2\int_{t_{n-1}}^{t_n}\left(\|U\|^{\frac{8\delta}{4-d}}_{\L^{4\delta}}+\|u\|^{\frac{8\delta}{4-d}}_{\L^{4\delta}}\right)\|e\|_{\L^2}^2\mathrm{~d}t.
	\end{align}
	It can be easily seen that the term $\beta(c(U)-c(u),(t-t_{n-1})e')$ in \eqref{5.57.1} can be expended as
	\begin{align}\label{5.60.1}
		\no	\mathcal{J}^n_{13}= \beta(c(U)-c(u),(t-t_{n-1})e')&=-\beta\gamma(e,(t-t_{n-1})e')-\beta(U^{2\delta+1}-(u)^{2\delta+1},(t-t_{n-1})e')\\&\quad+\beta(1+\gamma)(U^{1+\delta}-(u)^{\delta+1},(t-t_{n-1})e').
	\end{align}
	The first term of \eqref{5.60.1} be estimated as 
	\begin{align}\label{5.61.1}
		&-\beta\gamma\int_{t_{n-1}}^{t_n}(e,(t-t_{n-1})e') \mathrm{~d}t= -\frac{\beta\gamma}{2}\int_{t_{n-1}}^{t_n}  (t-t_{n-1})\frac{d}{dt}\|e\|^2 \mathrm{~d}t= -k_n\beta\gamma\|e^n\|^2 +\beta \gamma \int_{t_{n-1}}^{t_n}\|e\|^2\mathrm{~d}t.
	\end{align}
	To estimate the term $-\beta(U^{2\delta+1}-u^{2\delta+1},(t-t_{n-1})e')$ we employ Taylor's expansion, H\"older's inequality, Gagliardo-Nirenberg inequalities, Young's inequalities followed by inverse estimate as 
	\begin{align}\label{5.62.1}
	&\big|(2\delta+1)\beta\int_{0}^{t_n}\left((\Upsilon U+(1-\Upsilon)u)^{2\delta}e,(t-t_{n-1})e'\right)\mathrm{~d}t\big|\nonumber\\&\leq 2^{2\delta-1}(2\delta+1)\beta\int_{0}^{t_n}\left(( |U|^{2\delta}+|u|^{2\delta})e,(t-t_{n-1})e'\right)\mathrm{~d}t\nonumber\\&\leq  2^{2\delta-1}(2\delta+1)\beta\int_{0}^{t_n}\left(\|U\|_{\L^{4\delta}}^{2\delta}+\|u\|_{\L^{4\delta}}^{2\delta}\right)\|e\|_{\L^{\frac{2d}{d-1}}}\|(t-t_{n-1})e'\|_{\L^{2d}}\mathrm{~d}t\nonumber\\&\leq  2^{2\delta-1}(2\delta+1)\beta\int_{0}^{t_n}\left(\|U\|_{\L^{4\delta}}^{2\delta}+\|u\|_{\L^{4\delta}}^{2\delta}\right)\|e\|_{\L^2}^{\frac{1}{2}}\|\nabla e\|_{\L^2}^{\frac{1}{2}}\|(t-t_{n-1})e'\|_{\L^{2d}}\mathrm{~d}t	\nonumber\\&\leq 2^{2\delta-1}(2\delta+1)\beta p_n^{-2}k_n^2\int_{0}^{t_n}\|\nabla e'\|_{\L^2}^2 \mathrm{~d}t+ \frac{\nu}{4}\int_{0}^{t_n}\|\nabla e\|_{\L^2}^2 \mathrm{~d}t + \frac{2^{4\delta-2}(2\delta+1)^2\beta^2}{\nu}\int_{0}^{t_n}\left(\|U\|_{\L^{4\delta}}^{8\delta}+\|u\|_{\L^{4\delta}}^{8\delta}\right)\|e\|_{\L^2}^2\mathrm{~d}t\no\\&\leq \left(2^{2\delta-2}(2\delta+1)\beta p_n^{2}+\frac{\nu}{4}\right)\int_{t_{n-1}}^{t_n}\|\nabla e\|_{\L^2}^2 + \frac{2^{4\delta-4}(2\delta+1)^2\beta^2p_n^2}{\nu}\int_{t_{n-1}}^{t_n}\left(\|U\|_{\L^{4\delta}}^{8\delta}+\|u\|_{\L^{4\delta}}^{8\delta}\right)\|e\|_{\L^2}^2\mathrm{~d}t.
\end{align}
	Estimating $\beta(1+\gamma)(U^{\delta+1}-u^{\delta+1},(t-t_{n-1})e')$ using Taylor's formula, inverse estimate, H\"older's and Young's inequalities  as
	\begin{align}\label{5.63.1}
		&\no	\beta(1+\gamma)\int_{t_{n-1}}^{t_n}(U^{\delta+1}-u^{\delta+1},(t-t_{n-1})e')\mathrm{~d}t\\&
		\leq \frac{\nu p_n^{2} }{4}\int_{t_{n-1}}^{t_n}\|\nabla e\|_{\L^2}^2\mathrm{~d}t+C(\alpha,\beta,\gamma)p_n^2\int_{t_{n-1}}^{t_n}\left(\|U\|^{\frac{8\delta}{4-d}}_{\L^{4\delta}}+\|u\|^{\frac{8\delta}{4-d}}_{\L^{4\delta}}\right)\|e\|_{\L^2}^2\mathrm{~d}t.
	\end{align}
	Substituting \eqref{5.61.1}-\eqref{5.63.1} into \eqref{5.60.1}, we attain
	\begin{align}\label{5.64}
		\no \mathcal{J}^n_{13}=&\beta(c(U)-c(u),(t-t_{n-1})e')+k_n\beta\gamma\|e^n\|^2\\&\leq C p_n^2\int_{t_{n-1}}^{t_n}\|\nabla e\|_{\L^2}^2\mathrm{~d}t + C(\alpha,\beta,\gamma)p_n^2\int_{t_{n-1}}^{t_n}\left(\|U\|^{\frac{8\delta}{4-d}}_{\L^{4\delta}}+\|u\|^{\frac{8\delta}{4-d}}_{\L^{4\delta}}\right)\|e\|_{\L^2}^2\mathrm{~d}t \no \\&\quad + \frac{2^{4\delta-4}(2\delta+1)^2\beta^2}{\nu}\int_{t_{n-1}}^{t_n}\left(\|U\|_{\L^{4\delta}}^{8\delta}+\|u\|_{\L^{4\delta}}^{8\delta}\right)\|e\|_{\L^2}^2\mathrm{~d}t.
	\end{align}
using the inverse estimate, Poincar\'e  inequality and Young's inequalities
	\begin{align}\label{5.61.2}
\beta\gamma\int_{t_{n-1}}^{t_n}(e,(t-t_{n-1})\rho') \mathrm{~d}t&\leq \beta\gamma\int_{t_{n-1}}^{t_n}\|e\|_{\L^2}\|(t-t_{n-1})\rho'\|_{\L^2} \mathrm{~d}t\\&\leq   \frac{k_n^2p_n^{-2}}{2}\int_{t_{n-1}}^{t_n}\|\rho'\|^2_{\L^2} \mathrm{~d}t + \beta^2\gamma^2p_n^2\int_{t_{n-1}}^{t_n}\|e\|^2_{\L^2} \mathrm{~d}t\\&\leq C\left(\frac{p_n^2}{2}\int_{t_{n-1}}^{t_n}\|\nabla \rho\|^2_{\L^2} \mathrm{~d}t + \beta^2\gamma^2p_n^2\int_{t_{n-1}}^{t_n}\|\nabla e\|^2_{\L^2} \mathrm{~d}t\right).
\end{align}
	Using the analogues results as in \eqref{5.64}, we have 
	\begin{align}\label{5.65}
		\no \mathcal{J}^n_{13}=&\beta(c(U)-c(u),(t-t_{n-1})\rho')\\&\leq C p_n^2\int_{t_{n-1}}^{t_n}\|\nabla \rho\|_{\L^2}^2\mathrm{~d}t + C(\alpha,\beta,\gamma)p_n^2\int_{t_{n-1}}^{t_n}\left(\|U\|^{\frac{8\delta}{4-d}}_{\L^{4\delta}}+\|u\|^{\frac{8\delta}{4-d}}_{\L^{4\delta}}\right)\|e\|_{\L^2}^2\mathrm{~d}t \no \\&\quad + \frac{2^{4\delta-4}(2\delta+1)^2\beta^2p_n^2}{\nu}\int_{t_{n-1}}^{t_n}\left(\|U\|_{\L^{4\delta}}^{8\delta}+\|u\|_{\L^{4\delta}}^{8\delta}\right)\|e\|_{\L^2}^2\mathrm{~d}t.
	\end{align}
	Finally from the estmates of $\mathcal{J}^n_i ~(i = 10,\cdots,13)$, we obtain 
	\begin{align}
		&\int_{t_{n-1}}^{t_n} \Big[\alpha(B(U)-B(u),(t-t_{n-1})\Upsilon') - \beta (c(U)-c(u),(t-t_{n-1})\Upsilon') \Big] + k_n\beta\gamma\|e^n\|^2 \mathrm{~d}t \\&\leq Cp_n^2\bigg(\int_{t_{n-1}}^{t_n}\|\nabla \rho\|_{\L^2}^2\mathrm{~d}t +\int_{t_{n-1}}^{t_n}\|\nabla e\|_{\L^2}^2\mathrm{~d}t  +  \int_{t_{n-1}}^{t_n}\left(\|U\|^{\frac{8\delta}{4-d}}_{\L^{4\delta}}+\|u\|^{\frac{8\delta}{4-d}}_{\L^{4\delta}}\right)\|e\|_{\L^2}^2\mathrm{~d}t \no \\&\quad + \int_{t_{n-1}}^{t_n}\left(\|U\|_{\L^{4\delta}}^{8\delta}+\|u\|_{\L^{4\delta}}^{8\delta}\right)\|e\|_{\L^2}^2\mathrm{~d}t\bigg).
	\end{align}
	By utilizing the regularity of the solutions $u$ and $U$ along with the Poincar\'e  inequality, we can derive the desired result.
\end{proof}

Further we are getting the following bound
\begin{lemma}\label{5.l2.13}
	For $1\leq n\leq N,$ we achieve
	\begin{align}\label{5.100}
	\int_{t_{n-1}}^{t_n}\|\Upsilon'\|^2_{\L^2}(t-t_{n-1})\mathrm{~d}t &\leq C p_n^2\Big(\|U^0-u_0\|_{\L^2} + \int_{0}^{t_n} \|\nabla\rho\|^2_{\L^2} \mathrm{~d}t\Bigg).
	\end{align}
\end{lemma}
\begin{proof}
	By an application of Lemma \ref{5.l2.9} in Lemma \ref{5.l2.8}, we obtain:
\begin{align}\label{5.101}
	\int_{t_{n-1}}^{t_n}\|\Upsilon'\|^2_{\L^2}(t-t_{n-1})\mathrm{~d}t &\leq C p_n^2\Big(\|U^0-u_0\|_{\L^2} + \int_{0}^{t_n} \|\nabla\rho\|^2_{\L^2} \mathrm{~d}t  +\int_{t_{n-1}}^{t_n}\|\nabla e\|_{\L^2}^2\mathrm{~d}t\Big).
\end{align}
	Finally, using the estimate for $\int_{t_{n-1}}^{t_n}\|\nabla e\|_{\L^2}^2\mathrm{~d}t$ from Lemma \ref{5.l.4}, we conclude the proof.
\end{proof}
Let us now derive the following bound for $\Upsilon = U-\Pi u$ as
\begin{theorem}\label{5.lem12}
	For $1\leq n\leq N$, we have 
	\begin{align}\label{5.supnorm}
		\|\Upsilon\|_{I_n}^2\leq C \log(|\textbf{p}|_n+2)|\textbf{p}|_n^2\left(\|U^0-u_0\|_{\L^2}^2 + \int_{0}^{t_n}\|\nabla\rho\|_{\L^2}^2\mathrm{~d}t \right),
	\end{align}
and consequently the following estimate holds
\begin{align}
	\|U-u\|_{I_n}^2 \leq C \|u-\Pi u\|_{I_n}^2 + C \log(|\textbf{p}|+2)|\textbf{p}|_n^2\sum_{n=1}^N \hat{p_j}^{-2}\left(\frac{k_j}{2}\right)^{2q_j+2} {\Gamma_{p_j, q_j}} \int_{t_{j-1}}^{t_j}\left\|u^{\left(q_j+1\right)}\right\|_{\H_0^1}^2 \mathrm{~d}t,
\end{align}
where the norm $\H^1_{I_n}$, for $I_n = (0,t_n]$ is defined as $\|\phi\|_{I_n} = \sup\limits_{t\in I_n}\|\phi(t)\|_{\L^2}$ and $
|\textbf{p}|_n := \max \left\{ \max\limits_{j=1}^n p_j, 1 \right\}.
$
\end{theorem} 
\begin{proof}
	The first inequality follows using the inverse inequality stated in Lemma \ref{5.leminv} and the results derived in Lemma \ref{5.l.4} and Lemma \ref{5.lem12}, we obtain for $1\leq j\leq n\leq N$, 
	\begin{align}
		\|\Upsilon\|_{I_j}^2& \leq C \left(\log(p_j+2)\int_{t_{j-1}}^{t_j}\|\Upsilon'\|_{\L^2}(t-t_{j-1}) \mathrm{~d}t + \|\Upsilon^j\|_{\L^2}\right)\\&\leq C\log(p_j+2) \left(p_j^2\|U^0-u_0\|_{\L^2}^2 + p_j^2 \int_{0}^{t_j}\|\nabla\rho\|_{\L^2}^2\mathrm{~d}t\right)\\&\leq C \log(|\textbf{p}|_n+2)|\textbf{p}|_n^2\left(\|U^0-u_0\|_{\L^2}^2 + \int_{0}^{t_j}\|\nabla\rho\|_{\L^2}^2\mathrm{~d}t \right).
	\end{align} 
	Note that the right-hand side is independent of the time level $j$. Therefore, the desired result follows in the norm defined in Theorem \eqref{5.lem12}.
	Further the second estimate follows using triangle inequality and the estimate proved in \eqref{5.supnorm}.
\end{proof}
Let us now combine the results obtained above to derive explicit $hp$-version error estimates in terms of the step-size $k_j$, the polynomial degree $p_j$, and the regularity parameter $q_j$.
 \begin{theorem}\label{5..th2}
 	Let $u$ be the exact solution and let $U$ be the approximated solution defined by \eqref{5.DGS}, for $1\leq n \leq N, 0\leq q_j \leq p_j,$ and $u\in \H^{q_j+1}(J_j; \H_0^1),$ we have 
 	\begin{align}
 		\nonumber&\int_0^{t_n} \|\nabla(U-u)\|_{\L^2}^2 \mathrm{~d}t +\|(U-u)^n\|_{\L^2}^2 \leq C \sum_{j=1}^n \hat{p}_j^{-2}\left(\frac{k_j}{2}\right)^{2q_j+2} {\Gamma_{p_j, q_j}} \int_{t_{j-1}}^{t_j}\left\|\nabla u^{\left(q_j+1\right)}\right\|_{\L^2}^2 \mathrm{~d}t.
 	\end{align}
 \begin{align}
 	\|U-u\|_{I_n}^2 &\leq C \max_{j=1}^n\left(\frac{k_j}{2}\right)^{2q_j+1}{\Gamma_{p_j, q_j}} \int_{t_{j-1}}^{t_j}\left\|u^{\left(q_j+1\right)}\right\|_{\H_0^1}^2 \mathrm{~d}t \\&\quad+ C \log(|\textbf{p}|_n+2)|\textbf{p}|_n^2\sum_{j=1}^n \hat{p}_j^{-2}\left(\frac{k_j}{2}\right)^{2q_j+2} {\Gamma_{p_j, q_j}} \int_{t_{j-1}}^{t_j}\left\|\nabla u^{\left(q_j+1\right)}\right\|_{\L^2}^2 \mathrm{~d}t,
 \end{align}
where we define $\hat{p}_j \coloneqq \max\{1,p_j\}.$
 \end{theorem}
 \begin{proof}
The results are obtained immediately from Theorems \ref{5.th2} and \ref{5.lem12} together with the interpolation estimate in Theorem \ref{5.th1}. Additionally, the second estimate utilizes the Poincar\'e inequality to derive the result.
 \end{proof}
For the uniform parameters $k, p,$ and $q$ (i.e., $k_j=k, p_j=p, q_j=q$), the bounds in the above Theorem \ref{5.th2} results in the following error estimates.
\begin{corollary}\label{5.co1}
	For $1\leq n \leq N, 0\leq q_j \leq p_j,$ and $u\in \H^{q_j+1}(J_j; \H_0^1),$ we have the error bounds 
	\begin{align}
		\nonumber&\int_0^{t_n} \|\nabla(U-u)\|_{\L^2}^2 \mathrm{~d}t +\|(U-u)^n\|_{\L^2}^2 \leq  C\frac{k^{2\min\{p,q\}+2}}{p^{2q+2}} \int_{0}^{t_n}\|\nabla u^{(q+1)}\|^2_{\L^2} \mathrm{~d}t.
	\end{align}
and 
 \begin{align}
	\|U-u\|_{I_n}^2 \leq C \frac{k^{2\min\{p,q\}+2}}{p^{2q}}\left(\max_{j=1}^n\max_{t\in J_n}\|\nabla u^{\left(q_j+1\right)}\|_{\L^2}^2 + \log(p+2) \int_{0}^{t_n}\|\nabla u^{(q+1)}\|^2_{\L^2} \mathrm{~d}t\right).
\end{align}
\end{corollary}
\begin{proof}
	This follows directly from Theorem \ref{5.th2} and the fact that $\Gamma_{p,q} \sim p^{-2q}$ for $p \rightarrow \infty$ using Stirling's formula or Jordan's Lemma \cite{CSc}.
\end{proof}
Corollary \ref{5.co1} indicates that the DG time-stepping scheme converges either by decreasing the time step size $k$ (i.e., $k \rightarrow 0$) or by increasing the polynomial degree $p$ (i.e., $p \rightarrow \infty$). The convergence rate is optimal with respect to both $k$ and $p$ in the first case, while the second case is suboptimal by one power of $p$. For large values of $q$, it is more beneficial to increase $p$ while keeping $k$ fixed (the $p$-version of the DG method) rather than reducing $k$ with $p$ fixed (the $h$-version of the DG method). For smooth solutions $u$, arbitrarily high order convergence rates can be attained by increasing $p$, which is known as spectral convergence. Specifically, if $u$ is analytic on $[0, t_n]$ and belongs to $\H_0^1(\Omega)$, exponential convergence rates can be demonstrated for the $p$-version of the method with a fixed step size $k$:
\begin{align}
	\int_0^{t_n} \|\nabla(U-u)\|_{\L^2}^2 + \|U-u\|_{I_n}^2 \leq C \exp\left(-\tilde{b} p_n\right),
\end{align}
which directly follows from the first approximation result in Theorem \ref{5.th1}.

	\section{Fully-discrete scheme}\label{5.sec3}\setcounter{equation}{0}
	\subsection{$hp$-FEM in space}\label{5.sub2}
	In this section, we will discuss the fully-discrete scheme and the error estimate for the model GHBE \eqref{5.GBHE}. We will employ $hp$-DG time stepping combined with a  $hp$-finite element (conforming) discretization in space.
	Let $\mathcal{T}$ be a quasi uniform triangulation of $\Omega$, where $h=\max\limits_{K \in \mathcal{T}}(\operatorname{diam} K).$ To the triangulation $\mathcal{T}$ of $\overline{\Omega}$, we associate a finite dimensional space $\mathcal{V}_{\textbf{h}}^{\textbf{r}}$ of continuous piecewise  polynomials of degree $r_K$ in each $K\in \mathcal{T}$  stored in  $\textbf{r}= \{r_K \colon {K\in \mathcal{T}}\}$ and set $|\textbf{r}|  = \max\limits_{K\in \mathcal{T}} r_K.$ It is assumed that  that $\textbf{r}$ is of bounder variation meaning there exists a constant $\zeta\geq 1$ independent of the mesh, such that for any adjacent elements $K,K'\in \mathcal{T}$, we have $\zeta^{-1}\leq \frac{r_K}{r_{K'}}\leq \zeta$.  Similarly,  $h_K$ denotes the diameter of $K$ and $\textbf{h}  =\{h_K\colon K\in\mathcal{T}\}$ stores the mesh sizes corresponding to different elements. The discrete space  is given as
		$$
		\mathcal{V}_{\textbf{h}}^{\textbf{r}}=\left\{v_{h}: v_{h} \in C^{0}(\bar{\Omega})\cap\H_0^1(\Omega) : \left.v_{h}\right|_{K} \in \P_{r_K}(K) \ \forall \ K \in \mathcal{T}\right\}.
		$$
		For a partition $\mathcal{M}=\{\J_n\}_{n=1}^N $ of $(0,T)$ and a degree vector $\mathbf{p} = (p_1, p_2, \cdots, p_N)$, the trial space is now given by
 \begin{align}\label{5.3.1}
 	\mathcal{S}(\mathcal{M},\textbf{p},\mathcal{V}_{\textbf{h}}^{\textbf{r}})=\{U_h : [0,T] \rightarrow \mathcal{V}_{\textbf{h}}^{\textbf{r}} : {U_h}|_{\J_n} \in \mathbb{P}_{p_n}(\mathcal{V}_{\textbf{h}}^{\textbf{r}}), 1 \leq n \leq N\},
 \end{align}
where $\mathbb{P}_{p}(\mathcal{V}_{\textbf{h}}^{\textbf{r}})$ is defined as the space of polynomials of degree $\leq p$ in the time variable, where the coefficients are drawn from $\mathcal{V}_{\textbf{h}}^{\textbf{r}}$.  Consequently, functions $U_h(x,t)$ belonging to $\mathcal{S}(\mathcal{M},\textbf{p},\mathcal{V}_h)$ exhibit continuity in the spatial variable $x$, yet may exhibit discontinuities precisely at discrete time points $t = t_n$. 
Using standard finite elements in space and an $hp$-DG time-stepping method, we establish the following fully-discrete $hp$-DG finite element scheme: Find $U_h \in \mathcal{S}(\mathcal{M},\textbf{p},\mathcal{V}_h)$ such that
\begin{align}\label{5.3.2}
\nonumber	F_N(U_h,X) &= \langle U_h^0, X_{+}^0\rangle + \int_0^{t_N}\langle f(t),X(t) \rangle \mathrm{~d}t  \quad \forall X\in  \mathcal{S}(\mathcal{M},\textbf{p}, \mathcal{V}_h ),\\
U_h(0) &= U_h^0,
\end{align}
for a suitable approximation $U_h^0 \in \mathcal{V}_{\textbf{h}}^{\textbf{r}}$ to $u_0$.

Further to discuss the error estimates, we sate the approximation result defined in  \cite[Lemma 4.5]{BSu} as
\red{\begin{lemma}\label{5.lemm3}
	Suppose that a triangulation $\mathcal{T}$ of $\Omega$ is formed by $d$-dimensional simplices or parallelepipeds. Then for an arbitrary $u \in \H^{\mathbf{m}}(\Omega, \mathcal{T})$, $\mathbf{m} = (m_K, K \in \mathcal{T})$, and for each $r = (r_K, K \in \mathcal{T})$, $r_K \in \mathbb{N}$, there exists a projector
	$$
	\mathcal{R}^{\mathbf{h}}_{\mathbf{r}} : \H^{\mathbf{s}}(\Omega, \mathcal{T}) \to 	\mathcal{S}(\mathcal{M},\textbf{p},\mathcal{V}_{\textbf{h,r}}),
	$$
	such that for $0 \leq b \leq m_K$,
\begin{align}\label{5.7a1}
		\| u -\mathcal{R}^{\mathbf{h}}_{\mathbf{r}}  u \|_{\H^b} \leq C \frac{h_K^{s_K - b}}{r_K^{m_K-b}} \| u \|_{\H^{m_K}},
\end{align}
	where $s_K = \min(r_K +1,m_K )$ and $C$ is a constant independent of $u, h_K,$ and $r_K$ but
	dependent on $m = \max\limits_{K\in\mathcal{T}} m_K.$
\end{lemma}}

\subsection{Error estimates}
To derive error estimates for the fully-discrete case \eqref{5.3.2}, we decompose the error similarly to \eqref{5.10} into three distinct components:
\begin{align}\label{5.3.3}
	U_h - u &= \underbrace{(U_h - \Pi \mathcal{R}^{\mathbf{h}}_{\mathbf{r}} u)}_{\psi} + \underbrace{(\Pi \mathcal{R}^{\mathbf{h}}_{\mathbf{r}} u - \Pi u)}_{\Pi \xi } + \underbrace{(\Pi u - u)}_{\rho}, 
\end{align}
where  $\xi =\mathcal{R}^{\mathbf{h}}_{\mathbf{r}}  u - u$ represents the error due to spatial approximation, and $\rho$ signifies the error due to temporal approximation as defined in \eqref{5.10}.
\begin{theorem}\label{5.th3}
	If $u$ is the solution of problem \eqref{5.GBHE}, and $U_h \in \mathcal{V}_{\textbf{h}}^{\textbf{r}}$ is the approximate solution defined by  \eqref{5.3.2}, then 
	\begin{align}\label{5.3.5}
		&\nonumber F_N(U_h, X) - F_N(\Pi \mathcal{R}^{\mathbf{h}}_{\mathbf{r}} u, X)\\& \no=  \langle (U_h^0-\mathcal{R}^{\mathbf{h}}_{\mathbf{r}} u_0), X_{+}^0\rangle - \int_{0}^{t_N}\langle\xi',X\rangle \mathrm{~d}t   - \int_{0}^{t_N}\Big[ \nu a(\rho,X) +\nu a(\Pi\xi,X) + \alpha(B(\Pi \mathcal{R}^{\mathbf{h}}_{\mathbf{r}} u)-B(\Pi u),X)\\&\quad + \alpha(B(\Pi u)-B( u),X)   + \eta \left(K*\nabla  \rho,\nabla X \right)  - \beta (c(\Pi \mathcal{R}^{\mathbf{h}}_{\mathbf{r}}u)-c(\Pi u),X) - \beta (c(\Pi u)-c(u),X) \Big] \mathrm{~d}t, 
	\end{align}
for all $X\in  \mathcal{S}(\mathcal{M},\textbf{p}, \mathcal{V}_h )$, where $F_N(\cdot,\cdot)$ is defined in \eqref{5.AGWF}.
\end{theorem}
\begin{proof}
Using the definition of $F_N(\cdot,\cdot)$, we obtain
	\begin{align}
		F_N(U_h,X) - F_N(u,X) = \langle (U_h^0- u_0), X_{+}^{0} \rangle. 
	\end{align}
Additionally, the decomposition given in \eqref{5.3.3} implies
\begin{align}\label{5.38}
	\no &F_N(U_h,X) - F_N(\Pi \mathcal{R}^{\mathbf{h}}_{\mathbf{r}}u,X)\\ &= \langle (U_h^0- u_0), X_{+}^{0} \rangle - \left(F_N(\Pi \mathcal{R}^{\mathbf{h}}_{\mathbf{r}}u,X) - F_N(\Pi u,X)+	F_N(\Pi u,X) - F_N(u,X)\right).
\end{align}
Since the projection \eqref{5.8} gives $(\Pi \xi)^n = \xi^n$ and $\rho^n = 0$, and with the expression for $F_N$ defined in \eqref{5.AGWF}, we achieve
\begin{align}\label{5.39}
\nonumber &F_N(\Pi \mathcal{R}^{\mathbf{h}}_{\mathbf{r}}u,X)- F_N(\Pi u,X) +	F_N(\Pi u,X) - F_N(u,X)\\&\no= \langle \xi^N , X^{N} \rangle -\sum_{n=1}^{N-1}\langle \xi^n , [X]^n \rangle + \sum_{n=1}^N \int_{t_{n-1}}^{t_n}\Big[-\langle \Pi \xi + \rho,X'\rangle\\\nonumber&\quad  + \nu a(\Pi \xi + \rho,X) + \alpha(B(\Pi \mathcal{R}^{\mathbf{h}}_{\mathbf{r}}u)-B(\Pi u),X) + \alpha(B(\Pi u)-B( u),X)  \\&\quad + \eta \left(K*\nabla (\Pi \xi + \rho),\nabla X \right)  - \beta (c(\Pi \mathcal{R}^{\mathbf{h}}_{\mathbf{r}}u)-c(\Pi u),X) - \beta (c(\Pi u)-c(u),X)\Big] \mathrm{~d}t.
\end{align}
Applying integration by parts to $\int_{t_{n-1}}^{t_n} \langle \Pi \xi, X' \rangle \mathrm{~d}t$, we have the equality
\begin{align}\label{5.40}
	\int_{t_{n-1}}^{t_n}\langle\Pi\xi, X'\rangle \mathrm{~d}t = \int_{t_{n-1}}^{t_n}\langle\xi, X'\rangle \mathrm{~d}t = \langle\xi^n, X^n\rangle-\langle\xi^{n-1}, X^{n-1}_{+}\rangle-\int_{t_{n-1}}^{t_n}\langle\xi', X\rangle \mathrm{~d}t.
\end{align}
 Substituting equation \eqref{5.40} into \eqref{5.39}, we get
\begin{align}\label{5.41}
\nonumber &F_N(\Pi \mathcal{R}^{\mathbf{h}}_{\mathbf{r}}u,X)- F_N(\Pi u,X) +	F_N(\Pi u,X) - F_N(u,X)\\&= \langle \xi^0 , X_+^{0} \rangle - \sum_{n=1}^N \int_{t_{n-1}}^{t_n}\langle \rho,X'\rangle \mathrm{~d}t +\int_{0}^{t_N}\langle\xi', X\rangle \mathrm{~d}t \\\nonumber&\quad  +\int_{0}^{t_N} \bigg[\nu a(\rho,X) +\nu a(\Pi\xi,X) + \alpha(B(\Pi \mathcal{R}^{\mathbf{h}}_{\mathbf{r}}u)-B(\Pi u),X) + \alpha(B(\Pi u)-B( u),X)  \\&\quad + \eta \left(K*\nabla  \rho,\nabla X \right)  - \beta (c(\Pi \mathcal{R}^{\mathbf{h}}_{\mathbf{r}}u)-c(\Pi u),X) - \beta (c(\Pi u)-c(u),X)\bigg] \mathrm{~d}t.
\end{align}
Finally, inserting this expression into \eqref{5.38} and utilizing the fact that $\int_{t_{n-1}}^{t_n} \langle \rho, X' \rangle \mathrm{~d}t = 0,$ concludes the proof.
\end{proof}
 \begin{lemma}\label{5.l5}
	For $1\leq n\leq N,$ we have

 \begin{align}\label{5.42}
	\nonumber &\|\psi^n\|_{\L^2} ^2 + \nu \int_{0}^{t_n}\|\nabla\psi\|^2_{\L^2}\mathrm{~d}t+ \frac{\beta}{4}\int_0^{t_n}\|U_h^{\delta} \Upsilon\|_{\L^2}^2 \mathrm{~d}t+ \frac{\beta}{4} \int_0^{t_n}\|(\Pi \mathcal{R}^{\mathbf{h}}_{\mathbf{r}} u)^{\delta} \Upsilon\|_{\L^2}^2 \mathrm{~d}t \\\nonumber&\leq \|U^0-\mathcal{R}^{\mathbf{h}}_{\mathbf{r}}u_0\|_{\L^2} + 2\left(\nu+ \frac{C_K\eta^2}{\nu} \right) \int_{0}^{t_n} \|\nabla\rho\|^2_{\L^2} \mathrm{~d}t + \Xi \\&\quad
	 + \int_0^{t_n}C(\alpha,\nu)\left(\|U\|^{\frac{8\delta}{4-d}}_{\L^{4\delta}}+\|\Pi \mathcal{R}^{\mathbf{h}}_{\mathbf{r}} U\|^{\frac{8\delta}{4-d}}_{\L^{4\delta}}\right)\|\psi\|_{\L^2}^2\mathrm{~d}t+ C(\beta,\gamma, \delta) \int_{0}^{t_n}\|\psi\|_{\L^2} \mathrm{~d}t+ \frac{1}{2\nu}\int_{0}^{t_n} \|\xi'\|^2_{\L^2} \mathrm{~d}t \\&\quad  -  2\alpha\int_0^{t_n} \big(B(\Pi \mathcal{R}^{\mathbf{h}}_{\mathbf{r}}u)-B(\Pi u),\psi\big) \mathrm{~d}t +2\beta \int_0^{t_n}\big(c(\Pi \mathcal{R}^{\mathbf{h}}_{\mathbf{r}}u)-c(\Pi u),\psi\big) \mathrm{~d}t \\&\quad  -  2\alpha\int_0^{t_n} \big(B(\Pi u)-B(u),\psi\big) \mathrm{~d}t +2\beta \int_0^{t_n}\big(c(\Pi u)-c(u),\psi\big) \mathrm{~d}t,
\end{align}
where $\Xi = \nu\int_0^{t_n}\|\nabla(\Pi \xi -\xi)\| _{\L^2}^2 \mathrm{~d}t + \nu\int_0^{t_n}\|\nabla\xi\| _{\L^2}^2 \mathrm{~d}t +\eta\int_0^{t_n}\|K*\nabla(\Pi \xi -\xi)\| _{\L^2}^2 \mathrm{~d}t + \eta\int_0^{t_n}\|K*\nabla\xi\| _{\L^2}^2 \mathrm{~d}t.$
\end{lemma}
\begin{proof}
 We choose $X= \psi$ in \eqref{5.3.5}, follow the proof similar to Lemma \ref{5.l3} where in place of $a(\rho, \Upsilon)$ in $J_1^n$ we have $\langle \xi',\psi\rangle + a(\rho, \psi) $, and using the inequality $\|\psi\|_{\L^2} \leq \|\nabla \psi\|_{\L^2}$. Also, rewriting the term $a(\Pi\xi,X) = a(\Pi\xi-\xi,X) + a(\xi,X)$, the desired result then readily follows.
\end{proof}
Further aim is estimate $\Big(2\alpha\int_0^{t_n} \big(B(\Pi \xi)-B(\rho),\psi\big) \mathrm{~d}t  +2\beta \int_0^{t_n}\big(c(\Pi \xi)-c(\rho),\psi\big) \mathrm{~d}t\Big)$ using the following lemma
\begin{lemma}\label{5.43}
	For $u\in \L^2(0,T; \L^{4\delta}(\Omega)), 1\leq n \leq N, 0\leq q_j \leq p_j,$ and $u\in \H^{q_j+1}(J_j; \H_0^1),$ we have 
	\begin{align}\label{5.44}
	&\int_{0}^{t_N} \bigg[ \alpha(B(\Pi \mathcal{R}^{\mathbf{h}}_{\mathbf{r}}u)-B(\Pi u),\psi) + \alpha(B(\Pi u)-B( u),\psi)   - \beta (c(\Pi \mathcal{R}^{\mathbf{h}}_{\mathbf{r}}u)-c(\Pi u),\psi) - \beta (c(\Pi u)-c(u),\psi)\bigg] \mathrm{~d}t \\&\leq C \left(\int_0^{t_n} \left( \|\nabla \rho\|^2_{\L^2}+ \|\psi\|^2_{\L^2}  \right) \mathrm{~d}t\right) + \Xi.
\end{align}
\end{lemma}
\begin{proof}
	The proof following similar to the Lemma \ref{5.l4}, using the estimate \eqref{5.28}, we attain
		\begin{align}\label{5.103}
		&\int_{0}^{t_N} \bigg[ \alpha(B(\Pi \mathcal{R}^{\mathbf{h}}_{\mathbf{r}}u)-B(\Pi u),\psi) + \alpha(B(\Pi u)-B( u),\psi)   - \beta (c(\Pi \mathcal{R}^{\mathbf{h}}_{\mathbf{r}}u)-c(\Pi u),\psi) - \beta (c(\Pi u)-c(u),\psi)\bigg] \mathrm{~d}t\\&\leq  \left(\frac{2^{2\delta-1}\alpha^2}{\nu}+\frac{\beta}{2}\right)\int_0^{t_n}\left(\|\Pi \mathcal{R}^{\mathbf{h}}_{\mathbf{r}} u\|^{2\delta}_{\L^{4\delta}}+ \|\Pi u\|^{2\delta}_{\L^{4\delta}} \right)\|\nabla \Pi\xi\|^2_{\L^2} \mathrm{~d}t+ \frac{\nu}{4}\int_0^{t_n}\|\nabla \psi\|^2_{\L^2}\mathrm{~d}t\no\\&\quad+ \left(2^{2\delta-1}\beta(1+\gamma)^2(\delta+1)^2+\frac{\beta\gamma}{2}\right)\int_0^{t_n} \|\psi\|^2_{\L^2}\mathrm{~d}t+\frac{\beta\gamma}{2}\int_0^{t_n}\|\Pi\xi\|^2_{\L^2}\mathrm{~d}t\no\\&\quad+ 2^{4\delta-1}(2\delta+1)^2\beta^2 \int_0^{t_n}\left(\|{\Pi \mathcal{R}^{\mathbf{h}}_{\mathbf{r}}u}\|^{4\delta}_{\L^{4\delta}}+\|{\Pi u}\|^{4\delta}_{\L^{4\delta}}\right)\|\nabla\Pi\xi\|_{\L^2}^2\mathrm{~d}t\\&\quad+\left(\frac{2^{2\delta-1}\alpha^2}{\nu}+\frac{\beta}{2}\right)\int_0^{t_n}\left(\|\Pi u\|^{2\delta}_{\L^{4\delta}}+ \|u\|^{2\delta}_{\L^{4\delta}} \right)\|\nabla \rho\|^2_{\L^2} \mathrm{~d}t+ \frac{\nu}{4}\int_0^{t_n}\|\nabla \psi\|^2_{\L^2}\mathrm{~d}t\no\\&\quad+ \left(2^{2\delta-1}\beta(1+\gamma)^2(\delta+1)^2+\frac{\beta\gamma}{2}\right)\int_0^{t_n} \|\psi\|^2_{\L^2}\mathrm{~d}t+\frac{\beta\gamma}{2}\int_0^{t_n}\|\rho\|^2_{\L^2}\mathrm{~d}t\no\\&\quad+ 2^{4\delta-1}(2\delta+1)^2\beta^2 \int_0^{t_n}\left(\|{\Pi u}\|^{4\delta}_{\L^{4\delta}}+\|{u}\|^{4\delta}_{\L^{4\delta}}\right)\|\nabla\rho\|_{\L^2}^2\mathrm{~d}t
	\end{align}
we get the required proof.
\end{proof}
Let us know compile the above lemma as
\begin{lemma}\label{5.l14}
	For $1\leq n\leq N,$ we have  
		\begin{align}\label{5.104}
		\|\psi^n\|_{\L^2} ^2 + \nu \int_{0}^{t_n}\|\nabla\psi\|^2_{\L^2}\mathrm{~d}t&\leq  C\bigg( \|U_h^0-\mathcal{R}^{\mathbf{h}}_{\mathbf{r}}u_0\|_{\L^2}^2 + \int_0^{t_n}\|\xi'\|_{\L^2}^2 \mathrm{~d}t + \int_0^{t_n}\|\nabla \rho\|_{\L^2}^2\mathrm{~d}t+\Xi\bigg).
		\end{align}
	\end{lemma}
\begin{proof}
The desired result follows directly from Lemmas \ref{5.l5} and \ref{5.43}, in conjunction with Gr\"onwall's inequality.
\end{proof}
\begin{lemma}\label{5.l15}
	For $1\leq n\leq N$, we achieve
		\begin{align}\label{5.105}
		\int_{t_{n-1}}^{t_n}\|\psi'\|^2_{\L^2}(t-t_{n-1})\mathrm{~d}t &\leq C p_n^2\Big(\|U^0-u_0\|_{\L^2} +\int_0^{t_n}\|\psi'\|_{\L^2}\mathrm{~d}t + \int_{0}^{t_n} \|\nabla\rho\|^2_{\L^2} \mathrm{~d}t +\Xi \Big).
	\end{align}
\end{lemma}
\begin{proof}
The proof employs a similar approach to that used in Lemma \ref{5.l2.8} with modified
$$
J^n_9 = \nu \int_{t_{n-1}}^{t_n} \left[ \langle \xi', (t - t_{n-1}) \psi' \rangle + a(\rho, (t - t_{n-1}) \Upsilon') \right] \, \mathrm{~d}t.
$$
To establish the result, we apply Lemma \ref{5.l2.9} and then combine the findings in a manner similar to  Lemma \ref{5.l2.13}.
\end{proof}
Let us now estimate the error in $\psi$ as follow
\begin{lemma}
	If $U_h\in \mathcal{S}(\mathcal{M},\textbf{p},\mathcal{V}_h)$ is the approximate solution defined in \eqref{5.3.2} and $\psi = U_h-\Pi \mathcal{R}^{\mathbf{h}}_{\mathbf{r}}u$, then for $1\leq n \leq N$, 
\begin{align}
		\|\psi\|_{I_n}^2&\leq C \log(|\textbf{p}|_n+2)|\textbf{p}|_n^2\bigg(\|U_h^0-\mathcal{R}^{\mathbf{h}}_{\mathbf{r}} u_0\|_{\L^2}^2 + \int_{0}^{t_n}\|\xi'\|_{\L^2}^2\mathrm{~d}t + \int_{0}^{t_n}\|\nabla\rho\|_{\L^2}^2\mathrm{~d}t+ \Xi\bigg).
\end{align}
\end{lemma}
\begin{proof}
The proof follows a similar approach to that used in Lemma \ref{5.supnorm}. To achieve the desired results, we replace Lemma \ref{5.l.4} and Lemma \ref{5.lem12} with Lemma \ref{5.104} and Lemma \ref{5.105} respectively.
\end{proof}
The main result for the fully-discrete error estimates are given as follows in the following theorems
\begin{theorem}\label{5.th8}
	Let $U_h$ be the approximate solution and $u$ be the exact solution then the following bound holds
			\begin{align}
			\nonumber&\int_0^{t_n} \|\nabla(U_h-u)\|_{\L^2}^2 \mathrm{~d}t +\|(U-u)^n\|_{\L^2}^2 \leq  C \left(\|\Pi \xi\|_{I_n}^2 + \int_{0}^{t_n}(\|\xi'\|_{\L^2}^2+\|\nabla\rho\|_{\L^2}^2 )\mathrm{~d}t + \Xi \right),
		\end{align}
\begin{align}
	\|U_h-u\|_{I_n}^2&\leq C(\|\rho\|_{I_n}^2+\|\Pi\xi\|_{\J_n}^2) + C \log(|\textbf{p}|_n+2)|\textbf{p}|_n^2\bigg(\|\Pi \xi\|_{I_n}^2 + \int_{0}^{t_n}\Big(\|\xi'\|_{\L^2}^2+\|\nabla\rho\|_{\L^2}^2\Big)\mathrm{~d}t + \Xi \bigg).
\end{align}
\end{theorem}
\begin{proof}
	The first can be obtained using the decomposition \eqref{5.3.3} of $U_h-u$, employing triangle inequality, Lemma \ref{5.l14} and the definition of projection that $\rho^n=0$ for $1\leq N$ and similarly second bound is obtained from Lemma \ref{5.l15}.
	To derive the first bound, we utilize the decomposition of $U_h-u$ given by Equation \eqref{5.3.3}, apply the triangle inequality, and make use of Lemma \ref{5.l14}. Furthermore, the projection definition \eqref{5.8} ensures $\rho^n=0$ for $1\leq N$. The second bound follows similarly through the application of Lemma \ref{5.l15}.
\end{proof}
For the rest of this paper, we will proceed under the assumption that both $u$ and its corresponding initial condition $u_0$ meet the specified conditions of regularity.
\begin{align}\label{5.106}
u_0 \in \H^{m_K}(\Omega),\left.\quad u\right|_{I_n} \in \H^{q_n+1}\left(t_{n-1}, t_n ; \H^1(\Omega)\right) \cap \H^1\left(t_{n-1}, t_n ; \H^{m_K}(\Omega)\right).
\end{align}
for $1 \leq n \leq N$ and $1 \leq s \leq r$.
\begin{theorem}\label{5.th15}
For the solution exact solution $u$ satisfying the regularity \eqref{5.106}, the error for $1\leq n\leq N$ and for $0\leq q_j\leq p_j$ is given as follows
\begin{align}
\int_0^{t_n} \|\nabla(U_h-u)\|_{\L^2}^2 \mathrm{~d}t +\|(U-u)^n\|_{\L^2}^2 &\leq C \frac{h_K^{2s_K -2 }}{r_K^{2m_K-2}}\left(\left\|u_0\right\|_{\H^{m_K}}^2+\int_0^{t_n}\left\|u^{\prime}\right\|^2_{\H^{m_K}}  \mathrm{~d}t +\int_0^{t_n}\left\|u\right\|^2_{\H^{m_K}}  \mathrm{~d}t \right)\\&\quad + C \sum_{j=1}^n\hat{p}_j^{-2}\left(\frac{k_j}{2}\right)^{2q_j}{\Gamma_{p_j, q_j}} \int_{t_{j-1}}^{t_j}\|\nabla u^{\left(q_j+1\right)}\|_{\L^2}^2 \mathrm{~d}t.
\end{align}
and 
\begin{align}
	\|U_h-u\|_{I_n}^2 &\leq C \frac{h_K^{2s_K-2 }}{r_K^{2m_K-2}}\log(|\textbf{p}|_n+2)|\textbf{p}|_n^2\left(\left\|u_0\right\|_{\H^{m_K}}^2+\int_0^{t_n}\left\|u^{\prime}\right\|^2_{\H^{m_K}}  \mathrm{~d}t\right)\\&\quad+ C\left(\max_{j=1}^n\right)\left(\frac{k_j}{2}\right)^{2q_j}{\Gamma_{p_j, q_j}} \int_{t_{j-1}}^{t_j}\|\nabla u^{\left(q_j+1\right)}\|_{\L^2}^2 \mathrm{~d}t\\&\quad+\log(|\textbf{p}|_n+2)|\textbf{p}|_n^2\sum_{j=1}^n\hat{p}_j^{-2} \left(\frac{k_j}{2}\right)^{2q_j}{\Gamma_{p_j, q_j}} \int_{t_{j-1}}^{t_j}\|\nabla u^{\left(q_j+1\right)}\|_{\L^2}^2 \mathrm{~d}t.
\end{align}
where $\hat{p}_j=\max \left\{1, p_j\right\}$ and $s_K = \min(r_K +1,m_K ).$

\end{theorem}
\begin{proof}
	Using Theorem \ref{5.th8} and \ref{5.th1} our task reduces to estimate $\|\Pi \xi\|_{J_n}$, $\int_0^{t_n}\left\|\xi^{\prime}\right\|_{\L^2}^2 \mathrm{~d} t$ and $\Xi$. The triangle inequality yields
\begin{align}\label{5.107}
	\|\Pi \xi\|_{J_n} \leq\|\Pi \xi-\xi\|_{J_n}+\|\xi\|_{J_n} \leq\|\Pi \xi-\xi\|_{J_n}+\|\xi(0)\|_{
		\L^2}+\int_0^{t_n}\left\|\xi^{\prime}\right\|_{\L^2} \mathrm{~d}t.
\end{align}
	For the first term we have 
\begin{align}
	\|\Pi \xi-\xi\|_{J_n}^2=\max _{j=1}^n\left(\|\Pi \xi-\xi\|_{I_j}^2\right) \leq C \max _{j=1}^n\left(k_j \int_{t_{j-1}}^{t_j}\left\|\xi^{\prime}\right\|^2 d t\right),
\end{align}
	where we have used approximation result for $\Pi\xi-\xi$ from Theorem \ref{5.th1} for $q_n=0$. Further an application of Cauchy-Schwarz inequality in \eqref{5.107} yields
	\begin{align}
			\|\Pi \xi\|_{J_n} \leq C\left(\|\xi(0)\|_{
			\L^2}+\int_0^{t_n}\left\|\xi^{\prime}\right\|_{\L^2} \mathrm{~d}t\right).
\end{align}
Let us estimate $\Xi$ term by term.
To estimate $\int_0^{t_n}\|\nabla\Pi \xi\| _{\L^2}^2 \mathrm{~d} t$ using the approximation property defined in Lemma \ref{5.th1} and \ref{5.lemm3}, we have 
\red{\begin{align}\label{5.212}
	\int_0^{t_n}\|\nabla(\Pi \xi-\xi)\| _{\L^2}^2  \mathrm{~d} t &= \sum_{j=1}^n\int_{t_{j-1}}^{t_j}\|\nabla(\Pi \xi-\xi)\| _{\L^2}^2 \mathrm{~d} t\\& \leq \sum_{j=1}^n\frac{C}{\max \left\{1, p_j^2\right\}}\left(\frac{k_j}{2}\right)^{2 q_j+2} \Gamma_{p_j, q_j} \int_{t_{j-1}}^{t_j}\left\|\nabla \xi^{\left(q_j+1\right)}\right\|_{\L^2}^2 \mathrm{~d}t\\& \leq \sum_{j=1}^n\frac{C}{\max \left\{1, p_j^2\right\}}\left(\frac{k_j}{2}\right)^{2 q_j+2} \Gamma_{p_j, q_j}\left(\frac{h_K^{2s_k-2}}{r_K^{2m_K-2}}\right)\int_{t_{n-1}}^{t_n}\left\| u^{\left(q_j+1\right)}\right\|_{\H^{m_K}}^2 \mathrm{~d}t.
\end{align}
which is converges faster. Also we have}
\begin{align*}
	\red{\int_0^{t_n}\|\nabla\xi\| _{\L^2}^2 \mathrm{~d} t \leq \frac{h_K^{2s_K - 2}}{r_K^{2m_K-2}} \int_0^{t_n}\| u \|_{\H^{m_K}}^2 \mathrm{~d} t  }
\end{align*}
\begin{align}
	\|\xi(0)\|^2+\int_0^{t_n}\left\|\xi^{\prime}\right\|_{\L^2}^2 d t \leq C \frac{h_K^{2s_K - 2}}{r_K^{2m_K-2}}\left\|u_0\right\|_{\H^{m_K}}^2+C\frac{h_K^{2s_K - 2}}{r_K^{2m_K-2}} \int_0^{t_n} \left\|u^{\prime}\right\|_{\H^{m_K}}  \mathrm{~d}t,
\end{align}
and the last two terms of $\Xi$ gives the similar estimate using Lemma \ref{5.l11}
and hence the result.
\end{proof} 
For uniform parameters $k$, $ p $, and $q$ (i.e., $k_j = k, p_j = p$, and $q_j = q$), the bounds in Theorem \ref{5.th15} result in the following error estimates
\begin{corollary}\label{5.col2}
	 For $1 \leq n \leq N$, we have
	\begin{align}
		\int_0^{t_n}\left\|\nabla(U_h-u)\right\|_{\L^2}^2 d t+\left\|\left(U_h-u\right)^n\right\|^2 & \leq C \frac{h^{2s-2 }}{r^{2m-2}}\left(\left\|u_0\right\|_{\H^{m}}^2+\int_0^{t_n}\left\|u^{\prime}\right\|_{\H^{m}}  \mathrm{~d}t\right)\\
		& +C \frac{k^{2 \min \{p, q\}}}{p^{2 q}} \int_0^{t_n}\left\|\nabla u^{(q+1)}\right\|_{\L^2}^2  \mathrm{~d}t.
	\end{align}
and 
\begin{align}
	\left\|U_h-u\right\|_{J_n}^2 & \leq C\frac{h^{2s-2 }}{r^{2m-2}} \log (p+2)\left(\left\|u_0\right\|_{\H^{m}(\kappa)}^2+\int_0^{t_n}\left\|u^{\prime}\right\|_{\H^{m}(\kappa)}  \mathrm{~d}t\right)\\
	& +C \frac{k^{2 \min \{p, q\}}}{p^{2 q}}\left(\max _{j=1}^n \max _{t \in I_j}\left\|\nabla u^{(q+1)}(t)\right\|_{\L^2}+\log (p+2) \int_0^{t_n}\left\|\nabla u^{(q+1)}\right\|_{\L^2}^2 d t\right) .
\end{align}
\end{corollary}
\begin{proof}
	These estimates follow readily from Theorem \ref{5.th15}  and the fact that $\Gamma_{p, q}$ behaves like $p^{-2 q}$ for $p \rightarrow \infty$. 
\end{proof}
\begin{remark}
The estimates in Corollary \ref{5.col2} demonstrate that the discrete scheme converges either as $k, h \to 0$ or as $p, r \to \infty$. It is noted that the first estimate is optimal with respect to the four parameters $k$, $h$, $p$, and $r$, while the second estimate is one power short of being optimal in $p$. For smooth solutions $u$, spectral convergence rates are achieved by increasing the polynomial degrees $p$ and $r$ on fixed partitions.
\end{remark}
\begin{remark}
	The regularity assumption for $u$ needs to be rigorously established. Moreover, for the exponential convergence, proving the analytic regularity remains an active area of ongoing research.
\end{remark}
	\subsection{$hp$-Discontinuous Galerkin FEM in space}\label{5.sub3}
In this section, we will discuss the fully-discrete finite element scheme that employs both $hp$-Discontinuous Galerkin (DG) finite element methods (FEM) in space and $hp$-DG time-stepping. For the mesh discretization using $hp$-DG elements, let $\mathcal{T}$ denote the set of elements (triangles in 2D and tetrahedra in 3D), and let $\mathcal{E}$ be the set of edges. Within $\mathcal{E}_h$, $\mathcal{E}^i_h$ represents the set of internal edges, while $\mathcal{E}^{\partial}_h$ represents the set of boundary edges. For an edge $E = K_+ \cap K_- \in \mathcal{E}^i_h$, which is shared between the two mesh cells $K_\pm$, we denote the traces of a function $w \in C^0(\mathcal{T})$ on $E$ for $K_\pm$ by $w_\pm$, respectively.
We begin by considering the DG space formulation, which is given as follows:

\begin{align}\label{5.dgsubspace1}
	\mathcal{V}_{\textbf{h,r}}^{DG}=\{{v}\in \L^2(\Omega): \forall\  K\in \mathcal{T} : {v}|_K \in \mathbb{P}_{r_K}(K)\},
\end{align}
where $\textbf{h}$ and $\textbf{r}$ are as defined in the previous section.
The trial space for the fully-discrete case, given the partition $\mathcal{M}=\{\J_n\}_{n=1}^N $ of $(0,T)$ and the degree vector $\mathbf{p} = (p_1, p_2, \cdots, p_N)$, is defined as follows:
 \begin{align}\label{5.3.1}
	\mathcal{S}(\mathcal{M},\textbf{p},\mathcal{V}_{\textbf{h,r}}^{DG})=\{U_h : [0,T] \rightarrow \mathcal{V}_{\textbf{h,r}}^{DG} : {U_h}|_{\J_n} \in \mathbb{P}_{p_n}(\mathcal{V}_{\textbf{h,r}}^{DG}), 1 \leq n \leq N\}.
\end{align}

To define the DG formulation, we introduce the average operator $\{\!\{\cdot\}\!\}$ and the jump operator $[\![\cdot]\!]$ on the edge $E$ as follows:
	\begin{align*}
		\{\!\{w\}\!\}=\frac{1}{2}(w_+ + w_-) \quad  \text{and}	\quad[\![w]\!]={w_+\mathbf{{n}_+} + w_-\mathbf{{n}_-}},
	\end{align*}
	respectively. If $w \in C^1(\mathcal{T})$, we define the jump of the normal derivative as
	$
	[\![\partial w / \partial \mathbf{n}] \!] = \nabla(w_+ - w_-) \cdot \mathbf{n}_+,
	$
	where $\mathbf{n}_\pm$ are the unit outward normal vectors for the mesh cells $K_\pm$. For edges $E$ on the boundary, where $E \in K_+ \cap \partial \Omega$, we have
	$
	[\![w]\!] = w_+ \mathbf{n}_+
	$
	and
	$	\{\!\{w\}\!\} = w_+.$
	The exterior trace of a function $w$ is denoted by $w^e$. For boundary edges, we set $w^e = 0$.
	For each triangulation, we define the piecewise gradient operator $\nabla_h: \H^1(\mathcal{T}) \to L^2(\Omega; \mathbb{R}^d)$ such that for any $w \in \H^1(\mathcal{T})$, we have
	$$
	(\nabla_h w)|_K = \nabla w|_K \quad \forall K \in \mathcal{T}.
$$
The following approximation lemma is useful in the further analysis
\red{\begin{lemma}\label{5.lemm3.1}
		Suppose that a triangulation $\mathcal{T}$ of $\Omega$ is formed by $d$-dimensional simplices or parallelepipeds. Then for an arbitrary $u \in \H^{m}(\Omega, \mathcal{T})$, $m = (m_K, K \in \mathcal{T})$, and for each $r = (r_K, K \in \mathcal{T})$, $r_K \in \mathbb{N}$, there exists a projector
		$$
		\mathcal{R}^{\mathbf{h}}_{\mathbf{r}} : \H^s(\Omega, \mathcal{T}) \to 	\mathcal{S}(\mathcal{M},\textbf{p},\mathcal{V}_{\textbf{h,r}}^{DG}), \quad
		(	\Pi^{\mathbf{h}}_{\mathbf{r}} u)|_K = 	\Pi^{\mathbf{h}}_{\mathbf{r}} (u|_K),
		$$
		such that for $0 \leq b \leq m_K$,
		\begin{align}\label{5.7a1.1}
			\| u -\mathcal{R}^{\mathbf{h}}_{\mathbf{r}}  u \|_{\H^b(K)} \leq C \frac{h_K^{s_K - b}}{r_K^{m_K-b}} \| u \|_{\H^{m_K}(K)} \quad \forall K \in \mathcal{T},
		\end{align}
		where $s_K = \min(r_K +1,m_K )$ and $C$ is a constant independent of $u, h_K,$ and $r_K$ but
		dependent on $m = \max\limits_{K\in\mathcal{T}} m_K.$
\end{lemma}}
Using the above discretisation, the fully-discrete weak formulation of \eqref{5.GBHE} is given by:  Find $U^{DG}_h \in \mathcal{S}(\mathcal{M},\textbf{p},	\mathcal{V}_{\textbf{h,r}}^{DG})$ such that
\begin{align}\label{5.3.2Dg}
	\nonumber	F_N^{DG}(U_h^{DG},X) &= \langle (U_h^{DG})^{0}, X_{+}^0\rangle + \int_0^{t_N}\langle f(t),X(t) \rangle \mathrm{~d}t  \quad \forall X\in  \mathcal{S}(\mathcal{M},\textbf{p},	\mathcal{V}_{\textbf{h,r}}^{DG} ),\\
	U_h^{DG}(0) &= (U_h^{DG})^{0},
\end{align}
for a suitable approximation $U_h^{0,DG} \in\mathcal{V}_{\textbf{h,r}}^{DG} $ to $u_0$, where

 \begin{align}\label{5.GWFDG}
	\nonumber F^{DG}_N(U,X) &= \langle U_{+}^{0}, X_{+}^{0} \rangle +\sum_{n=1}^{N-1}\langle [U]^n, X_{+}^n \rangle + \sum_{n=1}^N \int_{t_{n-1}}^{t_n}\Big[\langle U',X\rangle + \nu a^{DG}(U,X) + \alpha b^{DG}(U,U,X) \\&\quad + \eta a^{DG}\left(K*U, X \right)  - \beta (c(U),X) \Big] \mathrm{~d}t,
\end{align} 
with
\begin{align}\label{2.adg}
	a_{DG}({u},{v})=(\nabla_h {u}, \nabla_h {v})-\sum_{E\in\mathcal{E}_h}\int_E {\{\!\{\nabla_h {u}\}\!\}}\!\cdot\![\![{v}]\!]\mathrm{~d}{s}-\sum_{E\in\mathcal{E}_h}\int_E {\{\!\{\nabla_h {v}\}\!\}}\!\cdot\![\![{u}]\!]\mathrm{~d}{s}
	+\frac{\gamma r_E^2}{h_E}\sum_{E\in\mathcal{E}_h}\int_E [\![{u}]\!]\!\cdot\![\![{v}]\!]\mathrm{~d}{s}.
\end{align}
Here the constant $\gamma>0$ represents the penality parameter chosen sufficiently larger 
independent of $\textbf{h}$, $\textbf{r}$ to ensure the stability and  well-posedness of the DG discretization and
$$r_E =
\begin{cases}
	\max\{r_K, r_{K'}\}, & \mathcal{E}_h = \partial K \cap \partial K' \in \mathcal{E}_h^i, \\
	r_K, & \mathcal{E}_h = \partial K \cap \partial \Omega \in \mathcal{E}_h^{\partial}.
\end{cases},\qquad h_E =\begin{cases}
\max\{h_{K}, h_{K'}\}, & \mathcal{E}_h = \partial K \cap \partial K' \in \mathcal{E}_h^i, \\
h_{K}, & \mathcal{E}_h = \partial K \cap \partial \Omega \in \mathcal{E}_h^{\partial}.
\end{cases}$$ The discrete convection term is defined as

\begin{align}\label{2.bDG}
	\nonumber	b_{DG}(\mathbf{w};u,v)=\frac{1}{\delta+2}\Bigg(& \sum_{K\in \mathcal{T}} \int_{K} \mathbf{w}\cdot \nabla u v \mathrm{~d}x 
	+\sum_{K\in \mathcal{T}} \int_{\partial K} \hat{\mathbf{w}}_{h,u}^{up}{v} \mathrm{~d}s\\& -\sum_{K\in \mathcal{T}} \int_{K} \mathbf{w}\cdot \nabla v {u} \mathrm{~d}x 
	-\sum_{K\in \mathcal{T}} \int_{\partial K} \hat{\mathbf{w}}_{h,v}^{up}{u} \mathrm{~d}s\Bigg).
\end{align}
where the upwind flux, $\hat{\mathbf{w}}_{h,u}^{up}=\frac{1}{2}\left[\mathbf{w}\cdot\mathbf{n}_K -|\mathbf{w}\cdot \mathbf{n}_K|\right](u^e\!-\!u)$ with $\mathbf{w}=(w,w)^T$. 
The length of the edge $E$ is represented by the parameter $h_E$. The following discrete norm is used for further error analysis: 
\begin{align}\label{5.1.2}
	|\!|\!|{ v}|\!|\!|_{DG}^2:= \sum_{K\in\mathcal{T}}\|\nabla_h v\|_{\L^2(\mathcal{T})}^2 + \sum_{E\in\mathcal{E}_h}\gamma_h\|[\![v]\!]\|_{\L^2(E)}^2.
\end{align}
\begin{remark}
The DG formulation described in \eqref{5.GWFDG} was introduced by the authors in \cite{GBHE2} for the $h$-version Discontinuous Galerkin Finite Element Method (DGFEM) in space with backward Euler discretization in time. Note that the non-linear discrete operator $b_{DG}(\cdot, \cdot, \cdot)$ is constructed such that $b_{DG}(u, u, u) = 0$ for all $u \in \mathcal{V}_{\textbf{h,r}}^{DG}$. This property facilitates proving the well-posedness and stability results without any restrictions on the parameters, as discussed in detail in \cite{GBHE2}.
\end{remark}
Under the DG norm defined in \eqref{5.1.2}, the discrete diffusion operator statises the following estimate
\begin{lemma}[Coercivity and Continuity]
	For any $u\in \mathcal{V}_{\textbf{h,r}}^{DG}$, the operator $a_{DG}(\cdot,\cdot)$ satisfies the following estimate
\begin{align}\label{5.1.3}
	a_{DG}(v,v) &\geq C_{\text{cor}} |\!|\!| v |\!|\!|_{DG}^2
	\qquad \hspace{2mm}\quad\quad\forall v\in \mathcal{V}_{\textbf{h,r}}^{DG}, \\
	a_{DG}(u,v) &\leq C_{\text{con}} |\!|\!| u |\!|\!|_{DG} |\!|\!| v |\!|\!|_{DG} 	\qquad \forall u,v\in \mathcal{V}_{\textbf{h,r}}^{DG},\label{5.1.3b}
\end{align}
The constants $ C_{\text{cor}} $ and $ C_{\text{con}} $ depend on the parameters and the shape regularity of the mesh but are independent of the discretization parameters $ \textbf{h} $ and $ \textbf{r} $.
\begin{proof}
	The proof for coercivity and continuity follows directly form \cite[Prop 3.1]{PSc}.
\end{proof}
\end{lemma}
 Let us now discuss the well-posedness of the fully-discrete DG scheme presented in \eqref{5.GWFDG}.
 \begin{lemma} The existence, uniqueness and stability of the fully-discrete DG scheme  \eqref{5.GWFDG} is given as
 	\begin{enumerate} 
 		\item 	For  $u_0\in \L^2(\Omega), f \in \L^2(0,T;\L^2(\Omega))$ there exist at least one solution $U_h^{DG}\in  \mathcal{V}_{\textbf{h,r}}^{DG}$. Moreover, for $(U_h^{DG})^{0}\in\L^{d\delta}(\Omega)$,  the weak solution to the system \eqref{5.GWFDG} is unique. 
 		\item 	Assume that $f\in\L^2(0,T;\L^2(\Omega))$ and $u_0\in \L^2(\Omega)$ then the fully-discretized solution $U_h^{DG}$ of the \eqref{5.GBHE} defined in \eqref{5.GWFDG} is stable in the sense that  
 		\begin{align}\label{2.DG13} 
 			& \|(U_h^{DG})^{n}\|_{\L^2} ^2 + \nu \int_{0}^{t_n}	|\!|\!|{ U_h^{DG}}|\!|\!|_{DG}^2\mathrm{~d}t \leq\left(\|u_0\|_{\L^2}^2+\frac{1}{\nu}\int_0^T\|f(t)\|_{\L^2}^2\mathrm{~d}t\right)e^{\beta(1+\gamma^2)T}.
 		\end{align}
 	\end{enumerate}
 	\begin{proof}
 	The proof of existence and uniqueness is analogous to that in \cite[Theorem 2.6]{GBHE2}. The stability estimate follows from the coercivity of the operator $a_{DG}$ defined in \eqref{5.1.3} and the property $ b_{DG}(u, u, u) = 0$, as demonstrated in \cite[Lemma 2.12]{GBHE2}, where the time-stepping estimates are derived similarly to \eqref{5.14}, but with $ U_h^{DG} $ replacing $\Upsilon $.
 	\end{proof}
 \end{lemma}
\subsection{Error estimates} To derive the error estimates for the fully-discrete case \ref{5.3.2Dg}, we use the decomposition \eqref{5.3.3}. 
\begin{theorem}
	If $u$ is a solution of problem \eqref{5.GBHE}, and $U_h^{DG} \in \mathcal{S}(\mathcal{M},\textbf{p},\mathcal{V}_{\textbf{h,r}}^{DG})$ is the approximate solution defined by  \eqref{5.3.2}, then
		\begin{align}\label{5.3.5.1}
			&\nonumber F_{N}^{DG}(U_h^{DG}, X) - F_N^{DG}(\Pi \mathcal{R}^{\mathbf{h}}_{\mathbf{r}}u, X)\\& \no=  \langle ((U_h^{DG})^{0}-\mathcal{R}^{\mathbf{h}}_{\mathbf{r}}u_0), X_{+}^0\rangle - \int_{0}^{t_N}\langle\xi',X\rangle \mathrm{~d}t   - \int_{0}^{t_N}\Big[ \nu a_{DG}(\rho,X)+\nu a_{DG}(\Pi\xi,X) + \alpha (b_{DG}(\Pi \mathcal{R}^{\mathbf{h}}_{\mathbf{r}}u,\Pi \mathcal{R}^{\mathbf{h}}_{\mathbf{r}}u,X)\\&\quad-b_{DG}(\Pi u,\Pi u, X) )+ \alpha (b_{DG}(\Pi u,\Pi u,X)-b_{DG}( u,u, X) ) + \eta \left(K*a_{DG} (\rho(s),\nabla X) \right) + \eta \left(K*a_{DG}(\Pi\xi(s),\nabla X) \right)  \\&\quad- \beta (c(\Pi \mathcal{R}^{\mathbf{h}}_{\mathbf{r}}u)-c(\Pi u),X) - \beta (c(\Pi u)-c(u),X) \Big] \mathrm{~d}t,  
		\end{align}
		where $F_N^{DG}(\cdot,\cdot)$ is defined in \eqref{5.GWFDG} and $X\in 	\mathcal{S}(\mathcal{M},\textbf{p},\mathcal{V}_{\textbf{h,r}}^{DG})$.
\end{theorem}
\begin{proof}
	The proof is similar to Theorem \ref{5.th3}.
\end{proof}
 For the error estimates we invoke the following Lemma from \cite{GBHE2}, which gives the estimate for the non-linear terms under DG setting in space 
 \begin{lemma}\label{5.lem111}
 	For $A_{DG}(U_h^{DG},\chi) = a_{DG}(U_h^{DG},\chi) + b_{DG}(U_h^{DG},u,\chi)-c(U_h^{DG},\chi)$, for $\chi\in\mathcal{V}_{\textbf{h,r}}^{DG}$,   there holds
 	\begin{align*}
 		-\alpha[b_{DG}(U_h^{DG};U_h^{DG},w)-b_{DG}(v_h;&v_h,w)]\le  \frac{\nu}{2}|\!|\!| w|\!|\!|^2_{DG}+C(\alpha,\nu)\left(\|U_h^{DG}\|^{\frac{8\delta}{4-d}}_{\L^{4\delta}}+\|v_h\|^{\frac{8\delta}{4-d}}_{\L^{4\delta}}\right)\|w\|_{\L^2}^2,\\
 		A_{DG}(U_h^{DG},w)-A_{DG}(v_h,w)&\ge  \frac{\nu}{2}|\!|\!| w|\!|\!|^2_{DG} +\frac{\beta}{4}(\|(U_h^{DG})^{\delta}w\|_{\L^2}^2+\|v^{\delta}_hw\|_{\L^2}^2)\nonumber\\
 		&\quad+\left(\beta\gamma-C(\beta,\alpha,\delta) - C(\alpha,\nu)\Big(\|U_h^{DG}\|^{\frac{8\delta}{4-d}}_{\L^{4\delta}}+\|v_h\|^{\frac{8\delta}{4-d}}_{\L^{4\delta}}\Big)\right)\|w\|_{\L^2}^2,
 	\end{align*}
 	where $U_h^{DG},v_h\in \mathcal{V}_{\textbf{h,r}}^{DG}$, $w=U_h^{DG}-v_h$, $C(\alpha, \nu) = \left(\frac{4+d}{4\nu}\right)^{\frac{4+d}{4-d}}\left(\frac{4-d}{8}\right)(\frac{2^{\delta-1}C\alpha}{(\delta+2)(\delta+1)})^{\frac{4-d}{8}}$ and  $C(\beta,\alpha,\delta)= \frac{\beta}{2}2^{2\delta}(1+\gamma)^2(\delta+1)^2$ is a positive constant depending on parameters.
 \end{lemma}
Further, to discuss the fully-discrete we discuss an Lemma analogues to estimate \eqref{5.212} for the diffusion term as
\red{\begin{lemma}\label{5.AdG}
		There holds
	\begin{align}\label{5.diff}
		\Xi^1 &= \nu \int_0^{t_n}a_{DG}(\Pi\xi,\psi) +\eta \int_0^{t_n} (K*a_{DG}(\Pi\xi,\psi))  \mathrm{~d}t\\& \leq C\bigg(\sum_{K \in \mathcal{T}} \left( \frac{h_K^{2s_K - 2}}{r_K^{2m_K-3}}\right) \int_0^{t_n}\| u \|_{\H^{m_K}(K)}\mathrm{~d}t +\frac{\nu}{4}\int_0^{t_n}|\!|\!| \psi |\!|\!|^2_{DG}  \mathrm{~d}t\bigg) + h.o.t.,
	\end{align}
where $h.o.t.$ are higher order terms.
\end{lemma}
\begin{proof}
	Using the continuity estimate defined in \eqref{5.1.3b}, we have 
	\begin{align}
		\int_0^{t_n} \nu a_{DG}(\Pi\xi,\psi)  \mathrm{~d}t\leq \nu	\int_0^{t_n} |\!|\!| \Pi\xi |\!|\!|_{DG} |\!|\!| \psi |\!|\!|_{DG}  \mathrm{~d}t\leq 	\nu\int_0^{t_n}|\!|\!| \Pi\xi |\!|\!|_{DG}^2  \mathrm{~d}t+\frac{\nu}{4} 	\int_0^{t_n}|\!|\!| \psi |\!|\!|^2_{DG}  \mathrm{~d}t.
	\end{align}
Now, further aim is to estimate $\int_0^{t_n}|\!|\!| \Pi\xi |\!|\!|_{DG}^2$. Rewriting the term and using an argument similar to \cite[Theorem 4.5]{HSS}, we have  
\begin{align}
	\int_0^{t_n} |\!|\!| \Pi\xi |\!|\!|_{DG}^2 &\leq \int_0^{t_n} |\!|\!| \Pi\xi - \xi |\!|\!|_{DG}^2 + \int_0^{t_n} |\!|\!| \xi |\!|\!|_{DG}^2 \\
	&\leq\sum_{j=1}^n \sum_{K \in \mathcal{T}} \frac{C}{\max \left\{1, p_j^2\right\}}\left(\frac{k_j}{2}\right)^{2 q_j+2} \Gamma_{p_j, q_j} \int_{t_{j-1}}^{t_j} |\!|\!| \xi^{\left(q_j+1\right)} |\!|\!| _{DG}^2 \mathrm{~d}t\\&\quad+\sum_{K \in \mathcal{T}} \left(\frac{h_K^{2s_K - 2}}{r_K^{2m_K-2}} + \frac{h_K^{2s_K - 2}}{r_K^{2m_K-3}}\right) \int_0^{t_n}\| u \|_{\H^{m_K}(K)}\mathrm{~d}t\\&\leq \sum_{j=1}^n \sum_{K \in \mathcal{T}} \frac{C}{\max \left\{1, p_j^2\right\}}\left(\frac{k_j}{2}\right)^{2 q_j+2} \Gamma_{p_j, q_j}\left(\frac{h_K^{2s_K - 2}}{r_K^{2m_K-2}} + \frac{h_K^{2s_K - 2}}{r_K^{2m_K-3}}\right) \int_{t_{j-1}}^{t_j} \| u^{(q_j+1)} \|_{\H^{m_K}(K)}^2 \mathrm{~d}t\\&\quad+\sum_{K \in \mathcal{T}} \left(\frac{h_K^{2s_K - 2}}{r_K^{2m_K-2}} + \frac{h_K^{2s_K - 2}}{r_K^{2m_K-3}}\right) \int_0^{t_n}\| u \|_{\H^{m_K}(K)}\mathrm{~d}t.
\end{align}
Further, using Lemma \ref{5.l11}, the similar estimate holds true for $\eta \int_0^{t_n} (K*a_{DG}(\Pi\xi,\psi))  \mathrm{~d}t$.
\end{proof}}
Using the above results Lemma \ref{5.lem111} and \ref{5.AdG}, we obtain the following estimate
 \begin{lemma}\label{5.l5.1}
	For $1\leq n\leq N,$ we have
		\begin{align}\label{5.42.1}
			\nonumber &\|\psi^n\|_{\L^2} ^2 + \nu \int_{0}^{t_n}	|\!|\!|{\psi}|\!|\!|_{DG}^2\mathrm{~d}t+ \frac{\beta}{4}\int_0^{t_n}\|(U_h^{DG})^{\delta} \Upsilon\|_{\L^2}^2 \mathrm{~d}t+ \frac{\beta}{4} \int_0^{t_n}\|(\Pi \mathcal{R}^{\mathbf{h}}_{\mathbf{r}} u)^{\delta} \Upsilon\|_{\L^2}^2 \mathrm{~d}t \\\nonumber&\leq \|U^0-\mathcal{R}^{\mathbf{h}}_{\mathbf{r}}u_0\|_{\L^2} + 2\left(\nu+ \frac{C_K\eta^2}{\nu} \right) \int_{0}^{t_n} 	|\!|\!|\rho|\!|\!|_{DG}^2 \mathrm{~d}t +\Xi^1\\&\quad + \int_0^{t_n}C(\alpha,\nu)\left(\|U\|^{\frac{8\delta}{4-d}}_{\L^{4\delta}}+\|\Pi \mathcal{R}^{\mathbf{h}}_{\mathbf{r}} U\|^{\frac{8\delta}{4-d}}_{\L^{4\delta}}\right)\|\psi\|_{\L^2}^2\mathrm{~d}t+ C(\beta,\gamma, \delta) \int_{0}^{t_n}\|\psi\|_{\L^2} \mathrm{~d}t+ \frac{1}{2\nu}\int_{0}^{t_n} \|\xi'\|^2_{\L^2} \mathrm{~d}t\\&\quad -2\alpha \int_0^{t_n}(b_{DG}(\Pi \mathcal{R}^{\mathbf{h}}_{\mathbf{r}}u,\Pi \mathcal{R}^{\mathbf{h}}_{\mathbf{r}}u,\psi)-b_{DG}( \Pi u,\Pi u, \psi) )\mathrm{~d}t +2\beta \int_0^{t_n}\big(c(\Pi \mathcal{R}^{\mathbf{h}}_{\mathbf{r}}u)-c(\Pi u),\psi\big) \mathrm{~d}t \\&\quad -2\alpha \int_0^{t_n}(b_{DG}(\Pi u,\Pi u,\psi)-b_{DG}( u, u, \psi) )\mathrm{~d}t +2\beta \int_0^{t_n}\big(c(\Pi u)-c(u),\psi\big) \mathrm{~d}t.
				\end{align}
	\end{lemma}
Further the non-linear terms are given by the following lemma 
\begin{lemma}\label{5.43.1}
	For $u\in \L^2(0,T; \L^{4\delta}(\Omega)), 1\leq n \leq N, 0\leq q_j \leq p_j,$ and $u\in \H^{q_j+1}(J_j;\H_0^1),$ we have 
	\begin{align}\label{5.44.1}
		&-\alpha \int_0^{t_n}(b_{DG}(\Pi \mathcal{R}^{\mathbf{h}}_{\mathbf{r}}u,\Pi \mathcal{R}^{\mathbf{h}}_{\mathbf{r}}u,\psi)-b_{DG}( \Pi u,\Pi u, \psi) )\mathrm{~d}t +\beta \int_0^{t_n}\big(c(\Pi \mathcal{R}^{\mathbf{h}}_{\mathbf{r}}u)-c(\Pi u),\psi\big) \mathrm{~d}t \\& -\alpha \int_0^{t_n}(b_{DG}(\Pi u,\Pi u,\psi)-b_{DG}( u, u, \psi) )\mathrm{~d}t +\beta \int_0^{t_n}\big(c(\Pi u)-c(u),\psi\big) \mathrm{~d}t\\&\leq C \left(\int_0^{t_n} \left( 	|\!|\!|\rho|\!|\!|_{DG}^2 + \|\psi\|^2_{\L^2} \right) \mathrm{~d}t + \Xi^1\right).
	\end{align}
\end{lemma}
\begin{proof}
	The proof is similar to Theorem \ref{5.43}.
\end{proof}
Combining the results obtained in the above two lemmas, we obtain
\begin{lemma}\label{5.l14.1}
	For $1\leq n\leq N,$ we have  
	\begin{align}\label{5.104.1}
		\|\psi^n\|_{\L^2} ^2 + \nu \int_{0}^{t_n}	|\!|\!|\psi|\!|\!|_{DG}^2 \mathrm{~d}t\leq  C \|(U_h^{DG})^{0}-\mathcal{R}^{\mathbf{h}}_{\mathbf{r}}u_0\|_{\L^2}^2 + C \int_0^{t_n}\|\xi'\|_{\L^2}^2 \mathrm{~d}t +C \int_0^{t_n}	|\!|\!|\rho|\!|\!|_{DG}^2 \mathrm{~d}t  +\Xi^1.
	\end{align}
\end{lemma}
The main results for the fully-discrete error estimates are given in the following theorems
\begin{theorem}\label{5.th8.1}
	Let $U_h^{DG}$ be the approximate solution and $u$ be the exact solution then the following bound holds
	\begin{align}
		\nonumber&\int_0^{t_n}|\!|\!|U_h^{DG}-u|\!|\!|_{DG}^2 \mathrm{~d}t +\|(U-u)^n\|_{\L^2}^2 \leq  C \left(\|\Pi \xi\|_{I_n}^2 + \int_{0}^{t_n}(\|\xi'\|_{\L^2}^2+|\!|\!|\rho|\!|\!|_{DG}^2 )\mathrm{~d}t +\Xi^1 \right),
	\end{align}
	\begin{align}
		\|U_h^{DG}-u\|_{I_n}^2\leq C(\|\rho\|_{I_n}^2+\|\Pi\xi\|_{\J_n}^2) + C \log(|\textbf{p}|_n+2)|\textbf{p}|_n^2\left(\|\Pi \xi\|_{I_n}^2 + \int_{0}^{t_n}(\|\xi'\|_{\L^2}^2+|\!|\!|\rho|\!|\!|_{DG}^2 )\mathrm{~d}t + \Xi^1\right).
	\end{align}
\end{theorem}
\begin{proof}
The proof is similar to \ref{5.th8}.
\end{proof}
\begin{theorem}\label{5.th15.1}
	For the solution exact solution $u$ satisfying the regularity \eqref{5.106}, the error for $1\leq n\leq N$ and for $0\leq q_j\leq p_j$ is given as follows
	\begin{align}
			&\int_0^{t_n}|\!|\!|U_h^{DG}-u|\!|\!|_{DG}^2\mathrm{~d}t +\|(U-u)^n\|_{\L^2}^2 \\&\leq C \sum_{K \in \mathcal{T}} \left( \frac{h_K^{2s_K - 2}}{r_K^{2m_K-3}}\right)\left(\|u_0\|_{\H^{m_K}(K)}^2 + \int_0^{t_n} \|u^{\prime}\|_{\H^{m_K}(K)} \, \mathrm{d}t + \int_0^{t_n} \| u \|_{\H^{m_K}(K)} \, \mathrm{d}t \right)\\&\quad+ C \sum_{j=1}^n\hat{p}_j^{-2}\left(\frac{k_j}{2}\right)^{2q_j}{\Gamma_{p_j, q_j}} \int_{t_{j-1}}^{t_j} \|u^{q_j+1}\|_{\H^{m_K}(K)}^2 \mathrm{~d}t.
	\end{align}
	and 
	\begin{align}
		\|U_h^{DG}-u\|_{I_n}^2 &\leq \left( \frac{h_K^{2s_K - 2}}{r_K^{2m_K-3}}\right)\log(|\textbf{p}|_n+2)|\textbf{p}|_n^2\left(\|u_0\|_{\H^{m_K}(K)}^2 + \int_0^{t_n} \|u^{\prime}\|_{\H^{m_K}(K)} \, \mathrm{d}t + \int_0^{t_n} \| u \|_{\H^{m_K}(K)} \, \mathrm{d}t \right)\\&\quad+ C\left(\max_{j=1}^n\right)\left(\frac{k_j}{2}\right)^{2q_j}{\Gamma_{p_j, q_j}} \int_{t_{j-1}}^{t_j}\|u^{q_j+1}\|_{\H^{m_K}(K)}^2 \mathrm{~d}t\\&\quad+\log(|\textbf{p}|_n+2)|\textbf{p}|_n^2\sum_{j=1}^n\hat{p}_j^{-2} \left(\frac{k_j}{2}\right)^{2q_j}{\Gamma_{p_j, q_j}} \int_{t_{j-1}}^{t_j}\|u^{q_j+1}\|_{\H^{m_K}(K)}^2 \mathrm{~d}t.
	\end{align}
	where $\hat{p}_j=\max \left\{1, p_j\right\}$ and $s_K = \min(r_K +1,m_K )$.
\end{theorem}
\begin{proof}
	Using Theorem \ref{5.th8} and \ref{5.th1} our task reduces to estimate $\|\Pi \xi\|_{J_n}$, $\int_0^{t_n}\left\|\xi^{\prime}\right\|_{\L^2}^2 d t$ and $\Xi^1$. The triangle inequality yields
	\begin{align}\label{5.107.1}
		\|\Pi \xi\|_{J_n} \leq\|\Pi \xi-\xi\|_{J_n}+\|\xi\|_{J_n} \leq\|\Pi \xi-\xi\|_{J_n}+\|\xi(0)\|_{
			\L^2}+\int_0^{t_n}\left\|\xi^{\prime}\right\|_{\L^2} \mathrm{~d}t.
	\end{align}
	For the first term we have 
	\begin{align}
		\|\Pi \xi-\xi\|_{J_n}^2=\max _{j=1}^n\left(\|\Pi \xi-\xi\|_{I_j}^2\right) \leq C \max _{j=1}^n\left(k_j \int_{t_{j-1}}^{t_j}\left\|\xi^{\prime}\right\|_{\L^2}^2 d t\right),
	\end{align}
	where we have used approximation result for $\Pi\xi-\xi$ from Theorem \ref{5.th1} for $q_n=0$. Further an application of Cauchy-Schwarz inequality in \eqref{5.107.1} yields
	\begin{align}
		\|\Pi \xi\|_{J_n} \leq C\left(\|\xi(0)\|_{
			\L^2}+\int_0^{t_n}\left\|\xi^{\prime}\right\|_{\L^2} \mathrm{~d}t\right).
	\end{align}
Finally, from the approximation property \eqref{5.7a1.1}, we have 
	\begin{align}
		\|\xi(0)\|_{\L^2}^2+\int_0^{t_n}\left\|\xi^{\prime}\right\|_{\L^2}^2 d t \leq C \frac{h_K^{2s_K }}{r_K^{2m_K}}\left\|u_0\right\|_{\H^{m_K}(\kappa)}^2+C\frac{h_K^{2s_K }}{r_K^{2m_K}} \int_0^{t_n} \left\|u^{\prime}\right\|_{\H^{m_K}(\kappa)}  \mathrm{~d}t,
	\end{align}
	and hence the result.
\end{proof} 
For uniform parameters $k$, $ p $, and $q$ (i.e., $k_j = k, p_j = p$, and $q_j = q$), the bounds in Theorem \ref{5.th15} result in the following error estimates
\begin{corollary}\label{5.col2.1}
	For $1 \leq n \leq N$, we have
	\begin{align}
		&\int_0^{t_n}\left\|U_h^{DG}-u\right\|^2 d t+\left\|\left(U_h^{DG}-u\right)^n\right\|^2 \\& \leq C  \sum_{K \in \mathcal{T}} \left( \frac{h^{2s- 2}}{r^{2m-3}}\right)\left(\|u_0\|_{\H^{m}(K)}^2 + \int_0^{t_n} \|u^{\prime}\|_{\H^{m}(K)} \, \mathrm{d}t + \int_0^{t_n} \| u \|_{\H^{m}(K)} \, \mathrm{d}t \right)+C \frac{k^{2 \min \{p, q\}}}{p^{2 q}} \int_0^{t_n}\|u^{q_j+1}\|_{\H^{m_K}(K)}^2,
	\end{align}
	and 
	\begin{align}
		\left\|U_h^{DG}-u\right\|_{J_n}^2 & \leq C \left( \frac{h^{2s- 2}}{r^{2m-3}}\right) \log (p+2)\left(\|u_0\|_{\H^{m}(K)}^2 + \int_0^{t_n} \|u^{\prime}\|_{\H^{m}(K)} \, \mathrm{d}t + \int_0^{t_n} \| u \|_{\H^{m}(K)} \, \mathrm{d}t \right)\\
		& +C \frac{k^{2 \min \{p, q\}}}{p^{2 q}}\left(\max _{j=1}^n \max _{t \in I_j}|\!|\!|u^{q_j+1}|\!|\!|_{DG}^2+\log (p+2) \int_0^{t_n}\|u^{q_j+1}\|_{\H^{m_K}(K)}^2 \mathrm{~d} t\right) .
	\end{align}
\end{corollary}
\begin{proof}
	These estimates follow readily from Theorem \ref{5.th15.1}  and the fact that $\Gamma_{p, q}$ behaves like $p^{-2 q}$ for $p \rightarrow \infty$. 
\end{proof}
\begin{remark}
	The estimates in Corollary \ref{5.col2.1} demonstrate that the discrete scheme converges either as $k, h \to 0$ or as $p, r \to \infty$. It is noted that the first estimate is optimal with respect to the three parameters $k$, $h$, $p$, and suboptimal in $r$, while the second estimate is one power short of being optimal in $p$ and suboptimal in $r$. For smooth solutions $u$, spectral convergence rates are achieved by increasing the polynomial degrees $p$ and $r$ on fixed partitions.
\end{remark}
\section{Numerical studies}\label{5.sec4} 
To validate the theoretical findings outlined in the preceding sections, numerical computations were conducted using the open-source finite element library \texttt{FEniCS} \cite{ABJ, LHL}. This section details the execution of these calculations and compares the resultant numerical outcomes with theoretical predictions.

In all the examples presented, Discontinuous Galerkin (DG) methods are used for time discretization, while both  Finite Element Method and DGFEM are employed for spatial discretization. The domain $\Omega$ is divided into a mesh of size $N \times N$, which corresponds to subdividing the domain into $N \times N$ rectangles, each further divided into triangles. The temporal domain is uniformly partitioned into $N$ intervals. For time approximation with a polynomial degree $k$, this results in a system of $k$ equations, as discussed in Remark \ref{5.rem2}. The rate of convergence $r$ is defined by:
$$
r = \frac{\log(\|e_1\| / \|e_2\|)}{\log(h_1 / h_2)},
$$
where $e_1$ and $e_2$ represent errors corresponding to discretization parameters $h_1$ and $h_2$, respectively, with $h_2$ being a refinement of $h_1$.

These computations provide a rigorous validation of our theoretical models against numerical results, highlighting the efficacy of the methods employed.

\subsection{Accuracy verification}
In this section, we examine the generalized Burgers-Huxley equation (GBHE) with a weakly singular kernel, defined by:
\begin{align}\label{5.1.6}
	\frac{\partial u}{\partial t}+\alpha u^{\delta}\sum_{i=1}^d\frac{\partial u}{\partial x_{i}}-\nu\Delta u-\eta \int_{0}^{t} \frac{1}{\sqrt{t-s}}\Delta u(s)\mathrm{~d}s=\beta u(1-u^{\delta})(u^{\delta}-\gamma)+f,
\end{align}
on the domain $\Omega = [0,1]^2$ with $T = 1$, the parameters used are $\alpha = \beta = \delta = \nu = 1$ and $\gamma = 0.5$. The relaxation parameter (or memory coefficient) $\eta$ is taken to be $0$ or $1$ to verify whether optimal convergence rates are achieved in both cases. The forcing term $f$ is determined using the closed-form solution. The kernel in \eqref{5.1.6} is a weakly singular kernel and belongs to $\L^1(0,T)$.

To verify the accuracy of the proposed scheme, we consider two test solutions:
\begin{align*}
	\text{Sol. 1 : }u = (t^3-t^2+1)\sin(\pi x)\sin(\pi y),  \qquad  \quad\qquad 	\text{Sol. 2 : }u = t\sqrt{t}\sin(2\pi x)\sin(2\pi y).
\end{align*}
For Sol. 2, it is noteworthy that $u_t \in \L^2(0,T;\L^2)$, but $u_{tt} \notin \L^2(0,T; \L^2)$.
To validate the theoretical results, we first approximate time using a constant and perform spatial discretization with linear polynomials. This approach yields first-order convergence in space for both CG and DG methods for the case without memory $\eta = 0$ and with memory $\eta = 1$, as shown in Tables \ref{table5.1} and \ref{table5.2}. Further, Tables \ref{table5.3} and \ref{table5.4} demonstrate optimal order convergence in both the $\L^2$ and $\H^1$ norms under conforming and DG settings for a first-degree approximation in time and a second-degree approximation in space with $\eta =1$. Higher-order results are discussed in Tables \ref{table5.5} and \ref{table5.6} for both solution choices. The proposed scheme is also applicable in three dimensions for various examples, with error tables reported for the solution $u = t\sqrt{t}\sin(2\pi x)\sin(2\pi y)\sin(2\pi z)$  in Tables \ref{table5.11} and \ref{table5.14} for $\eta =1$.

\begin{table}
	\begin{center}
		{\small
			\caption{Errors and convergence rates for the numerical solutions $u_h$ relative to the exact solution $ u(x,y,t)=(t^3-t^2+1)\sin(\pi x)\sin(\pi y)$ with DG time stepping with $0$ degree. }
			\label{table5.1}
			\begin{tabular}{| c | c | c | c | c | c| c  |}
				\hline
				\multicolumn{ 7}{|c|}{Error history in 2D with weakly singular kernel, $ u=(t^3-t^2+1)\sin(\pi x)\sin(\pi y)$ }\\
				\hline
				\multirow{7}{*}{$hp$-FEM}&{Mesh}&{Dof}&{$\H^1$-error for $\eta =0$}&{$O(h)$}&{$\H^1$-error for $\eta = 1$}&{$O(h)$}\\
				\cline{2- 7}
				&{$2\times 2$} &$9$ &$1.46(00)$ &$-$&$1.51(00)$ &$-$ \\
				\cline{2- 7}
				&{$4\times 4$}&$25$ &$7.91(-01)$ &$ 0.88$ &$8.21(-01)$ &$ 0.88$ \\
				\cline{2- 7}
				&{$8\times 8$}&$81$ &$4.02(-01)$ &$0.98$ &$4.20(-01)$ &$0.97$ \\
				\cline{2- 7}
				&{$16\times 16$}&$289$ &$ 2.02(-01)$ &$ 0.99$ &$ 2.12(-01)$ &$0.99$ \\
				\cline{2- 7}
				&{$32\times 32$} &$1089$&$ 1.01(-01)$ &$0.99$&$ 1.07(-01)$ &$0.99$ \\
				\cline{2- 7}
				&{$64\times 64$} &$4225$&$ 5.04(-02)$ &$ 1.00$ &$ 5.38(-02)$ &$ 0.99$ \\
				\cline{2- 7}
				&{$128\times 128$} &$16641$&$ 2.52(-02)$ &$ 1.00$ &$2.70(-02)$ &$0.99$ \\
				
				\cline{1- 7}
				\multicolumn{ 7}{|c|}{Error history in 2D with weakly singular kernel,$ u=(t^3-t^2+1)\sin(\pi x)\sin(\pi y)$ }\\
				\hline
				\multirow{7}{*}{$hp$-DGFEM}&{Mesh}&{Dof}&{$\H^1$-error for $\eta =0$}&{$O(h)$}&{$\H^1$-error for $\eta = 1$}&{$O(h)$}\\
				\cline{2- 7}
				
				&{$2\times 2$}&$24$&$1.46(00)$ &$-$&$1.73(00)$ &$-$ \\
				\cline{2- 7}
				&{$4\times 4$} &$96$&$7.90(-01)$ &$ 0.89$ &$8.57(-01)$ &$ 1.02$ \\
				\cline{2- 7}
				&{$8\times 8$} &$384$&$4.01(-01)$ &$0.98$ &$4.25(-01)$ &$1.01$ \\
				\cline{2- 7}
				&{$16\times 16$}&$1536$ &$ 2.01(-01)$ &$ 1.00$ &$ 2.13(-01)$ &$1.00$ \\
				\cline{2- 7}
				&{$32\times 32$}&$6144$ &$ 1.01(-01)$ &$0.99$&$1.07(-01) $ &$0.99$ \\
				\cline{2- 7}
				&{$64\times 64$} &$24576$&$ 5.04(-02)$ &$ 1.00$ &$ 5.40(-02)$ &$0.99$ \\
				\cline{2- 7}
				&{$128\times 128$}&$98304$ &$ 2.52(-02)$ &$ 1.00$ &$ 2.73(-02)$ &$ 0.98$ \\
				\cline{1- 7}
		\end{tabular}}
	\end{center}
\end{table}
\begin{table}
	\begin{center}
		{\small
			\caption{Errors and convergence rates for the numerical solutions $u_h$ relative to the exact solution $ u(x,y,t)=(t^3-t^2+1)\sin(\pi x)\sin(\pi y)$ with DG time stepping of $1$ degree. }
			\label{table5.3}
			\begin{tabular}{| c | c | c | c | c | c| c  |}
				\hline
				\multicolumn{ 7}{|c|}{Error history in 2D with weakly singular kernel, $ u=(t^3-t^2+1)\sin(\pi x)\sin(\pi y)$ }\\
				\hline
			\multirow{7}{*}{$hp$-FEM}&{Mesh}&{Dof}&{$\L^2$-error}&{$O(h^3)$}&{$\H^1$-error}&{$O(h^2)$}\\
			\cline{2-7}
			&{$2\times 2$} &$50$ &$5.18(-02)$ &$-$&$4.50(-01)$ &$-$ \\
			\cline{2-7}
			&{$4\times 4$}&$162$ &$7.97(-03)$ &$ 2.70$ &$1.21(-01)$ &$ 1.89$ \\
			\cline{2-7}
			&{$8\times 8$}&$578$ &$1.11(-03)$ &$2.84$ &$3.08(-02)$ &$1.97$ \\
			\cline{2-7}
			&{$16\times 16$}&$2178$ &$1.46(-04)$ &$2.93$ &$7.74(-03)$ &$1.99$ \\
			\cline{2-7}
			&{$32\times 32$} &$8450$&$1.87(-05)$ &$2.96$&$1.94(-03)$ &$2.00$ \\
			\cline{2-7}
			&{$64\times 64$} &$33282$&$2.34(-06)$ &$3.00$ &$4.84(-04)$ &$2.00$ \\
			\cline{2-7}
			&{$128\times 128$} &$132098$&$2.91(-07)$ &$3.01$ &$1.21(-04)$ &$2.00$ \\
			
				\cline{1- 7}
			\multicolumn{7}{|c|}{Error history in 2D with weakly singular kernel, $u=(t^3-t^2+1)\sin(\pi x)\sin(\pi y)$}\\
			\hline
			\multirow{7}{*}{$hp$-DGFEM}&{Mesh}&{Dof}&{$\L^2$-error}&{$O(h^3)$}&{$\H^1$- error}&{$O(h^2)$}\\
			\cline{2-7}
			&{$2\times 2$}&$96$&$4.65(-02)$ &$-$&$4.48(-01)$ &$-$ \\
			\cline{2-7}
			&{$4\times 4$} &$384$&$6.95(-03)$ &$2.74$ &$1.20(-01)$ &$1.90$ \\
			\cline{2-7}
			&{$8\times 8$} &$1536$&$9.66(-04)$ &$2.85$ &$3.07(-02)$ &$1.97$ \\
			\cline{2-7}
			&{$16\times 16$}&$6144$ &$1.30(-04)$ &$2.89$ &$7.71(-03)$ &$1.99$ \\
			\cline{2-7}
			&{$32\times 32$}&$24576$ &$1.70(-05)$ &$2.93$&$1.93(-03)$ &$2.00$ \\
			\cline{2-7}
			&{$64\times 64$} &$98304$&$2.18(-06)$ &$2.96$ &$4.89(-04)$ &$1.98$ \\
				\cline{1- 7}
		\end{tabular}}
	\end{center}
\end{table}
\begin{table}
	\begin{center}
		{\small
			\caption{Errors and convergence rates for the numerical solutions $u_h$ relative to the exact solution $ u (x,y,t) =(t^3-t^2+1)\sin(\pi x)\sin(\pi y)$ with DG time stepping of $2$ degree. }
			\label{table5.5}
			\begin{tabular}{| c | c | c | c | c | c| c  |}
				\hline
				\multicolumn{ 7}{|c|}{Error history in 2D with weakly singular kernel,  $ u=(t^3-t^2+1)\sin(\pi x)\sin(\pi y)$ }\\
				\hline
				\multirow{7}{*}{$hp$-FEM}&{Mesh}&{Dof}&{$\L^2$-error}&{$O(h^4)$}&{$\H^1$-error}&{$O(h^3)$}\\
				\cline{2-7}
				&{$2\times 2$}&$147$&$7.47(-03)$ &$-$&$7.96(-02)$ &$-$ \\
				\cline{2-7}
				&{$4\times 4$}&$507$ &$4.71(-04)$ &$3.99$ &$8.58(-03)$ &$3.21$ \\
				\cline{2-7}
				&{$8\times 8$} &$1875$&$2.47(-05)$ &$4.25$ &$9.63(-04)$ &$3.16$ \\
				\cline{2-7}
				&{$16\times 16$} &$7203$&$1.33(-06)$ &$4.22$ &$1.13(-04)$ &$3.09$ \\
				\cline{2-7}
				&{$32\times 32$} &$28227$&$7.78(-08)$ &$4.10$&$1.37(-05)$ &$3.04$ \\
				\cline{2-7}
				&{$64\times 64$}&$111747$ &$4.74(-09)$ &$4.04$ &$1.69(-06)$ &$3.02$ \\
				\cline{1-7}
				\multirow{7}{*}{$hp$-DGFEM}&{Mesh}&{Dof}&{$\L^2$-error}&{$O(h^4)$}&{$\H^1$-error}&{$O(h^3)$}\\
				\cline{2-7}
				&{$2\times 2$}&$240$&$7.38(-03)$ &$-$&$3.55(-01)$ &$-$ \\
				\cline{2-7}
				&{$4\times 4$} &$960$&$4.68(-04)$ &$3.98$ &$4.73(-02)$ &$2.91$ \\
				\cline{2-7}
				&{$8\times 8$}&$3840$ &$2.46(-05)$ &$4.25$ &$5.81(-03)$ &$3.03$ \\
				\cline{2-7}
				&{$16\times 16$} &$15360$&$1.33(-06)$ &$4.21$ &$7.14(-04)$ &$3.02$ \\
				\cline{2-7}
				&{$32\times 32$}&$61440$ &$7.74(-08)$ &$4.10$&$8.84(-05)$ &$3.01$ \\
				\cline{1-7}
				
		\end{tabular}}
	\end{center}
\end{table}

\begin{table}
	\begin{center}
		{\small
			\caption{Errors and convergence rates for the numerical solutions $u_h$ relative to the exact solution $ u(x,y,t) = t\sqrt{t}\sin(2\pi x)\sin(2\pi y)$ with DG time stepping of $0$ degree. }
			\label{table5.2}
			\begin{tabular}{| c | c | c | c | c | c| c  |}
				\hline
				\multicolumn{ 7}{|c|}{Error history in 2D with weakly singular kernel,  $ u=t\sqrt{t}\sin(2\pi x)\sin(2\pi y)$ }\\
				\hline
				\multirow{7}{*}{$hp$-FEM}&{Mesh}&{Dof}&{$\H^1$-error for $\eta =0$}&{$O(h)$}&{$\H^1$-error for $\eta = 1$}&{$O(h)$}\\
				\cline{2- 7}
				&{$2\times 2$} &$9$ &$3.12(00)$ &$-$&$3.11(00)$ &$-$ \\
				\cline{2- 7}
				&{$4\times 4$}&$25$ &$1.94(00)$ &$ 0.69$ &$1.93(00)$ &$ 0.69$ \\
				\cline{2- 7}
				&{$8\times 8$}&$81$ &$9.86(-01)$ &$0.98$ &$9.81(-01)$ &$0.98$ \\
				\cline{2- 7}
				&{$16\times 16$}&$289$ &$ 4.82(-01)$ &$ 1.03$ &$ 4.80(-01)$ &$1.03$ \\
				\cline{2- 7}
				&{$32\times 32$} &$1089$&$ 2.36(-01)$ &$1.03$&$ 2.35(-01)$ &$1.03$ \\
				\cline{2- 7}
				&{$64\times 64$} &$4225$&$ 1.17(-01)$ &$ 1.01$ &$ 1.16(-01)$ &$ 1.02$ \\
					\cline{2- 7}
				&{$128\times 128$} &$16641$&$ 5.80(-02)$ &$ 1.01$ &$ 5.79(-02)$ &$ 1.00$ \\
				
				\cline{1- 7}
					\multicolumn{ 7}{|c|}{Error history in 2D with weakly singular kernel,  $ u=t\sqrt{t}\sin(2\pi x)\sin(2\pi y)$ }\\
				\hline
				\multirow{7}{*}{$hp$-DGFEM}&{Mesh}&{Dof}&{$\H^1$-error for $\eta =0$}&{$O(h)$}&{$\H^1$-error for $\eta = 1$}&{$O(h)$}\\
				\cline{2- 7}
				
				&{$2\times 2$}&$24$&$3.11(00)$ &$-$&$3.07(00)$ &$-$ \\
				\cline{2- 7}
				&{$4\times 4$} &$96$&$1.92(00)$ &$ 0.70$ &$1.87(00)$ &$ 0.72$ \\
				\cline{2- 7}
				&{$8\times 8$} &$384$&$9.80(-01)$ &$0.97$ &$9.40(-01)$ &$0.99$ \\
				\cline{2- 7}
				&{$16\times 16$}&$1536$ &$ 4.79(-01)$ &$ 1.03$ &$ 4.57(-01)$ &$1.04$ \\
				\cline{2- 7}
				&{$32\times 32$}&$6144$ &$ 2.35(-01)$ &$1.03$&$2.24(-01) $ &$1.03$ \\
				\cline{2- 7}
				&{$64\times 64$} &$24576$&$ 1.16(-01)$ &$ 1.02$ &$ 1.11(-01)$ &$ 1.01$ \\
				\cline{2- 7}
				&{$128\times 128$}&$98304$ &$ 5.76(-02)$ &$ 1.01$ &$ 5.58(-02)$ &$ 0.99$ \\
				\cline{1- 7}
		\end{tabular}}
	\end{center}
\end{table}

\begin{table}
	\begin{center}
		{\small
			\caption{Errors and convergence rates for the numerical solutions $u_h$ relative to the exact solution $ u(x,y,t) = t\sqrt{t}\sin(2\pi x)\sin(2\pi y)$ with DG time stepping of $1$ degree. }
				\label{table5.4}
			\begin{tabular}{| c | c | c | c | c | c| c  |}
				\hline
				\multicolumn{ 7}{|c|}{Error history in 2D with weakly singular kernel,  $ u=t\sqrt{t}\sin(2\pi x)\sin(2\pi y)$ }\\
					\cline{1-7}
			\multirow{7}{*}{$hp$-FEM}&{Mesh}&{Dof}&{$\L^2$-error}&{$O(h^3)$}&{$\H^1$-error}&{$O(h^2)$}\\
			\cline{2-7}
			&{$2\times 2$} &$50$ &$2.18(-01)$ &$-$&$2.15(00)$ &$-$ \\
			\cline{2-7}
			&{$4\times 4$}&$162$ &$3.46(-02)$ &$2.66$ &$5.78(-01)$ &$1.90$ \\
			\cline{2-7}
			&{$8\times 8$}&$578$ &$4.47(-03)$ &$2.95$ &$1.45(-01)$ &$2.00$ \\
			\cline{2-7}
			&{$16\times 16$}&$2178$ &$5.63(-04)$ &$2.99$ &$3.55(-02)$ &$2.03$ \\
			\cline{2-7}
			&{$32\times 32$} &$8450$&$7.05(-05)$ &$3.00$&$8.68(-03)$ &$2.03$ \\
			\cline{2-7}
			&{$64\times 64$}&$33282$&$8.79(-06)$ &$3.00$ &$2.14(-03)$ &$2.02$ \\
			\cline{2-7}
			&{$128\times 128$}&$132098$&$1.10(-06)$ &$3.00$ &$5.32(-04)$ &$2.01$ \\
			\cline{1-7}
			
			\multicolumn{ 7}{|c|}{Error history in 2D with weakly singular kernel,  $ u=t\sqrt{t}\sin(2\pi x)\sin(2\pi y)$ }\\
			\hline
			\multirow{7}{*}{$hp$-DGFEM}&{Mesh}&{Dof}&{$\L^2$-error}&{$O(h^3)$}&{$\H^1$-error}&{$O(h^2)$}\\
			\cline{2-7}
			&{$2\times 2$}&$96$&$4.65(-02)$ &$-$&$4.74(-01)$ &$-$ \\
			\cline{2-7}
			&{$4\times 4$} &$384$&$6.95(-03)$ &$2.74$ &$1.30(-01)$ &$1.87$ \\
			\cline{2-7}
			&{$8\times 8$} &$1536$&$9.66(-04)$ &$2.85$ &$3.34(-02)$ &$1.96$ \\
			\cline{2-7}
			&{$16\times 16$}&$6144$ &$1.30(-04)$ &$2.89$ &$8.40(-03)$ &$1.99$ \\
			\cline{2-7}
			&{$32\times 32$}&$24576$ &$1.70(-05)$ &$2.93$&$2.10(-03)$ &$2.00$ \\
			\cline{2-7}
			&{$64\times 64$} &$98304$&$2.18(-06)$ &$2.96$ &$5.26(-04)$ &$2.00$ \\
			\hline
				\cline{1-7}
				
		\end{tabular}}
	\end{center}
\end{table}

\begin{table}
	\begin{center}
		{\small
			\caption{Errors and convergence rates for the numerical solutions $u_h$ relative to the exact solution $ u = t\sqrt{t}\sin(2\pi x)\sin(2\pi y)$ with DG time stepping of $2$ degree. }
				\label{table5.6}
			\begin{tabular}{| c | c | c | c | c | c| c  |}
				\hline
				\multicolumn{ 7}{|c|}{Error history in 2D with weakly singular kernel,  $ u =t\sqrt{t}\sin(2\pi x)\sin(2\pi y)$ }\\
				\hline
				\multirow{7}{*}{$hp$-FEM}&{Mesh}&{Dof}&{$\L^2$-error}&{$O(h^4)$}&{$\H^1$-error}&{$O(h^3)$}\\
				\cline{2-7}
				&{$2\times 2$}&$147$&$8.77(-02)$ &$-$&$1.11(00)$ &$-$ \\
				\cline{2-7}
				&{$4\times 4$}&$507$ &$5.45(-03)$ &$4.01$ &$1.26(-01)$ &$3.14$ \\
				\cline{2-7}
				&{$8\times 8$} &$1875$&$3.29(-04)$ &$4.05$ &$1.48(-02)$ &$3.09$ \\
				\cline{2-7}
				&{$16\times 16$} &$7203$&$1.97(-05)$ &$4.06$ &$1.75(-03)$ &$3.08$ \\
				\cline{2-7}
				&{$32\times 32$} &$28227$&$1.20(-06)$ &$4.04$&$2.12(-04)$ &$3.05$ \\
				\cline{2-7}
				&{$64\times 64$}&$111747$ &$7.46(-08)$ &$4.01$ &$2.60(-05)$ &$3.03$ \\
				\cline{1-7}
				

				\cline{1- 7}
				\multirow{7}{*}{$hp$-DGFEM}&{Mesh}&{Dof}&{$\L^2$-error}&{$O(h^4)$}&{$\H^1$-error}&{$O(h^3)$}\\
				\cline{2-7}
				&{$2\times 2$}&$240$&$8.67(-02)$ &$-$&$5.73(00)$ &$-$ \\
				\cline{2-7}
				&{$4\times 4$} &$960$&$5.37(-03)$ &$4.01$ &$7.27(-01)$ &$2.98$ \\
				\cline{2-7}
				&{$8\times 8$} &$3840$ &$3.26(-04)$ &$4.04$ &$9.02(-02)$ &$3.01$ \\
				\cline{2-7}
				&{$16\times 16$} &$15360$&$1.95(-05)$ &$4.06$ &$1.10(-02)$ &$3.04$ \\
				\cline{2-7}
				&{$32\times 32$}&$61440$ &$1.20(-06)$ &$4.02$&$1.36(-03)$ &$3.02$ \\
				\cline{1-7}
				
		\end{tabular}}
	\end{center}
\end{table}

\begin{table}
	\begin{center}
		{\small
			\caption{Errors and convergence rates for the numerical solutions $u_h$ relative to the exact solution $ u = t\sqrt{t}\sin(2\pi x)\sin(2\pi y)\sin(2\pi z)$ with DG time stepping of $0$ degree. }
			\label{table5.11}
			\begin{tabular}{| c | c | c | c | c |}
				\hline
				\multicolumn{ 5}{|c|}{Error history in 3D with weakly singular kernel }\\
				\hline
				\multirow{7}{*}{$hp$-FEM}&{Mesh}&{Dof}&{$\H^1$-error}&{$O(h)$}\\
				\cline{2- 5}
				&{$2\times 2$}&$27$&$3.04(00)$ &$-$ \\
				\cline{2- 5}
				&{$4\times 4$}&$125$ &$1.94(00)$ &$ 0.65$ \\
				\cline{2- 5}
				&{$8\times 8$} &$729$&$1.05(00)$ &$0.89$ \\
				\cline{2- 5}
				&{$16\times 16$} &$4913$&$ 5.23(-01)$ &$1.00$ \\
				\cline{2- 5}
				&{$32\times 32$} &$35937$&$ 2.58(-01)$ &$1.02$ \\
				\cline{2- 5}
				&{$64\times 64$}&$274625$ &$ 1.28(-01$ &$ 1.02$ \\
				\cline{1- 5}
				\multirow{7}{*}{$hp$-DGFEM}&{Mesh}&{Dof}&{$\H^1$-error}&{$O(h)$}\\
				\cline{2- 5}
				&{$2\times 2$}&$192$&$3.03(00)$ &$-$ \\
				\cline{2- 5}
				&{$4\times 4$} &$1536$&$1.90(00)$ &$0.67$ \\
				\cline{2- 5}
				&{$8\times 8$}&$12288$ &$1.02(00)$ &$0.90$ \\
				\cline{2- 5}
				&{$16\times 16$} &$98304$&$ 5.07(-01)$ &$1.01$ \\
				\cline{2- 5}
				&{$32\times 32$}&$786432$ &$2.50(-01)$ &$1.02$ \\
				\cline{1- 5}
		\end{tabular}}
	\end{center}
\end{table}
\begin{table}
	\begin{center}
		{\small
			\caption{Errors and convergence rates for the numerical solutions $u_h$ relative to the exact solution $ u = t\sqrt{t}\sin(2\pi x)\sin(2\pi y)\sin(2\pi z)$ with DG time stepping of $1$ degree. }
			\label{table5.14}
		\begin{tabular}{| c | c | c | c | c | c| c|}
			\hline
			\multicolumn{7}{|c|}{Error history in 3D with weakly singular kernel.} \\
			\hline
			\multirow{4}{*}{$hp$-FEM} & {Mesh} & {Dof} & {$\L^2$-error} & {$O(h^3)$} & {$\H^1$-error} & {$O(h^2)$} \\
			\cline{2-7}
			& {$2\times 2$} & $250$ & $1.52(-01)$ & $-$ & $1.94(00)$ & $-$ \\
			\cline{2-7}
			& {$4\times 4$} & $1458$ & $4.54(-02)$ & $1.75$ & $7.00(-01)$ & $1.47$ \\
			\cline{2-7}
			& {$8\times 8$} & $9826$ & $5.82(-03)$ & $2.96$ & $1.88(-01)$ & $1.89$ \\
			\cline{2-7}
			& {$16\times 16$} & $71874$ & $7.21(-04)$ & $3.01$ & $4.77(-02)$ & $1.98$ \\
			\hline
			\multirow{4}{*}{$hp$-DGFEM} & {Mesh} & {Dof} & {$\L^2$-error} & {$O(h^3)$} & {$\H^1$-error} & {$O(h^2)$} \\
			\cline{2-7}
			& {$2\times 2$} & $960$ & $1.52(-01)$ & $-$ & $1.93(00)$ & $-$ \\
			\cline{2-7}
			& {$4\times 4$} & $7680$ & $4.54(-02)$ & $1.74$ & $7.00(-01)$ & $1.48$ \\
			\cline{2-7}
			& {$8\times 8$} & $61440$ & $5.82(-03)$ & $2.96$ & $1.88(-01)$ & $1.89$ \\
			\cline{2-7}
			& {$16\times 16$} & $491520$ & $7.19(-04)$ & $3.02$ & $4.77(-02)$ & $1.98$ \\
			\hline
		\end{tabular}
	}
	\end{center}
\end{table}
\subsection{Application to Fractional differential equations}
In this example, we perform accuracy verification for the time-fractional generalized Burgers-Huxley equation given by:
\begin{align}\label{5.1.7}
	\partial_t^{\mu} u+\alpha u^{\delta}\sum_{i=1}^d\frac{\partial u}{\partial x_{i}}-\nu\Delta u-\eta \int_{0}^{t} \frac{1}{\sqrt{t-s}}\Delta u(s)\mathrm{~d}s=\beta u(1-u^{\delta})(u^{\delta}-\gamma)+f,
\end{align}
where $\partial_t^\mu$ denotes the left-sided Caputo fractional derivative (\cite[Pg 81]{KST} and \cite{CCH}) of order $\mu \ge 0$ with respect to $t$, defined as
$$
\partial_t^{\mu} u(t) = \frac{1}{\Gamma(1-\mu)}\int_0^t \frac{1}{(t-\tau)^{\mu}} \frac{\mathrm{du(\tau)}}{\mathrm{d}\tau} \mathrm{d}\tau,
$$
where $\Gamma(\cdot)$ represents the Gamma function. Specifically, we consider $\mu = \frac{1}{2}$. For the discrete formulation corresponding to the fractional derivative term, we refer to \cite{CCH}. 

Tables \ref{table5.7} to \ref{table5.9} present the error and convergence rates using the exact solution defined in Sol. 2 for the case with memory ($\eta =1$) and without memory($\eta=0$). In Table \ref{table5.7}, we approximate the time derivative with a constant and the spatial derivative with a linear polynomial. The error converges with optimal first-order accuracy using both $hp$-FEM and $hp$-DGFEM in space. Table \ref{table5.8} shows the results when approximating time with a linear polynomial (degree 1) and space with a quadratic polynomial, achieving optimal second-order convergence. Finally, increasing the polynomial degree by one results in higher-order convergence. Hence, the proposed scheme also effectively handles the time-fractional generalized Burgers-Huxley equation with a weakly singular kernel. For three dimensions the results are illustrate in Table \ref{table5.9} and \ref{table5.13}.
\begin{table}
	\begin{center}
		{\small
			\caption{Errors and convergence rates for the numerical solutions $u_h$ relative to the exact solution $ u = t\sqrt{t}\sin(2\pi x)\sin(2\pi y)$ using DG time stepping with $0$ degree for fractional equation. }
				\label{table5.7}
			\begin{tabular}{| c | c | c | c | c | c| c  |}
				\hline
				\multicolumn{ 7}{|c|}{Error history in 2D with weakly singular kernel,  $ u=t\sqrt{t}\sin(2\pi x)\sin(2\pi y)$ }\\
				\hline
				\multirow{7}{*}{$hp$-FEM}&{Mesh}&{Dof}&{$\H^1$-error for $\eta =0$}&{$O(h)$}&{$\H^1$-error for $\eta = 1$}&{$O(h)$}\\
				\cline{2- 7}
				&{$2\times 2$} &$9$ &$3.12(00)$ &$-$&$3.11(00)$ &$-$ \\
				\cline{2- 7}
				&{$4\times 4$}&$25$ &$1.93(00)$ &$ 0.69$ &$1.93(00)$ &$ 0.69$ \\
				\cline{2- 7}
				&{$8\times 8$}&$81$ &$9.80(-01)$ &$0.98$ &$9.79(-01)$ &$0.98$ \\
				\cline{2- 7}
				&{$16\times 16$}&$289$ &$ 4.76(-01)$ &$ 1.04$ &$ 4.77(-01)$ &$1.04$ \\
				\cline{2- 7}
				&{$32\times 32$} &$1089$&$ 2.31(-01)$ &$1.05$&$ 2.33(-01)$ &$1.04$ \\
				\cline{2- 7}
				&{$64\times 64$} &$4225$&$ 1.12(-01)$ &$ 1.04$ &$ 1.14(-01)$ &$ 1.03$ \\
				\cline{2- 7}
				&{$128\times 128$} &$16641$&$ 5.51(-02)$ &$ 1.02$ &$ 5.56(-02)$ &$ 1.03$ \\
				
				\cline{1- 7}
				\multicolumn{ 7}{|c|}{Error history in 2D with weakly singular kernel,  $ u=t\sqrt{t}\sin(2\pi x)\sin(2\pi y)$ }\\
				\hline
				\multirow{7}{*}{$hp$-DGFEM}&{Mesh}&{Dof}&{$\H^1$-error for $\eta =0$}&{$O(h)$}&{$\H^1$-error for $\eta = 1$}&{$O(h)$}\\
				\cline{2- 7}
				
				&{$2\times 2$}&$24$&$3.12(00)$ &$-$&$3.08(00)$ &$-$ \\
				\cline{2- 7}
				&{$4\times 4$} &$96$&$1.93(00)$ &$ 0.69$ &$1.88(00)$ &$ 0.71$ \\
				\cline{2- 7}
				&{$8\times 8$} &$384$&$9.80(-01)$ &$0.98$ &$9.46(-01)$ &$0.99$ \\
				\cline{2- 7}
				&{$16\times 16$}&$1536$ &$ 4.76(-01)$ &$ 1.04$ &$ 4.59(-01)$ &$1.04$ \\
				\cline{2- 7}
				&{$32\times 32$}&$6144$ &$ 2.30(-01)$ &$1.05$&$2.24(-01) $ &$1.04$ \\
				\cline{2- 7}
				&{$64\times 64$} &$24576$&$ 1.12(-01)$ &$ 1.04$ &$ 1.11(-01)$ &$ 1.02$ \\
				\cline{2- 7}
				&{$128\times 128$}&$98304$ &$ 5.51(-02)$ &$ 1.02$ &$ 5.60(-02)$ &$ 0.98$ \\
				\cline{1- 7}
		\end{tabular}}
	\end{center}
\end{table}

\begin{table}
	\begin{center}
		{\small
			\caption{Errors and convergence rates for the numerical solutions $u_h$ relative to the exact solution $ u(x,y,t) = t\sqrt{t}\sin(2\pi x)\sin(2\pi y)$ with DG time stepping of $1$ degree for fractional equation. }
				\label{table5.8}
			\begin{tabular}{| c | c | c | c | c | c| c  |}
				\hline
				\multicolumn{ 7}{|c|}{Error history in 2D with weakly singular kernel,  $ u=t\sqrt{t}\sin(2\pi x)\sin(2\pi y)$ }\\
				\hline
				\multirow{7}{*}{$hp$-FEM}&{Mesh}&{Dof}&{$\H^1$-error for $\eta =0$}&{$O(h^2)$}&{$\H^1$-error for $\eta = 1$}&{$O(h^2)$}\\
				\cline{2- 7}
				&{$2\times 2$} &$50$ &$2.15(-00)$ &$-$&$2.15(00)$ &$-$ \\
				\cline{2- 7}
				&{$4\times 4$}&$162$ &$5.77(-01)$ &$ 1.89$ &$5.78(-01)$ &$ 1.89$ \\
				\cline{2- 7}
				&{$8\times 8$}&$578$ &$1.45(-01)$ &$1.99$ &$1.45(-01)$ &$1.99$ \\
				\cline{2- 7}
				&{$16\times 16$}&$2178$ &$ 3.55(-02)$ &$ 2.03$ &$ 3.55(-02)$ &$2.04$ \\
				\cline{2- 7}
				&{$32\times 32$} &$8450$&$ 8.69(-03)$ &$2.03$&$ 8.69(-03)$ &$2.03$ \\
				\cline{2- 7}
				&{$64\times 64$} &$33282$&$ 2.14(-03)$ &$ 2.02$ &$ 2.14(-03)$ &$ 2.02$ \\
				\cline{2- 7}
				&{$128\times 128$} &$132098$&$ 5.34(-04)$ &$ 2.00$ &$ 5.33-04)$ &$ 2.01$ \\
				
				\cline{1- 7}
				\multicolumn{ 7}{|c|}{Error history in 2D with weakly singular kernel,  $ u=t\sqrt{t}\sin(2\pi x)\sin(2\pi y)$ }\\
				\hline
				\multirow{7}{*}{$hp$-DGFEM}&{Mesh}&{Dof}&{$\H^1$-error for $\eta =0$}&{$O(h^2)$}&{$\H^1$-error for $\eta = 1$}&{$O(h^2)$}\\
				\cline{2- 7}
				
				&{$2\times 2$}&$96$&$2.13(00)$ &$-$&$2.14(00)$ &$-$ \\
				\cline{2- 7}
				&{$4\times 4$} &$384$&$5.74(-01)$ &$ 1.89$ &$5.76(-01)$ &$ 1.89$ \\
				\cline{2- 7}
				&{$8\times 8$} &$1536$&$1.44(-01)$ &$1.99$ &$1.44(-01)$ &$1.99$ \\
				\cline{2- 7}
				&{$16\times 16$}&$6144$ &$ 3.52(-02)$ &$ 2.03$ &$ 3.54(-02)$ &$2.03$ \\
				\cline{2- 7}
				&{$32\times 32$}&$24576$ &$ 8.62(-03)$ &$2.03$&$8.67(-03) $ &$2.03$ \\
				\cline{2- 7}
				&{$64\times 64$} &$98304$&$ 2.13(-03)$ &$ 2.02$ &$ 2.19(-03)$ &$ 1.99$ \\
				\cline{1- 7}
		\end{tabular}}
	\end{center}
\end{table}

\begin{table}
	\begin{center}
		{\small
			\caption{Errors and convergence rates for the numerical solutions $u_h$ relative to the exact solution $ u(x,y,t) = t\sqrt{t}\sin(2\pi x)\sin(2\pi y)$ with DG time stepping of $2$ degree for fractional equation. }
				\label{table5.9}
			\begin{tabular}{| c | c | c | c | c | c| c  |}
				\hline
				\multicolumn{ 7}{|c|}{Error history in 2D with weakly singular kernel,  $ u=t\sqrt{t}\sin(2\pi x)\sin(2\pi y)$ }\\
				\hline
				\multirow{7}{*}{$hp$-FEM}&{Mesh}&{Dof}&{$\H^1$-error for $\eta =0$}&{$O(h^3)$}&{$\H^1$-error for $\eta = 1$}&{$O(h^3)$}\\
				\cline{2- 7}
				&{$2\times 2$} &$147$ &$1.11(00)$ &$-$&$1.11(00)$ &$-$ \\
				\cline{2- 7}
				&{$4\times 4$}&$507$ &$1.26(-01)$ &$ 3.15$ &$1.26(-01)$ &$ 3.15$ \\
				\cline{2- 7}
				&{$8\times 8$}&$1875$ &$1.48(-02)$ &$3.09$ &$1.49(-02)$ &$3.09$ \\
				\cline{2- 7}
				&{$16\times 16$}&$7203$ &$ 1.75(-03)$ &$ 3.08$ &$ 1.75(-03)$ &$3.08$ \\
				\cline{2- 7}
				&{$32\times 32$} &$28227$&$ 2.12(-04)$ &$3.04$&$ 2.12(-04)$ &$3.05$ \\
				\cline{2- 7}
				&{$64\times 64$} &$11747$&$ 2.60(-05)$ &$ 3.02$ &$ 2.60(-04)$ &$ 3.02$ \\
	
				\cline{1- 7}
				\multicolumn{ 7}{|c|}{Error history in 2D with weakly singular kernel,  $ u=t\sqrt{t}\sin(2\pi x)\sin(2\pi y)$ }\\
				\hline
				\multirow{7}{*}{$hp$-DGFEM}&{Mesh}&{Dof}&{$\H^1$-error for $\eta =0$}&{$O(h^3)$}&{$\H^1$-error for $\eta = 1$}&{$O(h^3)$}\\
				\cline{2- 7}
				
					\cline{2- 7}
				&{$2\times 2$}&$240$&$5.73(00)$ &$-$&$5.73(00)$ &$-$ \\
				\cline{2- 7}
				&{$4\times 4$} &$960$&$7.27(-01)$ &$ 2.98$ &$7.27(-01)$ &$2.98$ \\
				\cline{2- 7}
				&{$8\times 8$}&$3840$ &$9.02(-02)$ &$3.01$ &$9.02(-02)$ &$3.01$ \\
				\cline{2- 7}
				&{$16\times 16$} &$15360$&$ 1.10(-02)$ &$3.02$ &$1.10(-02)$ &$3.03$ \\
				\cline{2- 7}
				&{$32\times 32$}&$61440$ &$ 1.36(-03)$ &$3.01$&$ 1.36(-03)$ &$3.02$ \\
				\cline{1- 7}
				\cline{1- 7}
		\end{tabular}}
	\end{center}
\end{table}

\begin{table}
	\begin{center}
		{\small
			\caption{Errors and convergence rates for the numerical solutions $u_h$ relative to the exact solution $ u = t\sqrt{t}\sin(2\pi x)\sin(2\pi y)\sin(2\pi z)$ with DG time stepping of $0$ degree  for fractional equation. }
			\label{table5.13}
			\begin{tabular}{| c | c | c | c | c |}
				\hline
				\multicolumn{ 5}{|c|}{Error history in 3D with weakly singular kernel. }\\
				\hline
				\multirow{7}{*}{$hp$-FEM}&{Mesh}&{Dof}&{$\H^1$-error for $\eta = 1$}&{$O(h)$}\\
				\cline{2- 5}
				&{$2\times 2$}&$27$&$3.04(00)$ &$-$ \\
				\cline{2- 5}
				&{$4\times 4$}&$125$ &$1.94(00)$ &$ 0.65$ \\
				\cline{2- 5}
				&{$8\times 8$} &$729$&$1.05(00)$ &$0.89$ \\
				\cline{2- 5}
				&{$16\times 16$} &$4913$&$ 5.23(-01)$ &$1.01$ \\
				\cline{2- 5}
				&{$32\times 32$} &$35937$&$ 2.56(-01)$ &$1.02$ \\
				\cline{2- 5}
				&{$64\times 64$}&$274625$ &$ 1.26(-01)$ &$ 1.02$ \\
				\cline{1- 5}
				\multirow{7}{*}{$hp$-DGFEM}&{Mesh}&{Dof}&{$\H^1$-error for $\eta = 1$}&{$O(h)$}\\
				\cline{2- 5}
				&{$2\times 2$}&$192$&$3.03(00)$ &$-$ \\
				\cline{2- 5}
				&{$4\times 4$} &$1536$&$1.91(00)$ &$0.67$ \\
				\cline{2- 5}
				&{$8\times 8$}&$12288$ &$1.02(00)$ &$0.90$ \\
				\cline{2- 5}
				&{$16\times 16$} &$98304$&$ 5.07(-01)$ &$1.01$ \\
				\cline{2- 5}
				&{$32\times 32$}&$786432$ &$2.50(-01)$ &$1.02$ \\
				\cline{1- 5}
		\end{tabular}}
	\end{center}
\end{table}
\begin{table}
	\begin{center}
		{\small
			\caption{Errors and convergence rates for the numerical solutions $u_h$ relative to the exact solution $ u = t\sqrt{t}\sin(2\pi x)\sin(2\pi y)\sin(2\pi z)$ with DG time stepping of $1$ degree  for fractional equation. }
			\label{table5.12}
			\begin{tabular}{| c | c | c | c | c | c | c |}
				\hline
				\multicolumn{7}{|c|}{Error history in 3D with weakly singular kernel.} \\
				\hline
				\multirow{6}{*}{$hp$-FEM} & Mesh & Dof & $\L^2$-error & $O(h^3)$ & $\H^1$-error & $O(h^2)$ \\
				\cline{2-7}
				& $2\times 2$ & $250$ & $1.52(-01)$ & $-$ & $1.93(00)$ & $-$ \\
				\cline{2-7}
				& $4\times 4$ & $1458$ & $4.53(-02)$ & $1.75$ & $7.00(-01)$ & $1.47$ \\
				\cline{2-7}
				& $8\times 8$ & $9826$ & $5.81(-03)$ & $2.96$ & $1.88(-01)$ & $1.89$ \\
				\cline{2-7}
				& $16\times 16$ & $71874$ & $7.21(-04)$ & $3.01$ & $4.77(-02)$ & $1.98$ \\
				\hline
				\multirow{5}{*}{$hp$-DGFEM} & Mesh & Dof & $\H^1$-error for $\eta = 1$ & $O(h^3)$ & $\H^1$-error & $O(h^2)$ \\
				\cline{2-7}
				& $2\times 2$ & $960$ & $1.51(-01)$ & $-$ & $1.93(00)$ & $-$ \\
				\cline{2-7}
				& $4\times 4$ & $7680$ & $4.49(-02)$ & $1.75$ & $6.95(-01)$ & $1.48$ \\
				\cline{2-7}
				& $8\times 8$ & $61440$ & $5.78(-03)$ & $2.96$ & $1.88(-01)$ & $1.89$ \\
				\cline{2-7}
				& $16\times 16$ & $491520$ & $7.11(-04)$ & $3.02$ & $4.73(-02)$ & $1.99$ \\
				\hline
			\end{tabular}
		}
	\end{center}
\end{table}
\subsection{Prey-predator system}

Consider the prey-predator system of partial differential equations with prey growth damped by the Allee effect, as described in \cite{MLi} and references therein given as:

\begin{align}
	\frac{\partial u}{\partial t} - \Delta u &= \gamma u (u - \beta)(1 - u) - \frac{uv}{1 + \alpha u}, \\
	\frac{\partial v}{\partial t} - \epsilon \Delta v &= \frac{uv}{1 + \alpha u}v - \delta v,
\end{align}

with the initial conditions:

$$
u_0(x,y) = \begin{cases} 
	p, \quad &\text{if } |x - \frac{L}{2}| \leq \Delta \text{ and } |y - \frac{L}{2}| \leq \Delta, \\
	0,  &\text{otherwise},
\end{cases}
\quad
v_0(x,y) = \begin{cases} 
	q, \quad &\text{if } |x - \frac{L}{2} - a| \leq \Delta \text{ and } |y - \frac{L}{2} - b| \leq \Delta, \\
	0, & \text{otherwise}.
\end{cases}
$$

Here, the dimensionless variables $ u $ represents the prey, and $ v $ represents the predator at time $ t $ and position $ (x,y) $. The constant $ \epsilon = \frac{D_2}{D_1} $ is the ratio of the diffusion coefficients, where $ D_1 $ and $ D_2 $ are the diffusion coefficients of the prey and predator, respectively. The parameters $ \alpha, \beta, \gamma, \delta $ are ecological constants, and $ L $ represents the length of the spatial domain.

For this problem, the time domain is $ [0,90] $, and the spatial domain is $ \Omega = [0,200]^2 $. The parameters are chosen as follows: $ \epsilon = 1 $, $ \alpha = 0.2 $, $ \beta = 0.1 $, $ \delta = 0.37 $, $ p = 1 $, $ q = 0.5 $, $ L = 200 $, $ a = 5 $, $ b = 30 $, and $ \Delta_{11} = \Delta_{12} = \Delta_{21} = \Delta_{22} = 20 $.

The spatial domain $ \Omega $ is discretized into a mesh of $ 64 \times 64 $ elements, where each square is further divided into a pair of triangles. The time domain is discretized with a zero-degree polynomial and a time step size of $ k_n = 0.1 $.

Figure \ref{5.f} illustrates the solution plots of the prey and predator densities, showing the existence of multiple attractors and the persistence of the species under three different cases of the memory coefficient $ \eta $. Specifically:

\begin{itemize}
	\item Figure \ref{5.f} (A) represents the case with no memory effect ($ \eta = 0 $).
	\item Figures \ref{5.f} (B) and \ref{5.f} (C) show the cases for $ \eta = 0.01 $ and $ \eta = 0.1 $, respectively.
\end{itemize}

The results indicate that the memory term significantly influences the dynamics of the system. In particular, the plots for $ \eta = 0.1 $ demonstrate that the memory effect prolongs the dynamics, suggesting that past interactions contribute to the current densities of both prey and predator.

\begin{figure}[ht!]
	\centering
	\caption{The solution plot for Prey(1st row) and predator (2nd row) density for different memory coefficient.}\label{5.f}
	\begin{subfigure}[b]{0.8\textwidth}
		\centering
		\includegraphics[width=0.48\textwidth]{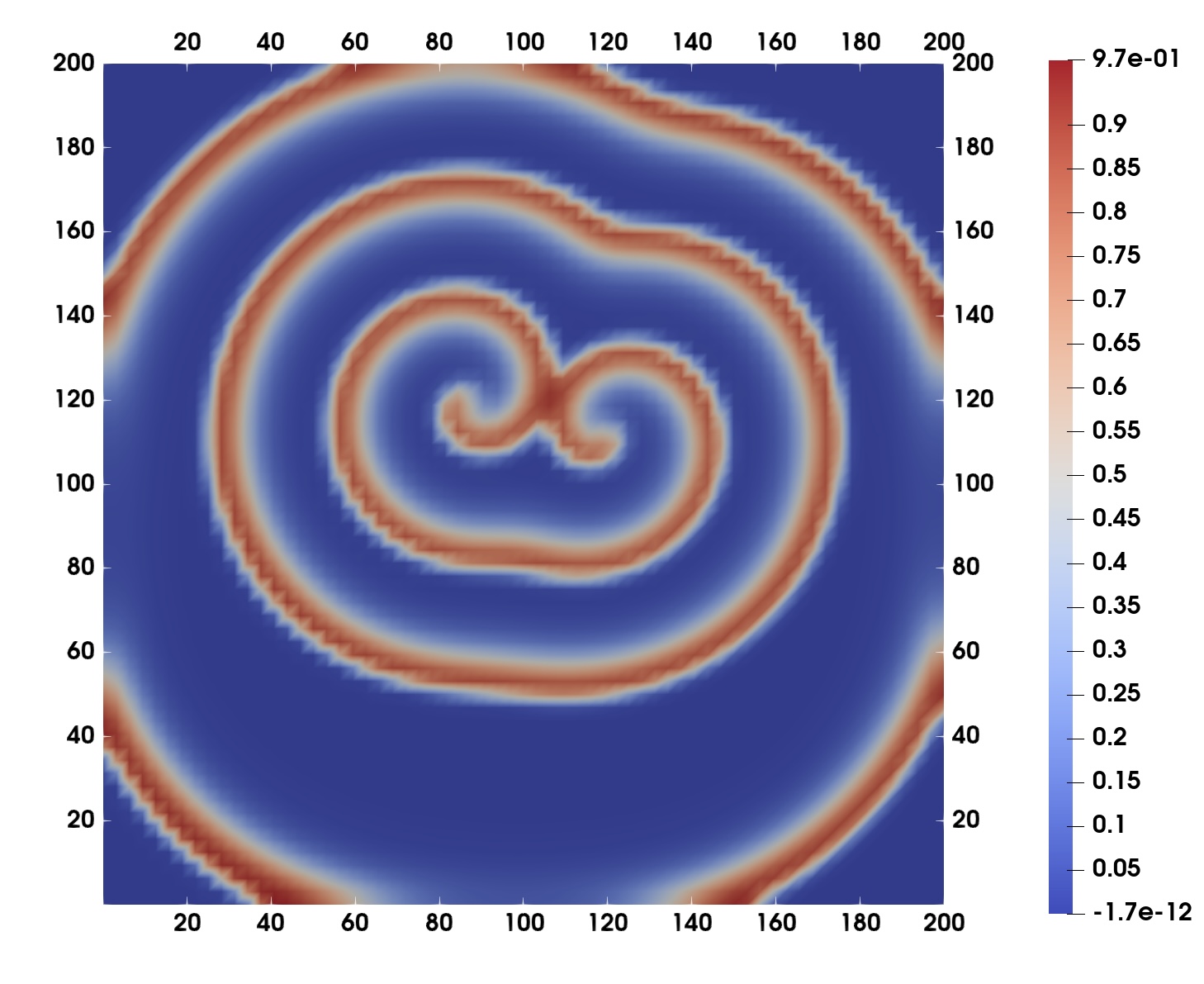}
		\includegraphics[width=0.48\textwidth]{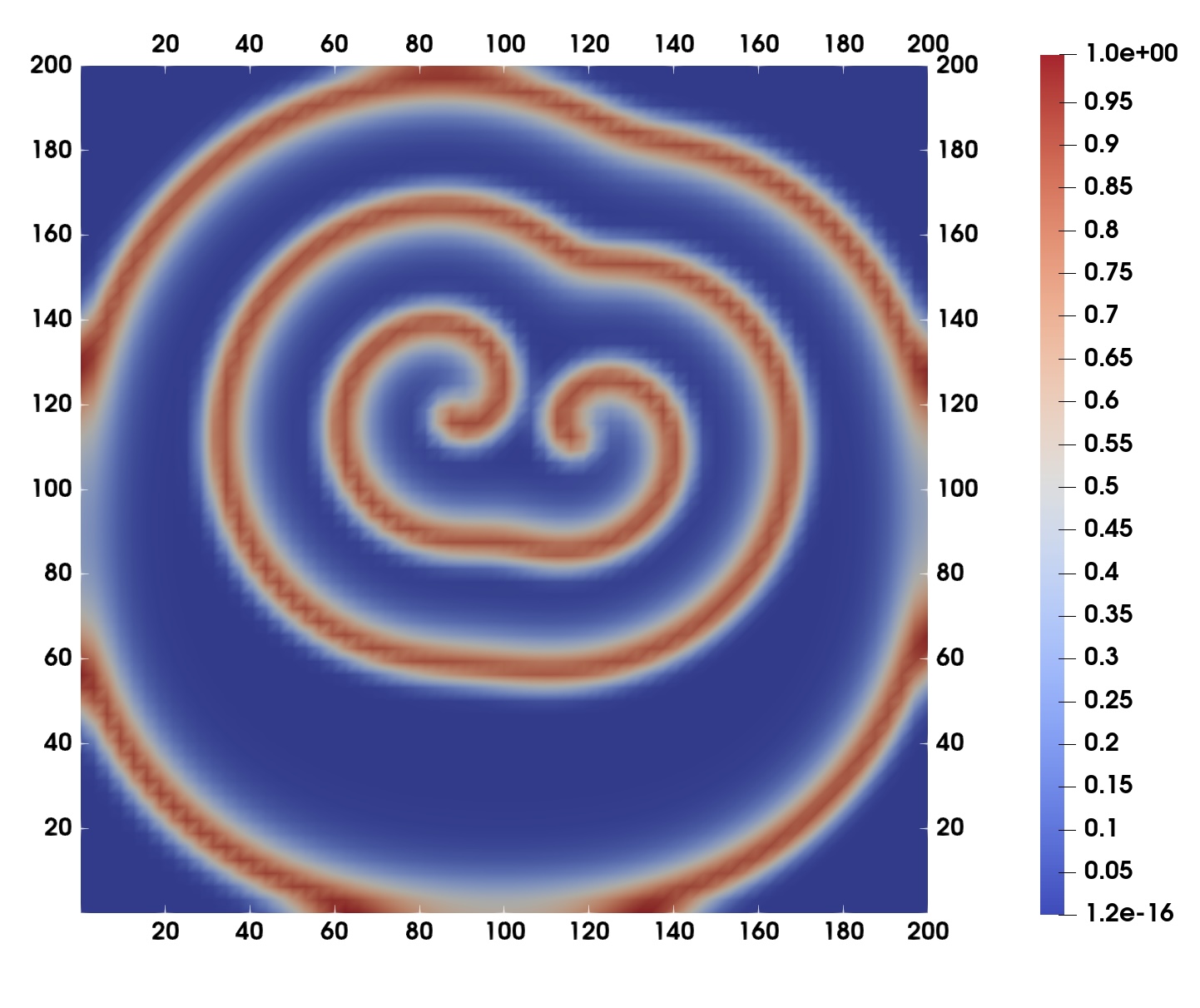}
		\caption{$\eta=0$}
	\end{subfigure}\\
	\begin{subfigure}[b]{0.8\textwidth}
		\centering
		\includegraphics[width=0.48\textwidth]{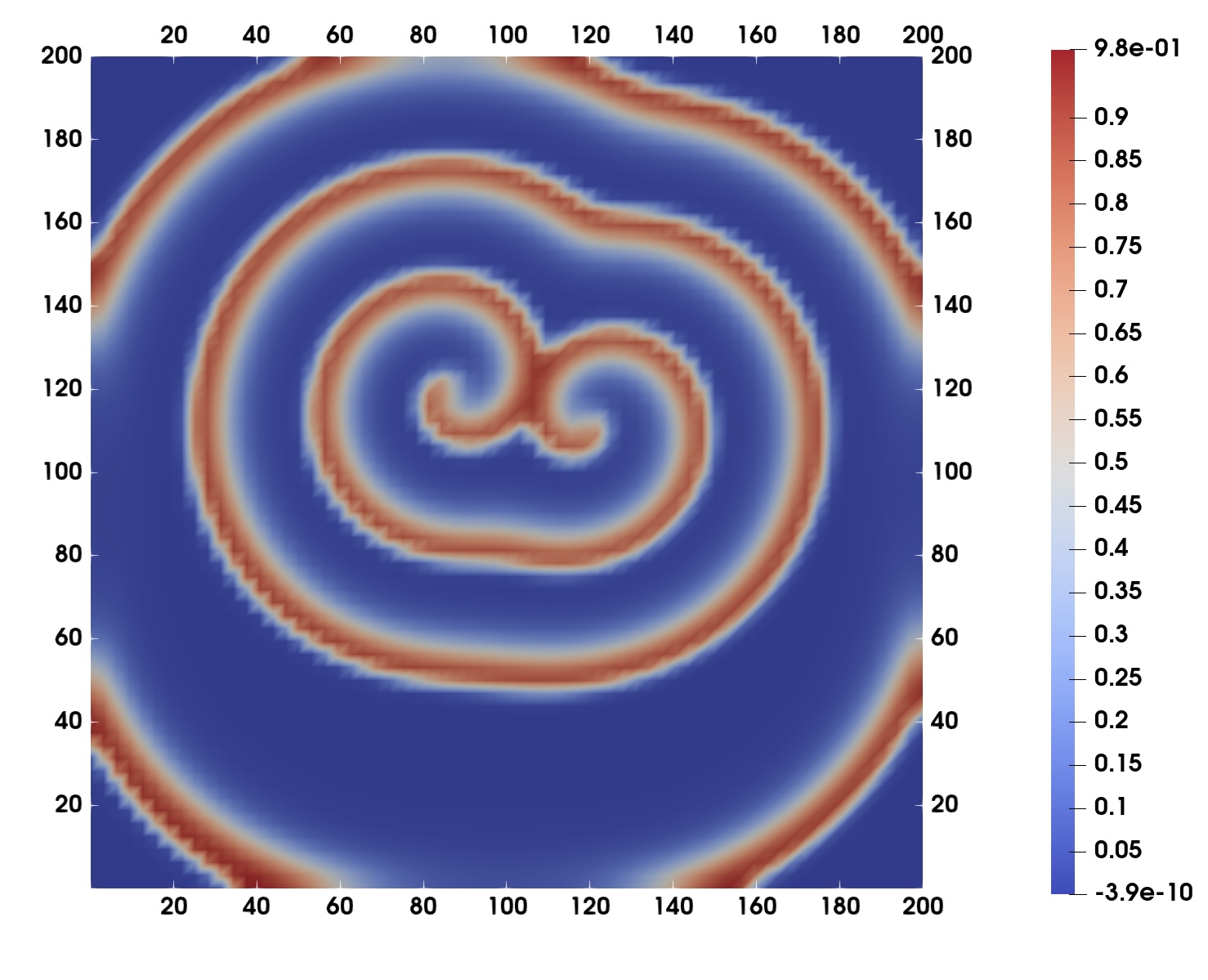}
		\includegraphics[width=0.48\textwidth]{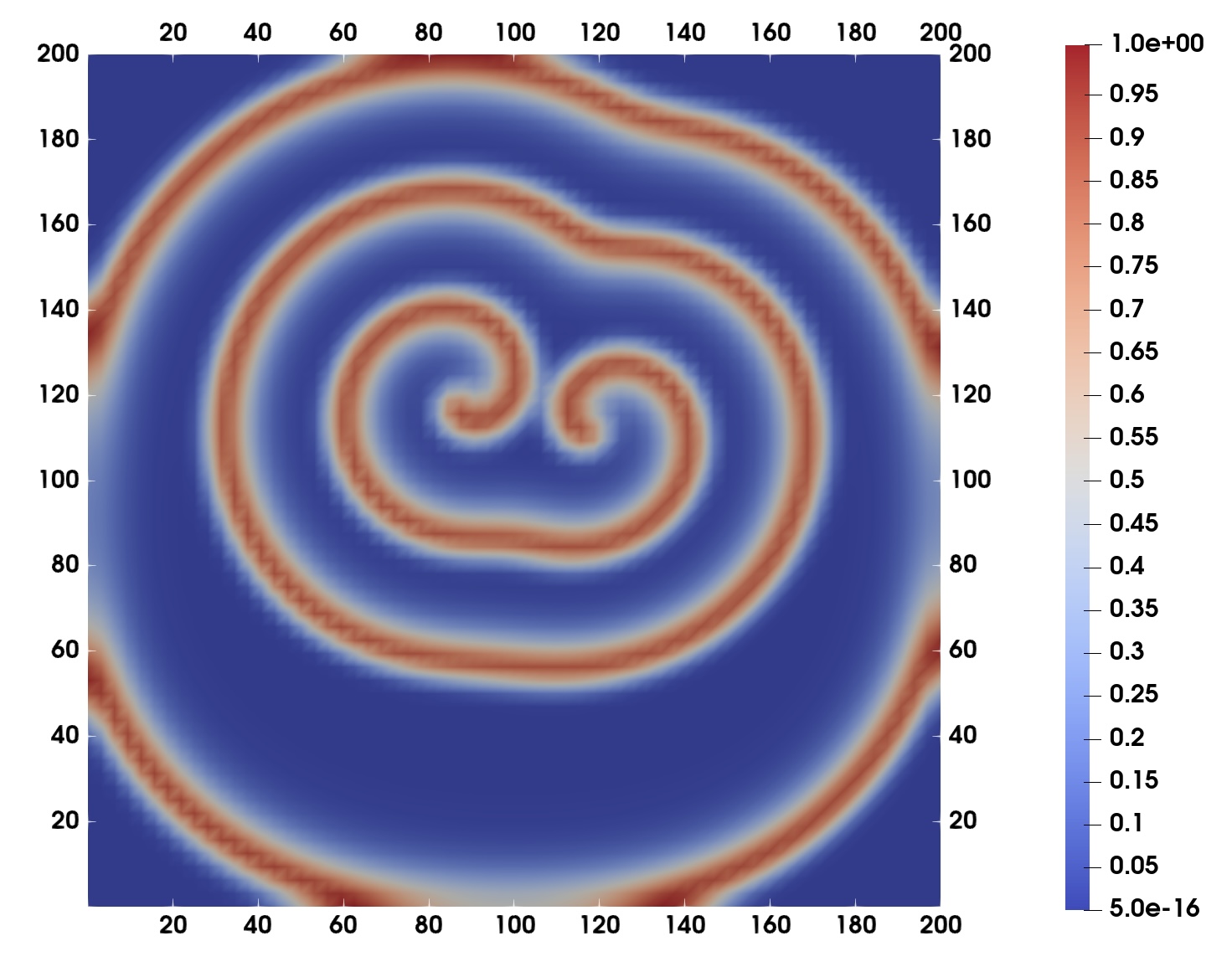}
		\caption{ $
			\eta=0.01$}
	\end{subfigure}\\
	\begin{subfigure}[b]{0.8\textwidth}
		\centering
		\includegraphics[width=0.48\textwidth]{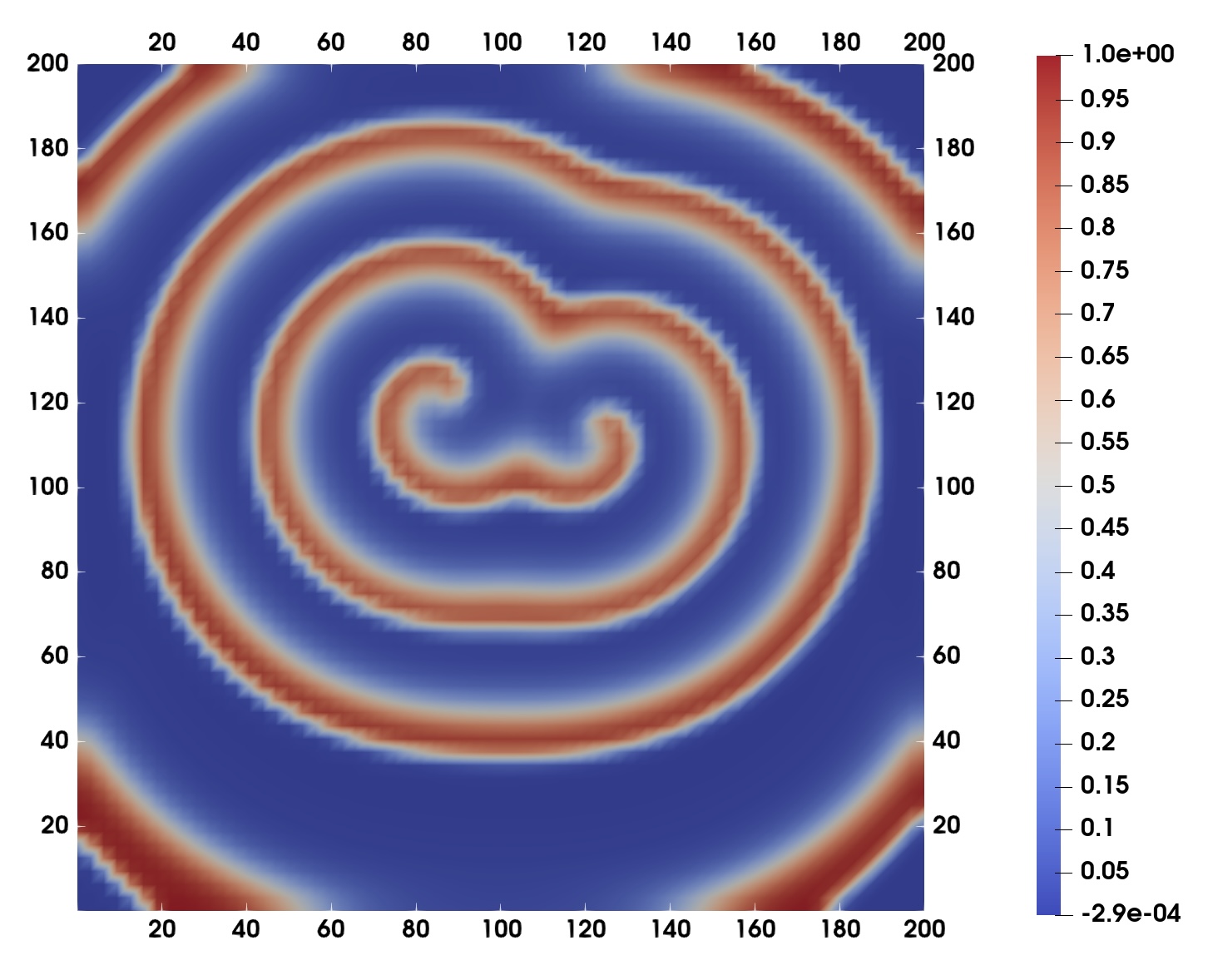}
		\includegraphics[width=0.48\textwidth]{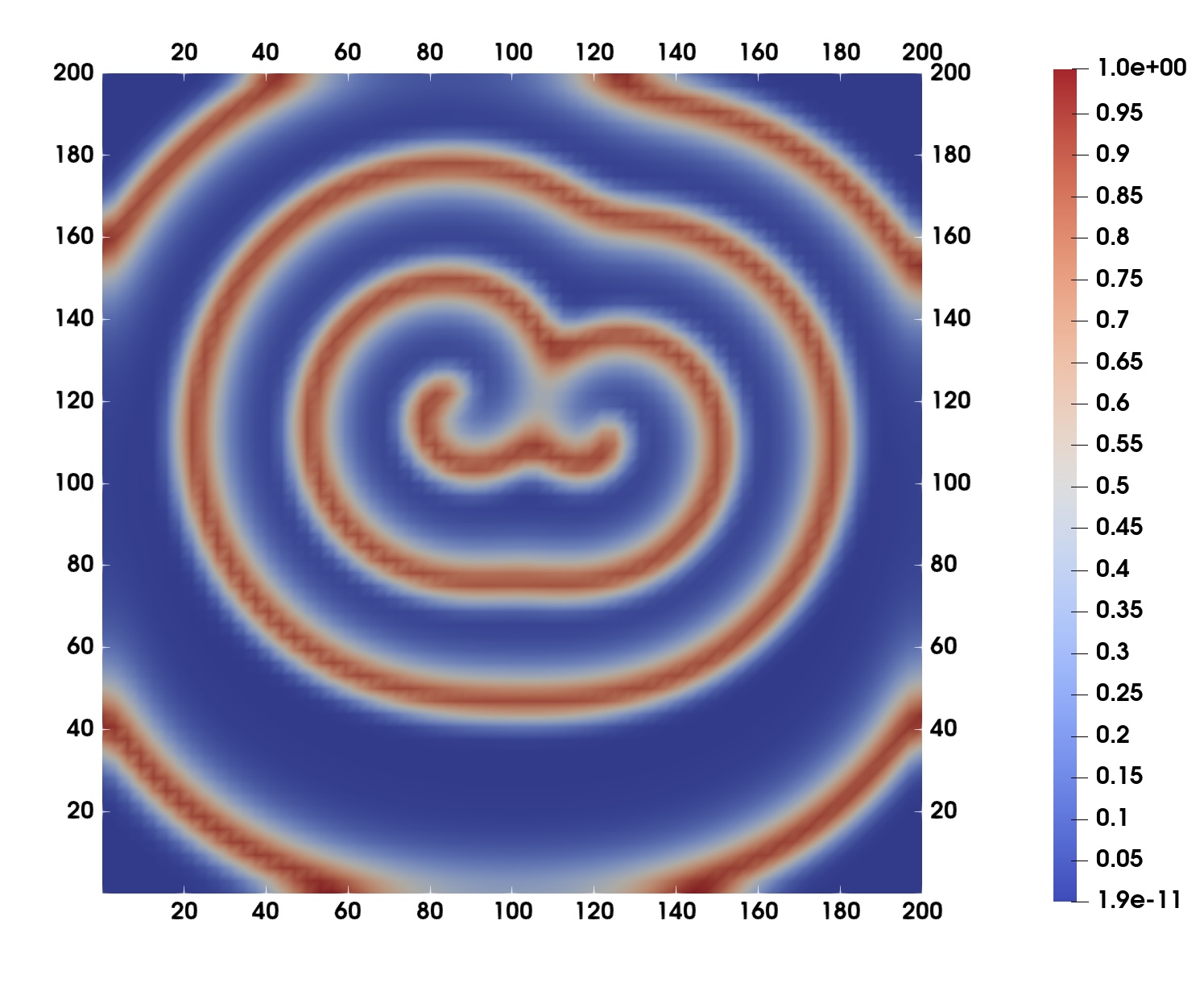}
		\caption{$\eta=0.1$}
	\end{subfigure}
\end{figure}

\section{Acknowledgement}
S. Mahajan gratefully acknowledges the financial support provided by the Ministry of Education, Government of India, through the Prime Minister Research Fellowship (PMRF ID: 2801816), which made this research possible. We also extend our sincere thanks to Dr. Manil T. Mohan, Associate Professor in the Department of Mathematics at IIT Roorkee, for his insightful discussions and for introducing us to this model. \\
\textbf{Data availability}
In case of reasonable request, datasets generated during the research discussed in the paper will be available from the corresponding author.\\
\textbf{Declarations}
The authors declare that they have no conflict of interest.

\bibliographystyle{abbrv}
\bibliography{ref}

\begin{thebibliography}{10}

\bibitem{ABJ}
M.~Aln{\ae}s, J.~Blechta, J.~Hake, A.~Johansson, B.~Kehlet, A.~Logg,
  C.~Richardson, J.~Ring, M.~E. Rognes, and G.~N. Wells.
\newblock The fenics project version 1.5.
\newblock {\em Archive of Numerical Software}, 3, 2015.

\bibitem{BSu}
I.~Babu\v{s}ka and M.~Suri.
\newblock The {$p$} and {$h$}-{$p$} versions of the finite element method,
  basic principles and properties.
\newblock {\em SIAM Rev.}, 36(4):578--632, 1994.

\bibitem{BSc1}
H.~Brunner and D.~Sch\"otzau.
\newblock {$hp$}-discontinuous {G}alerkin time-stepping for {V}olterra
  integrodifferential equations.
\newblock {\em SIAM J. Numer. Anal.}, 44(1):224--245, 2006.

\bibitem{CCH}
Y.~Chen, Z.~Chen, and Y.~Huang.
\newblock An {\it hp}-version of the discontinuous {G}alerkin method for
  fractional integro-differential equations with weakly singular kernels.
\newblock {\em BIT}, 64(3):Paper No. 27, 2024.

\bibitem{RDJL}
R.~Dautray and J.-L. Lions.
\newblock {\em Mathematical analysis and numerical methods for science and
  technology. {V}ol. 4}.
\newblock Springer-Verlag, Berlin, 1990.

\bibitem{EJo}
K.~Eriksson and C.~Johnson.
\newblock Error estimates and automatic time step control for nonlinear
  parabolic problems. {I}.
\newblock {\em SIAM J. Numer. Anal.}, 24(1):12--23, 1987.

\bibitem{KJC}
K.~Eriksson and C.~Johnson.
\newblock An adaptive finite element method for linear elliptic problems.
\newblock {\em Math. Comp.}, 50(182):361--383, 1988.

\bibitem{EJT}
K.~Eriksson, C.~Johnson, and V.~Thom\'{e}e.
\newblock Time discretization of parabolic problems by the discontinuous
  {G}alerkin method.
\newblock {\em RAIRO Mod\'{e}l. Math. Anal. Num\'{e}r.}, 19(4):611--643, 1985.

\bibitem{HSS}
P.~Houston, C.~Schwab, and E.~S\"uli.
\newblock Discontinuous {$hp$}-finite element methods for
  advection-diffusion-reaction problems.
\newblock {\em SIAM J. Numer. Anal.}, 39(6):2133--2163, 2002.

\bibitem{Jam}
P.~Jamet.
\newblock Galerkin-type approximations which are discontinuous in time for
  parabolic equations in a variable domain.
\newblock {\em SIAM J. Numer. Anal.}, 15(5):912--928, 1978.

\bibitem{KMR}
A.~Khan, M.~T. Mohan, and R.~Ruiz-Baier.
\newblock Conforming, nonconforming and {DG} methods for the stationary
  generalized {B}urgers'-{H}uxley equation.
\newblock {\em J. Sci. Comput.}, 88:1--26, 2021.

\bibitem{KST}
A.~A. Kilbas, H.~M. Srivastava, and J.~J. Trujillo.
\newblock {\em Theory and applications of fractional differential equations},
  volume 204.
\newblock Elsevier Science B.V., Amsterdam, 2006.

\bibitem{LHL}
H.~P. Langtangen and A.~Logg.
\newblock {\em Solving {PDE}s in {P}ython}.
\newblock Springer, Cham, 2016.

\bibitem{SVL}
S.~Larsson, V.~Thom\'{e}e, and L.~B. Wahlbin.
\newblock Numerical solution of parabolic integro-differential equations by the
  discontinuous {G}alerkin method.
\newblock {\em Math. Comp.}, 67(221):45--71, 1998.

\bibitem{LRA}
P.~Lasaint and P.-A. Raviart.
\newblock On a finite element method for solving the neutron transport
  equation.
\newblock In {\em Mathematical aspects of finite elements in partial
  differential equations ({P}roc. {S}ympos., {M}ath. {R}es. {C}enter, {U}niv.
  {W}isconsin, {M}adison, {W}is., 1974)}, Publication No. 33, pages 89--123.
  Math. Res. Center, Univ. of Wisconsin-Madison, Academic Press, New York,
  1974.

\bibitem{GBHE2}
S.~Mahajan and A.~Khan.
\newblock Finite element approximation for a delayed generalized
  {B}urgers'-{H}uxley equation with weakly singular kernels: Part {II}
  {N}on-{C}onforming and {DG} approximation.
\newblock {\em arXiv preprint arXiv:2310.07788v1}, 2023.

\bibitem{GBHE}
S.~Mahajan, A.~Khan, and M.~T. Mohan.
\newblock Finite element approximation for a delayed generalized
  {B}urgers'-{H}uxley equation with weakly singular kernels: Part {I}
  {W}ell-posedness, {R}egularity and {C}onforming approximation.
\newblock {\em arXiv preprint arXiv:2309.01636}, 2023.

\bibitem{MTM}
M.~T. Mohan.
\newblock On the three dimensional {K}elvin-{V}oigt fluids: global solvability,
  exponential stability and exact controllability of {G}alerkin approximations.
\newblock {\em Evol. Equ. Control Theory}, 9:301--339, 2020.

\bibitem{MKH}
M.~T. Mohan and A.~Khan.
\newblock On the generalized {B}urgers'-{H}uxley equation: {E}xistence,
  uniqueness, regularity, global attractors and numerical studies.
\newblock {\em Discrete Contin. Dyn. Syst. Ser. B}, 26:3943--3988, 2021.

\bibitem{MLi}
A.~Morozov and B.-L. Li.
\newblock On the importance of dimensionality of space in models of
  space-mediated population persistence.
\newblock {\em Theoretical population biology}, 71(3):278--289, 2007.

\bibitem{MBM}
K.~Mustapha, H.~Brunner, H.~Mustapha, and D.~Sch\"{o}tzau.
\newblock An {$hp$}-version discontinuous {G}alerkin method for
  integro-differential equations of parabolic type.
\newblock {\em SIAM J. Numer. Anal.}, 49(4):1369--1396, 2011.

\bibitem{PSc}
I.~Perugia and D.~Sch\"otzau.
\newblock An {$hp$}-analysis of the local discontinuous {G}alerkin method for
  diffusion problems.
\newblock In {\em Proceedings of the {F}ifth {I}nternational {C}onference on
  {S}pectral and {H}igh {O}rder {M}ethods ({ICOSAHOM}-01) ({U}ppsala)},
  volume~17, pages 561--571, 2002.

\bibitem{PSZ}
I.~Perugia, C.~Schwab, and M.~Zank.
\newblock Exponential convergence of {$hp$}-time-stepping in space-time
  discretizations of parabolic {PDE}s.
\newblock {\em ESAIM Math. Model. Numer. Anal.}, 57(1):29--67, 2023.

\bibitem{RHi}
W.~H. Reed and T.~R. Hill.
\newblock Triangular mesh methods for the neutron transport equation.
\newblock Technical report, Los Alamos Scientific Lab., N. Mex.(USA), 1973.

\bibitem{SSc1}
D.~Sch\"otzau and C.~Schwab.
\newblock An {$hp$} a priori error analysis of the {DG} time-stepping method
  for initial value problems.
\newblock {\em Calcolo}, 37(4):207--232, 2000.

\bibitem{SSc2}
D.~Sch\"otzau and C.~Schwab.
\newblock Time discretization of parabolic problems by the {$hp$}-version of
  the discontinuous {G}alerkin finite element method.
\newblock {\em SIAM J. Numer. Anal.}, 38(3):837--875, 2000.

\bibitem{CSc}
C.~Schwab.
\newblock {\em {$p$}- and {$hp$}-finite element methods}.
\newblock The Clarendon Press, Oxford University Press, New York, 1998.

\bibitem{SWW}
J.~Shi, C.~Wang, and H.~Wang.
\newblock Diffusive spatial movement with memory and maturation delays.
\newblock {\em Nonlinearity}, 32(9):3188--3208, 2019.

\bibitem{Vth}
V.~Thom{\'e}e.
\newblock {\em Galerkin finite element methods for parabolic problems}.
\newblock Springer Science \& Business Media, 2007.

\bibitem{Wih}
T.~P. Wihler.
\newblock An a priori error analysis of the {$hp$}-version of the continuous
  {G}alerkin {FEM} for nonlinear initial value problems.
\newblock {\em J. Sci. Comput.}, 25(3):523--549, 2005.

\bibitem{Yi}
L.~Yi.
\newblock An {$h$}-{$p$} version of the continuous {P}etrov-{G}alerkin finite
  element method for nonlinear {V}olterra integro-differential equations.
\newblock {\em J. Sci. Comput.}, 65(2):715--734, 2015.

\bibitem{YGU}
L.~Yi and B.~Guo.
\newblock An {$h$}-{$p$} version of the continuous {P}etrov-{G}alerkin finite
  element method for {V}olterra integro-differential equations with smooth and
  nonsmooth kernels.
\newblock {\em SIAM J. Numer. Anal.}, 53(6):2677--2704, 2015.

\end{thebibliography}
\include{ref}
\end{document}